\documentclass[a4paper,12pt,leqno,oneside]{amsart}
\usepackage{amssymb}
\usepackage{mathrsfs}
\usepackage{graphicx,color}
\usepackage{newtxtext}
\usepackage[varg]{newtxmath}
\usepackage[shortlabels]{enumitem}
\usepackage{geometry}%
\theoremstyle{plain} 
\newtheorem{theorem}{Theorem}[section]

\newtheorem{proposition}[theorem]{Proposition}
\newtheorem{lemma}[theorem]{Lemma}

\theoremstyle{definition} 
\newtheorem{definition}[theorem]{Definition}
\newtheorem{remark}[theorem]{Remark}

\DeclareMathOperator{\tr}{tr}
\DeclareMathOperator{\Div}{div}

\newcommand{\R}{\mathbb{R}}

\usepackage{constants}
\newconstantfamily{eps}{symbol=\varepsilon}
\numberwithin{equation}{section}
\newcommand{\Feq}{f^{\mathrm{eq}}}
\usepackage{lineno}

\usepackage{bm}
\renewcommand{\vec}[1]{\bm{#1}}
\usepackage{hyperref}

\usepackage{float}

\graphicspath{{./fig/}}

\newcommand{\commentB}[1]{}

\newcommand{\plotFig}[2]{\includegraphics[width=#2]{#1}}

\newcommand{\figref}[1]{Figure~\ref{#1}}
\newcommand{\secref}[1]{Section~\ref{#1}}
\newcommand{\rmref}[1]{Remark~\ref{#1}}
\newcommand{\appref}[1]{Appendix~\ref{#1}}

\newcommand{\MT}[1]{\mathbf{#1}}

\newcommand{\grad}[1]{\nabla #1}

\newcommand{\energy}{F}

\newcommand{\dom}{\Omega}

\newcommand{\Ntot}{N_{\text{tot}}}

\newcommand{\FP}{Fokker-Plank}

\newcommand{\BC}{boundary condition}
\newcommand{\IC}{initial condition}

\title{Nonlinear inhomogeneous Fokker-Planck models:  energetic-variational structures and long time behavior}

\author{Yekaterina Epshteyn}
\address[Yekaterina Epshteyn]%
{Department of Mathematics,
The University of Utah,
Salt Lake City, UT 84112, USA}
\email{epshteyn@math.utah.edu}

\author{Chang Liu}
\address[Chang Liu]%
{Department of Mathematics,
The University of Utah,
Salt Lake City, UT 84112, USA}
\email{liukamala@math.utah.edu}

\author{Chun Liu}
\address[Chun Liu]%
{Department of Applied Mathematics, Illinois Institute of Technology.
Chicago, IL 60616, USA}
\email{cliu124@iit.edu}

\author{Masashi Mizuno}
\address[Masashi Mizuno]%
{Department of Mathematics, College of Science
and Technology, Nihon University, Tokyo 101-8308 JAPAN}
\email{mizuno.masashi@nihon-u.ac.jp}

\allowdisplaybreaks[1]
\begin{document}

%
%


%
%

\begin{abstract}
	
	Inspired by the modeling of grain growth in polycrystalline materials, we consider a nonlinear \FP\ model, with inhomogeneous diffusion and with variable mobility parameters.
	We develop large time asymptotic analysis of such nonstandard models by reformulating and extending the classical entropy method, under the assumption of periodic \BC.
	In addition, illustrative numerical tests are presented to
        highlight the essential points of the current analytical
        results and to motivate future analysis.	
	
\end{abstract}


%
%
%


\maketitle

\section{Introduction}
\label{sec:0}

Fokker-Planck type models
are widely used as a robust tool to describe the macroscopic behavior of the systems
that involve various  fluctuations  \cite{MR987631,MR2053476,MR3932086,MR3019444,MR3485127,MR4196904,MR4218540}, among
many others. In our previous work we derived Fokker-Planck type
systems as a part of grain growth models in polycrystalline materials,
e.g. \cite{DK:gbphysrev,MR2772123,MR3729587,epshteyn2021stochastic}. 
In this paper, we focus on those inhomogeneous fluctuations which
play essential roles in the modeling of the observations of the physical experiments
of these complex processes.

Most technologically useful materials are polycrystalline
microstructures composed of a myriad of small monocrystalline grains
separated by grain boundaries. The energetics and dynamics of
the grain boundaries provide the multiscale
properties of such materials.
Classical models of Mullins and
Herring for the evolution of the grain
boundaries in polycrystalline materials are based on the motion by mean
curvature as the local evolution law \cite{doi:10.1007/978-3-642-59938-5_2,
doi:10.1063/1.1722511,doi:10.1063/1.1722742}. Over the years, this idea
has motivated extensive relevant mathematical analysis of the
motion by mean curvature, e.g. ~\cite{MR1100211,MR2024995,MR1100206,MR3155251}, and the
study of the curvature flow on networks
~\cite{MR1833000,MR3967812,MR2075985,MR3495423,MR3612327,MR0485012}.
Furthermore, almost all previous work required the assumption of the specific equilibrium force balance condition at the triple
junctions points (triple junctions are where three grain boundaries
meet), e.g.~ \cite{MR1240580,MR1833000}.

Grain growth can be viewed as a complex multiscale process involving
dynamics of grain boundaries, triple junctions and the dynamics of lattice misorientations (difference in the orientation between two neighboring grains
that share the grain boundary).  Recently, there are some studies that
consider interactions among grain boundaries and triple junctions, e.g.,
\cite{upmanyu1999triple,upmanyu2002molecular,barmak_grain_2013,zhang2017equation,zhang2018motion,doi:10.1126/science.abj3210}. 
 In \cite{MR4263432,MR4283537},  by employing   the
energetic-variational  approach, we have developed a new model for the evolution of the
2D grain-boundary network with finite mobility of the triple junctions
and with dynamic lattice misorientations.  Under the assumption of no curvature effect,  we established a
local well-posedness result, as well as large time asymptotic behavior
for the model. Our results included obtaining explicit energy decay
rate for the system in terms of mobility of the triple junction and
the misorientation parameter. Further, we conducted extensive numerical experiments for the 2D grain boundary network 
in order to further understand/illustrate the effect of relaxation time scales, e.g. of the curvature of grain
boundaries, mobility of triple junctions, and dynamics of misorientations on how the grain boundary 
system decays energy and coarsens with time \cite{MR4283537,Katya-Chun-Mzn4}. 
Some relevant experimental results of the grain growth in thin
films have also been presented and discussed in \cite{KatyMattpaper,Katya-Chun-Mzn4}.

Note, the mathematical analysis in  \cite{MR4263432, MR4283537} was done under assumption of no critical
events/no disappearance events, e.g., grain
disappearance, facet/grain boundary disappearance, facet interchange,
splitting of unstable junctions (however, numerical simulations were
performed with critical events). Therefore, we began to extend our
models to incorporate the effect of critical events and we proposed a
Fokker-Plank type 
approach \cite{epshteyn2021stochastic} (which is also a further
extension of the earlier work on a simplified 1D critical event model
in \cite{DK:BEEEKT,DK:gbphysrev,MR2772123,MR3729587}).  Moreover, in \cite{epshteyn2021stochastic}
we have established the long time asymptotics of the corresponding Fokker-Planck solutions, namely the
joint probability density function of misorientations and triple junctions, and closely related the
marginal probability density of misorientations. Moreover, for an equilibrium configuration of a
boundary network, we have derived explicit local algebraic relations, a generalized Herring Condition
formula, as well as novel relation that connects grain boundary energy density with the geometry of
the grain boundaries that share a triple junction.

Here we will consider a class of nonlinear Fokker-Planck
equations. As discussed above, such models appear as a part of our studies of non-isothermal
thermodynamics \cite{BA00323160,BA47390682,PhysRevE.94.062117} with applications to
macroscopic models for grain boundary dynamics in polycrystalline
materials \cite{epshteyn2021stochastic,epshteyn2022local}. Fokker-Plank equations can be viewed as
generalized diffusion models in the framework of the energetic-variational
approach \cite{MR3916774, MR1607500}. Such systems are determined by the kinematic transport of the
probability density function, the free energy functional and the dissipation
(entropy production), \cite{baierlein_1999, MR1839500}. The
conventional mathematical analysis of the Fokker-Planck models is
usually developed for the simplified cases only.  In particular, this is especially true
for the well-known entropy methods developed for the asymptotic analysis of such
equations, e.g. \cite{MR1842428, MR3497125,MR1812873,MR1639292}. The
classical entropy methods
rely on the specific algebraic structures of the system, and seem to
have limited applications. 

We will consider two nonstandard generalized Fokker-Planck models,
one with the inhomogeneous diffusion and constant mobility 
parameters,  and the other one with both inhomogeneous diffusion and
variable mobility parameters. Therefore, to develop large time
asymptotic analysis for such systems, we first reformulate
the conventional entropy method in terms of the velocity field of the
probability density function  (rather
than using entropy method directly in terms of the probability density function). This key idea allows us to extend the
entropy method to  Fokker-Planck models (including nonlinear models)
with variable coefficients under assumption of the periodic boundary
conditions. 
\par The paper is organized as follows. In Sections~\ref{sec:1}-\ref{sec:1a}, we formulate the nonlinear Fokker-Planck model with the inhomogeneous diffusion and
variable mobility parameters, introduce notations and review important
results for such model. In Section~\ref{sec:2}, we first illustrate
large time asymptotic analysis for the Fokker-Planck model via
the idea of the entropy method in terms of the velocity field of the
solution under the assumption of the constant diffusion and mobility
parameters (hence, the Fokker-Planck system becomes a linear
model). In Sections~\ref{sec:3}-\ref{sec:4}, we extend the analysis to
the Fokker-Planck model with the inhomogeneous diffusion and constant mobility 
parameters,  and  to the Fokker-Planck model with  the inhomogeneous diffusion and
variable mobility parameters, respectively. Some conclusions and
numerical tests to illustrate essential points of the analytical
results are
given in Section~\ref{sec:6}.

\subsection{Model formulation and notations}
\label{sec:1}
In this paper, we consider the following Fokker-Planck model
subject to the periodic boundary condition on a domain
$\Omega=[0,1)^n\subset\R^n$
\begin{equation}
 \label{eq:1.1}
 \left\{
  \begin{aligned}
   \frac{\partial f}{\partial t}
   -
   \Div
   \left(
   \frac{f}{\pi(x,t)}
   \nabla
   \left(
   D(x)\log f
   +
   \phi(x)
   \right)
   \right)
   &=
   0,
   &\quad
   &x\in\Omega,\quad
   t>0, \\
   f(x,0)&=f_0(x),&\quad
   &x\in\Omega.
  \end{aligned}
 \right.
\end{equation}
Here $D=D(x):\Omega\rightarrow\R$,
$\pi=\pi(x,t):\Omega\times[0,\infty)\rightarrow\R$ are given positive
periodic functions on $\Omega$ and $\phi=\phi(x):\Omega\rightarrow\R$ is a
given periodic function on $\Omega$. The periodic boundary condition for
$f$ means,
\begin{equation}
 \nabla^l f(x_{b,1},t)=\nabla^l f(x_{b,2},t),
\end{equation}
for $x_{b,1}=(x_1,x_2,\ldots,x_{k-1},1,x_{k+1},\ldots,x_n)$,
$x_{b,2}=(x_1,x_2,\ldots,x_{k-1},0,x_{k+1},\ldots,x_n)\in\partial\Omega$,
$t>0$ and $l=0,1,2,\ldots$. In other words, $f$ can be smoothly
extended to a
function on the entire space $\R^n$ with the condition
$f(x,t)=f(x+e_j,t)$ for $x\in\R^n$, $t>0$ and $j=1,2,\ldots,n$, where
$e_j=(0,\ldots,1,\ldots,0)$, with the 1 in the $j$th place. Note that,
the periodic boundary condition for the function $f(x,t)$ is equivalent
to the condition that $f(x,t)$ is the function on the $n$-dimensional
torus for $t>0$. The periodic function is defined in the same way. The meaning of the periodic boundary condition for the Fokker-Planck equation can be seen in \cite[\S 4.1]{MR3288096}.
%
%

Let us introduce $\vec{u}$,  a velocity vector, namely,
\begin{equation}
 \label{eq:1.7}
 \vec{u}
  =
  -
  \frac{1}{\pi(x,t)}
  \nabla
  \left(
   D(x)\log f
   +
   \phi(x)
  \right).
\end{equation}
Then,  the system \eqref{eq:1.1} becomes,
\begin{equation}
\label{velf}
 \left\{
  \begin{aligned}
   \frac{\partial f}{\partial t}
  +
  \Div
  (f\vec{u})  &=
   0,
   &\quad
   &x\in\Omega,\quad
   t>0, \\
   f(x,0)&=f_0(x),&\quad
   &x\in\Omega.
  \end{aligned}
 \right.
\end{equation}
The form of the first equation in \eqref{velf} will make it possible to extend
entropy methods to nonlinear Fokker-Planck model with inhomogeneous
temperature parameter $D(x)$. Next, using \eqref{velf} together with
integration by parts  and with the
periodic boundary condition,  it is easy to obtain that,
\begin{equation}
 \frac{d}{dt}
  \int_\Omega
  f\,dx
  =
  \int_\Omega
  \frac{\partial f}{\partial t}\,dx
  =
  \int_\Omega
  \Div (f\vec{u})\,dx
  =
  0.
\end{equation}
Therefore, if $f_0$ is a probability density function on $\Omega$, we have,
\begin{equation}
\label{eq:1.10}
 \int_\Omega
  f\,dx
  =
  \int_\Omega
  f_0\,dx
  =1.
\end{equation}

Let $F$ be a free energy defined by,
\begin{equation}
 \label{eq:1.2}
  F[f]
  :=
  \int_\Omega
  \left(
   D(x)f(\log f -1)+f\phi(x)
  \right)
  \,dx.
\end{equation}
Hence,  we can establish the energy law for \eqref{velf}.

\begin{proposition}
 Let $f$ be a solution of the periodic boundary value problem
 \eqref{velf},  $\vec{u}$
 be the velocity vector defined in \eqref{eq:1.7}, and let $F$ be a free
 energy defined in \eqref{eq:1.2}.  Then, for $t>0$,
 \begin{equation}
  \label{eq:1.3}
  \frac{dF}{dt}[f](t)
   =
   -\int_\Omega
   \pi(x,t)
   |\vec{u}|^2f\,dx.
 \end{equation}
\end{proposition}

\begin{proof}
 Take a time-derivative on the left-hand side of \eqref{eq:1.2}, then
 apply  integration by parts  and use the form \eqref{velf}
 together with the periodic boundary condition, one derives,
 \begin{equation}
  \label{eq:1.4}
   \begin{split}
    \frac{dF}{dt}[f]
    &=
    \int_\Omega
    \left(
    D(x)\log f+\phi(x)
    \right)
    f_t
    \,dx   \\
    &=
    -
    \int_\Omega
    \left(
    D(x)\log f+\phi(x)
    \right)
    \Div(f\vec{u})
    \,dx   \\
    &=
    \int_\Omega
    \nabla
    \left(
    D(x)\log f+\phi(x)
    \right)
    \cdot\vec{u}
    f
    \,dx.
   \end{split}
\end{equation}
 Recalling relation,  $-\pi(x,t)\vec{u}=   \nabla
 \left(
 D(x)\log f+\phi(x)
 \right)$, we have,
 \begin{equation*}
  \int_\Omega
   \nabla
   \left(
    D(x)\log f+\phi(x)
   \right)
   \cdot\vec{u}
   f
   \,dx
   =
   -
   \int_\Omega
   \pi(x,t)
   |\vec{u}|^2
   f
   \,dx,
 \end{equation*}
 thus,  we obtain \eqref{eq:1.3}.
\end{proof}

Hereafter, we define
the
right-hand side of~\eqref{eq:1.3} as $-D_{\mathrm{dis}}[f](t)$. One can observe from the energy law
\eqref{eq:1.3}, that an equilibrium state $\Feq$ for the model
\eqref{velf} satisfies $\vec{u}=0$.  Here, we derive the explicit
representation of the equilibrium solution for the Fokker-Planck model
\eqref{velf}.

\begin{proposition}
 The equilibrium state $\Feq$ for the system \eqref{velf}
 is given by,
 \begin{equation}
  \label{eq:1.5}
   \Feq(x)
   =
   \exp
   \left(
   -\frac{\phi(x)-\Cl{const:1.1}}{D(x)}\right),
 \end{equation}
 where $\Cr{const:1.1}$ is a constant, which satisfies,
 \begin{equation}
  \int_\Omega
   \exp
   \left(
   -\frac{\phi(x)-\Cr{const:1.1}}{D(x)}\right)
   \,dx
   =1.
 \end{equation}
\end{proposition}

\begin{proof}
 We have from the energy law \eqref{eq:1.3} that,
 \begin{equation*}
  0
   =
  \frac{dF}{dt}[\Feq]
   =
   -\frac{1}{\pi(x, t)}\int_\Omega |\nabla(D(x)\log \Feq+\phi(x))|^2 f\,dx,
 \end{equation*}
 hence,  $\nabla(D(x)\log \Feq+\phi(x))=0$. Thus, there is a constant
 $\Cr{const:1.1}$ such that,
 \begin{equation}
  \label{eq:1.6}
  D(x)\log \Feq+\phi(x)=\Cr{const:1.1},
 \end{equation}
 and, hence, we obtain \eqref{eq:1.5}.
\end{proof}

Now, let us define the scaled function $\rho$ by taking the ratio of $f$
and $\Feq$ \eqref{eq:1.5},
\begin{equation}
 \label{eq:1.8}
 \rho
  =
  \frac{f}{\Feq},\quad
  \text{or}\quad
  f(x,t)
  =
  \rho(x,t) \Feq(x,t)
  =
  \rho(x,t)
  \exp
  \left(
   -
   \frac{\phi(x)-\Cr{const:1.1}}{D(x)}
  \right).
\end{equation}
Using the relation \eqref{eq:1.6}, we have,
\begin{equation*}
 D(x)\log f+\phi(x)
  =
  D(x)\log\rho+D(x)\log\Feq+\phi(x)
  =
  D(x)\log\rho+\Cr{const:1.1},
\end{equation*}
hence,  the velocity $\vec{u}$ becomes,
\begin{equation}
 \vec{u}
  =
  -
  \frac{1}{\pi(x,t)}
  \nabla
  \left(
   D(x)\log \rho
  \right).
\end{equation}

In this paper, we show exponential convergence to the equilibrium
state via  energy
law,
\begin{itemize}
 \item in case of the homogeneous $D$ and the constant mobility $\pi$ in
       Section \ref{sec:2};
 \item in case of the inhomogeneous $D=D(x)$ and the constant mobility
       $\pi$ in Section \ref{sec:3};
 \item in case of the inhomogeneous $D=D(x)$ and the variable mobility
       $\pi=\pi(x,t)$ in Section \ref{sec:4}.
\end{itemize}
For the homogeneous $D$ and the constant mobility $\pi$, we can
reformulate classical entropy dissipation methods and show the exponential
decay of the global solution of \eqref{eq:1.1} in the $L^1$ space,
provided the logarithmic Sobolev inequality. In Appendix \ref{sec:5}, we
reformulate the entropy dissipation method in terms of the velocity
$\vec{u}$.
\begin{remark}
Finally we note that, when the coefficients and the solution $f$ are sufficiently smooth
functions, the classical approach to study model \eqref{eq:1.1} is to
rewrite it in the non-divergence form,
\begin{equation}
\label{nondiv}
 \frac{\partial f}{\partial t}
 -
 Lf
 +N(f)
 =0,
\end{equation}
where,
\begin{equation}
\label{nondivL}
 Lf
  =
  \frac{D(x)}{\pi(x,t)}\Delta f
  +
  \left(
   \nabla\left(
	  \frac{D(x)}{\pi(x,t)}
	 \right)
   +
   \frac{1}{\pi(x,t)}\nabla D(x)
   +
   \frac{1}{\pi(x,t)}\nabla\phi(x)
  \right)
  \cdot\nabla f
  +
  \left(
   \frac{\Delta\phi(x)}{\pi(x,t)}
   -
   \frac{\nabla \pi(x,t)\cdot\nabla \phi(x)}{\pi^2(x,t)}
  \right)
  f
\end{equation}
is a linear part and,
\begin{equation}
 \label{eq:1.9}
 N(f)
  =
 -
 \frac{1}{\pi(x,t)}\log f\nabla D(x)\cdot\nabla f
 +
 \frac{\nabla \pi(x,t)\cdot\nabla D(x)}{\pi^2(x,t)}f\log f
 -
 \frac{1}{\pi(x,t)}\Delta D(x) f\log f
\end{equation}
is a nonlinear part of \eqref{eq:1.1}. However, due to specific form
of the nonlinearity in \eqref{eq:1.9}, the existing entropy methods
will fail if one applies them to the non-divergence form
\eqref{nondiv}-\eqref{eq:1.9} instead.
\end{remark}
In this paper, we are studying the asymptotic behavior for the classical
solutions for these nonlinear Fokker-Planck systems. The solutions we
define below will be smooth enough so that all the derivatives and
integrations evolved in the equations and the estimates will make sense
in the usual sense (see \cite{epshteyn2022local, MR0241822, MR1465184}).

\begin{definition}
 \label{def:1.1}
 A periodic function in space $f=f(x,t)$ is a classical solution of the
 problem \eqref{eq:1.1} in $\Omega\times[0,T)$, subject to the periodic
 boundary condition,  if $f\in C^{2,1}(\Omega\times(0,T))\cap
 C^{1,0}(\overline{\Omega}\times[0,T))$, $f(x,t)>0$ for $(x,t)\in
 \Omega\times[0,T)$, and satisfies equation \eqref{eq:1.1} in a
 classical sense.
\end{definition}

In the next subsection, we show local existence of a classical solution and the
 maximum principle for \eqref{eq:1.1} subject to the periodic boundary
 condition.

\subsection{Local existence and a priori estimates}
\label{sec:1a}
Here we briefly explain local existence of a solution of \eqref{eq:1.1}.
To state the result, we give assumptions for the coefficients and the
initial data.

First, we assume the strong positivity for the coefficients $\pi(x,t)$
and $D(x)$, namely, there are constants
$\Cl{const:1.2},\Cl{const:1.3}>0$ such that for $x\in\Omega$ and $t>0$,
\begin{equation}
 \label{eq:1.12}
  \pi(x,t)\geq \Cr{const:1.2},\quad
  D(x)\geq \Cr{const:1.3}.
\end{equation}
Next, we assume the H\"older regularity for $0<\alpha<1$: coefficients
$\pi(x,t)$, $D(x)$, $\phi(x)$ and initial datum $f_0$ satisfy,
\begin{equation}
 \label{eq:1.13}
 \pi\in C_{\mathrm{per}}^{1+\alpha, (1+\alpha)/2}(\Omega\times[0,T)),\quad
  D\in C_{\mathrm{per}}^{2+\alpha}(\Omega),\quad
  \phi\in C_{\mathrm{per}}^{2+\alpha}(\Omega),\quad
  f_0\in C_{\mathrm{per}}^{2+\alpha}(\Omega),
\end{equation}
where
\begin{equation*}
 \begin{split}
  C_{\mathrm{per}}^{2+\alpha}(\Omega)
  &:=
  \{g\in C^{2+\alpha}(\Omega):
  g\ \text{is a periodic function on }\Omega
  \}, \\
  C_{\mathrm{per}}^{1+\alpha, (1+\alpha)/2}(\Omega\times[0,T))
  &:=
  \{g\in C^{1+\alpha, (1+\alpha)/2}(\Omega\times[0,T)):
  g(\cdot,t)\ \text{is a periodic function on }\Omega\
  \text{for}\ t>0
  \}. \\
 \end{split}
\end{equation*}
To state the following existence theorem, we also use a
periodic function space,
\begin{equation*}
 C_{\mathrm{per}}^{2+\alpha, 1+\alpha/2}(\Omega\times[0,T))
  :=
  \{g\in C^{2+\alpha, 1+\alpha/2}(\Omega\times[0,T)):
  g(\cdot,t)\ \text{is a periodic function on }\Omega\
  \text{for}\ t>0
  \}.
\end{equation*}

The next proposition guarantees the existence of a
local classical solution as defined in the Definition \ref{def:1.1} for \eqref{eq:1.1}, subject to the periodic
boundary condition.

\begin{proposition}
 \label{prop:1.1}
 Let the coefficients $\pi(x,t)$, $\phi(x)$, $D(x)$, and a positive
 probability density function $f_0(x)$ satisfy the strong positivity
 \eqref{eq:1.12} and the H\"older regularity \eqref{eq:1.13} for
 $0<\alpha<1$. Then,  there exists a time interval $T>0$ and a classical
 solution $f=f(x,t)$ of \eqref{eq:1.1} on $\Omega\times[0,T)$ with the
 Hölder regularity $f\in C_{\mathrm{per}}^{2+\alpha,
 1+\alpha/2}(\Omega\times[0,T))$.
\end{proposition}

Here we briefly sketch the proof of Proposition \ref{prop:1.1}. First,
we introduce the change of variable
$\rho(x,t)=f(x,t)/f^{\mathrm{eq}}(x)$. Note that from \eqref{eq:1.6}, $\nabla (D(x)f^{\mathrm{eq}}(x)+\phi(x))=\vec{0}$ hence the
equation \eqref{eq:1.1},  becomes,
\begin{equation}
 \label{eq:1.11}
 f^{\mathrm{eq}}(x)\rho_t
  =
  \Div
  \left(
   \frac{f^{\mathrm{eq}}(x)}{\pi(x,t)}\rho
   \nabla
   \left(
    D(x)
   \log\rho\right)
  \right).
\end{equation}

In order to explore the underlying structure of the equation, we
introduce new auxiliary variable,  $h(x,t)=D(x)\log\rho(x,t)$. By direct
calculation, we have,
\begin{equation*}
 \rho_t=\frac{\rho}{D(x)}h_t,\quad
  \nabla \rho
  =
  \rho\nabla
  \left(
   \frac{h}{D(x)}\right).
\end{equation*}
Hence, equation \eqref{eq:1.11} becomes,
\begin{equation}
 \label{eq:1.14}
 h_t
  =
  \frac{D(x)}{\pi(x,t)}\Delta h
  +
  \frac{D(x)}{f^{\mathrm{eq}}(x)}
  \nabla
  \left(
   \frac{f^{\mathrm{eq}}(x)}{\pi(x,t)}
  \right)
   \cdot
   \nabla h
   +
   \frac{D(x)}{\pi(x,t)}
  \nabla
  \left(
   \frac{h}{D(x)}
  \right)
   \cdot
   \nabla h.
\end{equation}
To proceed, we further introduce another auxiliary variable $\xi(x,t)$
as $h(x,t)=h(x,0)+\xi(x,t)$. Using the Schauder estimates for a
linearized problem, we can make a contraction mapping on the closed set of
$C_{\mathrm{per}}^{2+\alpha, 1+\alpha/2}(\Omega\times[0,T))$. Detailed
argument for the proof of Proposition \ref{prop:1.1} (under the natural
boundary condition), see
\cite{epshteyn2022local}.

Since we consider the periodic boundary condition, we can furthermore
show the following maximum principle, which gives the boundedness of the
classical solutions $f$ for the equation \eqref{eq:1.1}.

\begin{proposition}
 \label{prop:1.5} Let the coefficients $\pi(x,t)$, $\phi(x)$, $D(x)$,
 and a positive probability density function $f_0(x)$ satisfy the strong
 positivity \eqref{eq:1.12} and the H\"older regularity \eqref{eq:1.13}
 for $0<\alpha<1$. Assume $\pi(x,t)$, $\phi(x)$, $D(x)$, and $f_0(x)$
 are bounded functions, and there is a positive constant
 $\Cl{const:1.4}>0$, such that $f_0(x)\geq \Cr{const:1.4}$ for
 $x\in\Omega$. Let $f$ be a classical solution of \eqref{eq:1.1}. Then,
 \begin{equation}
  \label{eq:1.17}
 \exp\left(
       \frac{1}{D(x)}
       \min_{y\in\Omega}
       \left(
	D(y)
	\log \frac{f_0(y)}{f^{\mathrm{eq}}(y)}
       \right)
      \right)
  f^{\mathrm{eq}}(x)
  \leq
  f(x,t)
  \leq
  \exp\left(
       \frac{1}{D(x)}
       \max_{y\in\Omega}
       \left(
	D(y)
	\log \frac{f_0(y)}{f^{\mathrm{eq}}(y)}
       \right)
      \right)
  f^{\mathrm{eq}}(x),
 \end{equation}
 for $x\in\Omega$, $t>0$. In particular, there are positive constants
 $\Cl{const:1.5}, \Cl{const:1.6}>0$ such that,
 \begin{equation}
  \label{eq:1.18}
  \Cr{const:1.5}\leq f(x.t)\leq \Cr{const:1.6}
 \end{equation}
 for $x\in\Omega$, $t>0$.
\end{proposition}

\begin{proof}
 First, note that using the auxiliary variable
 \begin{equation}
  \rho(x,t)=\frac{f(x,t)}{f^{\mathrm{eq}}(x)},\qquad
   h(x,t)=D(x)\log\rho(x,t),
 \end{equation}
 the function $h$ is a classical solution of \eqref{eq:1.14}, the right
 hand side of which only includes the Laplacian and the gradient terms
 of $h$.

 We show that for $x\in\Omega$ and $t>0$,
 \begin{equation}
  \label{eq:1.15}
  \min_{y\in\Omega}h(y,0)
   \leq
   h(x,t)
   \leq
   \max_{y\in\Omega}h(y,0).
 \end{equation}
 The idea of the proof follows the argument of the proof of the maximum
 principle for the linear parabolic equation (see for instance,
 \cite[Theorem 8 in \S7.1]{MR2597943}, \cite[Section 1 in Chapter
 3]{MR0762825}). Here we give a complete proof.

 Write $h^\varepsilon(x,t)=h(x,t)-\varepsilon t$ for
 $\varepsilon>0$. Then,  $h_t^\varepsilon=h_t-\varepsilon$, $\nabla
 h^\varepsilon=\nabla h$, and $\Delta h^\varepsilon=\Delta h$,  hence by
 \eqref{eq:1.14}, we have,
 \begin{equation}
  \label{eq:1.19}
  h_t^\varepsilon+\varepsilon
   =
  \frac{D(x)}{\pi(x,t)}\Delta h^\varepsilon
  +
  \frac{D(x)}{f^{\mathrm{eq}}(x)}
  \nabla
  \left(
   \frac{f^{\mathrm{eq}}(x)}{\pi(x,t)}
  \right)
   \cdot
   \nabla h^\varepsilon
   +
   \frac{D(x)}{\pi(x,t)}
  \nabla
  \left(
   \frac{h^\varepsilon+\varepsilon t}{D(x)}
  \right)
   \cdot
   \nabla h^\varepsilon.
 \end{equation}
 Let $(x_0,t_0)\in\Omega\times(0,\infty)$ be a point,  such that
 $\max_{(x,t)\in\Omega\times[0,\infty)}h^\varepsilon(x,t)=h^\varepsilon(x_0,t_0)$. Note
 that $t_0>0$, hence at $(x_0,t_0)$, $h_t^\varepsilon\geq0$, $\nabla
 h^\varepsilon=0$, and $\Delta h^\varepsilon\leq0$. Thus, using
 \eqref{eq:1.19} and positivity of $D(x)$ and $p(x,t)$, at $(x_0,t_0)$,
\begin{equation*}
 0< h_t^\varepsilon+\varepsilon
  =
  \frac{D(x)}{\pi(x,t)}\Delta h^\varepsilon
  \leq
  0,
\end{equation*}
 which is a contradiction. Therefore $t_0=0$ and
 $\max_{(x,t)\in\Omega\times[0,\infty)}h^\varepsilon(x,t)=\max_{y\in\Omega}
 h^\varepsilon(y,0)$. By the definition of $h^\varepsilon$, we have
 $\max_{(x,t)\in\Omega\times[0,\infty)}h^\varepsilon(x,t)=\max_{y\in\Omega}
 h(y,0)$.  Let $\varepsilon\rightarrow0$ to find,
 \begin{equation*}
  \max_{(x,t)\in\Omega\times[0,\infty)}h(x,t)=\max_{y\in\Omega} h(y,0).
 \end{equation*}
 Hence,  we obtain $h(x,t)\leq \max_{y\in\Omega} h(y,0)$ for $x\in\Omega$
 and $t>0$. Proof of $h(x,t)\geq \min_{y\in\Omega} h(y,0)$ follows similar
 idea.

 Since,
 \begin{equation}
  \label{eq:1.16}
   h(x,t)
   =
   D(x)\log\rho(x,t)
   =
   D(x)\log\frac{f(x,t)}{{f^{\mathrm{eq}}(x)}},\qquad
   h(y,0)
   =
   D(y)\log\frac{f_0(y)}{{f^{\mathrm{eq}}(y)}},
 \end{equation}
 we obtain \eqref{eq:1.17} from \eqref{eq:1.15} and \eqref{eq:1.16}.

 Since $D(x)$, $\pi(x,t)$, $f_0(x)$ satisfy the strong positivity and
 $D(x)$, $\pi(x,t)$, $\phi(x)$, $f_0(x)$ are bounded, the right-hand
 side and the left-hand side of \eqref{eq:1.17} are bounded below and
 above by positive numbers $\Cr{const:1.5},\ \Cr{const:1.6}>0$
 respectively, thus, the result \eqref{eq:1.18} is deduced.
\end{proof}
%

Throughout this paper, we assume that coefficients $\pi(x,t)$,
$\phi(x)$, $D(x)$ and a positive probability density function $f_0(x)$
satisfy the strong positivity \eqref{eq:1.12} and the H\"older
regularity \eqref{eq:1.13} for $0<\alpha<1$. Further, we assume that
there is a positive constant $\Cr{const:1.4}>0$ such that $f_0(x) \geq
\Cr{const:1.4}$ for $x\in\Omega$. Then, by the
following proposition, we obtain the uniform lower bounds of $F[f](t)$
for $t>0$.

\begin{proposition}
 \label{prop:1.7}
 %
 Let $f$ be a solution of \eqref{eq:1.1}. Then,  there is a positive
 constant $\Cl{const:1.7}>0$,  such that,
 \begin{equation}
  \label{eq:1.20}
  F[f](t)\geq -\Cr{const:1.7},
 \end{equation}
 for $t>0$.
\end{proposition}

\begin{proof}
 By the triangle inequality, we have,
 \begin{equation}
  \begin{split}
   \int_\Omega D(x)f(\log f-1)\,dx
   &\geq
   -
   \int_\Omega D(x)f|\log f-1|\,dx \\
   &\geq
   -
   \|D\|_{L^\infty(\Omega)}\|\log f-1\|_{L^\infty(\Omega\times[0,\infty))}
   \int_\Omega f\,dx.
  \end{split} 
 \end{equation}
 Using $f\phi(x)\geq -f\|\phi\|_{L^\infty(\Omega)}$,  and \eqref{eq:1.10},
 we obtain,
 \begin{equation}
  F[f](t)
   \geq
   -
   \left(
    \|D\|_{L^\infty(\Omega)}\|\log f-1\|_{L^\infty(\Omega\times[0,\infty))}
    +
    \|\phi\|_{L^{\infty}(\Omega)}
   \right).
 \end{equation}
 By the maximum principle \eqref{eq:1.18}, $\log f-1$ is uniformly
 bounded on $\Omega\times[0,\infty)$. Therefore,  \eqref{eq:1.20} is
 deduced by choosing $\Cr{const:1.7}=\|D\|_{L^\infty(\Omega)}\|\log
 f-1\|_{L^\infty(\Omega\times[0,\infty))} +
 \|\phi\|_{L^{\infty}(\Omega)}$.
\end{proof}

\subsection{Notation}
Here we define some useful notations in this paper. For a vector field,
$\vec{u}=(u^k)_{k}$, we write,
\begin{equation*}
\nabla\vec{u}
 =
 (u^k_{x_l})_{k,l}
 =
 \begin{pmatrix}
  \frac{\partial u^1}{\partial x_1} &
  \frac{\partial u^1}{\partial x_2} &
  \cdots &
  \frac{\partial u^1}{\partial x_n} \\
  \frac{\partial u^2}{\partial x_1} &
  \frac{\partial u^2}{\partial x_2} &
  \cdots &
  \frac{\partial u^2}{\partial x_n} \\
  \vdots &
  \vdots &
  \ddots &
  \vdots \\
  \frac{\partial u^n}{\partial x_1} &
  \frac{\partial u^n}{\partial x_2} &
  \cdots &
  \frac{\partial u^n}{\partial x_n}
 \end{pmatrix}
 ,
\end{equation*}
and the transpose of $\nabla\vec{u}$ is
denoted by $\mathstrut^T\nabla\vec{u}=(u^l_{x_k})_{k,l}$. We denote the Frobenius norm of $\nabla \vec{u}$ by $|\nabla \vec{u}|$, namely $|\nabla \vec{u}|=\tr(\mathstrut^T\nabla\vec{u}\nabla \vec{u})$.
For the two vectors $\vec{u}=(u^k)_k$ and $\vec{v}=(v^l)_l$, we write,
\begin{equation}
 \vec{u}\otimes \vec{v}=(u^kv^l)_{k,l}
  =
  \begin{pmatrix}
   u^1 v^1&
   u^1v^2&
   \cdots &
   u^1v^n \\
   u^2v^1 &
   u^2v^2 &
   \cdots &
   u^2v^n \\
  \vdots &
   \vdots &
   \ddots &
   \vdots \\
   u^nv^1 &
   u^nv^2 &
   \cdots &
   u^nv^n
  \end{pmatrix}.
\end{equation}

\section{Homogeneous diffusion case}
\label{sec:2}

In this section, we consider the case of homogeneous diffusion and a
constant mobility, namely $D$ is a positive constant and
$\pi\equiv1$. We study the following periodic boundary value problem,

\begin{equation}
 \label{eq:2.1}
 \left\{
  \begin{aligned}
   \frac{\partial f}{\partial t}
   +
   \Div
   \left(
   f\vec{u}
   \right)
   &=
   0,
   &\quad
   &x\in\Omega,\quad
   t>0, \\
   \vec{u}
   &=
   -
   \nabla
   \left(
   D\log f
   +
   \phi(x)
   \right),
   &\quad
   &x\in\Omega,\quad
   t>0, \\
   f(x,0)&=f_0(x),&\quad
   &x\in\Omega.
  \end{aligned}
 \right.
\end{equation}
The equation \eqref{eq:2.1} is the linear Fokker-Planck
equation. The entropy dissipation method is among the powerful
tool (for instance, see~\cite{MR1842428, MR3497125,MR1812873,MR1639292}) for the study of long-time asymptotic behavior of solutions
to \eqref{eq:2.1}. Here,  we present a new take on the
entropy dissipation method with a help of the velocity vector
$\vec{u}$, \eqref{eq:2.1}. Such approach makes it possible to extend entropy dissipation method to the nonlinear problem
\eqref{eq:1.1}.

Under the assumption of the constant coefficients  $D>0$ and
$\pi=1$, the free energy $F$ and the energy law \eqref{eq:1.3} take
the form,
\begin{equation}
 \label{eq:2.2}
  F[f]
  :=
  \int_\Omega
  \left(
   Df(\log f -1)+f\phi(x)
  \right)
  \,dx,
\end{equation}
and
\begin{equation}
 \label{eq:2.3}
  \frac{dF}{dt}[f](t)
  =
  -\int_\Omega
  |\vec{u}|^2f\,dx
  =:
  - D_{\mathrm{dis}}[f](t),
\end{equation}
where $D_{\mathrm{dis}}[f](t)$ is the dissipation rate
of the free energy $F[f]$.

Let us first state the main result of this section,

\begin{theorem}
 \label{thm:2.9}
 Let $\phi=\phi(x)$ be a periodic function, and let $f_0=f_0(x)$ be a
 periodic probability density function which satisfies both the finite conditions that
 $F[f_0]<\infty$ and
 $D_{\mathrm{dis}}[f_0]<\infty$. Let $f$ be a solution of \eqref{eq:2.1} subject to
 the periodic boundary condition. Let $\vec{u}$ be defined as in
 \eqref{eq:2.1}. Assume, that there is a positive constant $\lambda>0$,
 such that $\nabla^2\phi\geq \lambda I$, where $I$ is the identity
 matrix. Then, we obtain that,
 \begin{equation}
  \label{eq:2.40}
   \int_\Omega
   |\vec{u}|^2f\,dx
   \leq
   e^{-2\lambda t}
   \int_\Omega
   |\nabla
   \left(
    D\log f_0
    +
    \phi(x)
   \right)|^2f_0\,dx.
 \end{equation}
 In particular, we have that,
 \begin{equation}
  \label{eq:2.36}
  \frac{dF}{dt}[f](t)
   =
   -\int_\Omega |\vec{u}|^2f\,dx
   \rightarrow 0
   \qquad
   \text{as}
   \
   t\rightarrow\infty.
 \end{equation}
\end{theorem}



In order to establish statement of Theorem \ref{thm:2.9}, first, we need to
obtain additional results as in
Lemmas~\ref{lem:2.1}-\ref{lem:2.6}. Using \eqref{eq:2.3} and the
Fubini theorem,  we start by showing that we can take subsequence such that $\frac{dF}{dt}[f]$
converges to $0$,.

\begin{lemma}
 \label{lem:2.1}
 Let $f$ be a solution of \eqref{eq:2.1}. Then,  there is an increasing
 sequence $\{t_j\}_{j=1}^\infty$,  such that $t_j\rightarrow\infty$ and,
 \begin{equation}
  \label{eq:2.31}
   \frac{dF}{dt}[f](t_j)\rightarrow 0,\qquad
   j\rightarrow\infty.
 \end{equation}
\end{lemma}

\begin{proof}
 Integrate \eqref{eq:2.3} with respect to $t$, we have that,
 \begin{equation}
  F[f](t)+\int_0^t\left(\int_\Omega|\vec{u}|^2f\,dx\right)\,d\tau=F[f_0].
 \end{equation}
 Since $F[f](t)\geq -\Cr{const:1.7}$ by Proposition
 \ref{prop:1.7}, we obtain the uniform bound,
 \begin{equation*}
   \int_0^t\left(\int_\Omega|\vec{u}|^2f\,dx\right)\,d\tau
   \leq
   F[f_0] + \Cr{const:1.7},
 \end{equation*}
 for $t>0$. Hence, there is an increasing sequence
 $\{t_j\}_{j=1}^\infty$ such that $t_j\rightarrow\infty$, and
 \begin{equation}
  \int_\Omega|\vec{u}|^2f\,dx\bigg|_{t=t_j}\rightarrow0,
   \qquad
   j\rightarrow\infty.
 \end{equation}
 Next, using \eqref{eq:2.3}, we obtain that $\frac{dF}{dt}[f](t_j)\rightarrow0$ as
 $j\rightarrow\infty$.
\end{proof}

Henceforth we compute the second derivative of $F$ and represent it in
terms of
$\vec{u}$. To do this, we first give a relation between $\nabla f$ and
$\vec{u}$. By direct calculation of the velocity $\vec{u}$, we have the
following result.
\begin{lemma}
 Let $\vec{u}$ be defined by \eqref{eq:2.1}. Then,
 \begin{equation}
  \label{eq:2.8}
   f\vec{u}
   =
   -
   D\nabla f
   -
   f\nabla\phi(x),
 \end{equation}
 and
 \begin{equation}
  \label{eq:2.6}
   \rho\vec{u}
   =
   -D\nabla \rho,
 \end{equation}
 where $\rho$ is defined in \eqref{eq:1.8}.
\end{lemma}

Next, let us take a second derivative of the free energy $F$ and we
have that,
\begin{equation}
\label{eq:2.7u}
 \frac{d^2F}{dt^2}[f]
  =
  \frac{d}{dt}
  \left(
   -
   \int_\Omega
   |\vec{u}|^2f\,dx
  \right)
  =
  -
  2
  \int_\Omega
  \vec{u}\cdot\vec{u}_t f\,dx
  -
   \int_\Omega
   |\vec{u}|^2f_t\,dx.
\end{equation}
Thus,  we need to compute the time derivative of $\vec{u}$.

\begin{lemma}
 Let $\vec{u}$ be defined as in \eqref{eq:2.1}. Then,
 \begin{equation}
  \label{eq:2.4}
   \vec{u}_t
   =
   -
   \frac{D}{\rho}\nabla \rho_t
   -
   \frac{\rho_t}{\rho}\vec{u}.
 \end{equation}
\end{lemma}

\begin{proof}
 Take a time-derivative of $\vec{u}$, then we obtain that,
 \begin{equation}
  \label{eq:2.5}
  \vec{u}_t
   =
   -\nabla \left(
	    D\frac{\rho_t}{\rho}
	   \right)
   =
   -
   \frac{D}{\rho}\nabla \rho_t
   +
   \frac{D\rho_t}{\rho^2}\nabla \rho.
 \end{equation}
 Using \eqref{eq:2.6} in \eqref{eq:2.5}, we derive \eqref{eq:2.4}.
\end{proof}

Next, we rewrite the second derivative of $F$ \eqref{eq:2.7u} in terms of $\rho_t$ and $f_t$,
instead of $\vec{u}_t$.

\begin{lemma}
 Let $f$ be a solution of \eqref{eq:2.1} and let $\vec{u}$ be defined
 as in
 \eqref{eq:2.1}. Then,
 \begin{equation}
  \label{eq:2.7}
   \frac{d^2F}{dt^2}[f](t)
   =
   2\int_\Omega D\vec{u}\cdot\nabla\rho_t\Feq\,dx
   +
   \int_\Omega|\vec{u}|^2f_t\,dx,
 \end{equation}
 where $\Feq$ is given by \eqref{eq:1.5}.
\end{lemma}

\begin{proof}
 Using \eqref{eq:2.7u} together with \eqref{eq:2.4}, we obtain that,
 \begin{equation*}
  \begin{split}
   \frac{d^2F}{dt^2}[f](t)
   &=
   -2\int_\Omega\vec{u}\cdot\vec{u}_t f\,dx
   -
   \int_\Omega|\vec{u}|^2f_t\,dx \\
   &=
   2\int_\Omega D\vec{u}\cdot\nabla\rho_t \frac{f}{\rho}\,dx
   +
   2\int_\Omega|\vec{u}|^2\rho_t \frac{f}{\rho}\,dx
   -
   \int_\Omega|\vec{u}|^2f_t\,dx.
  \end{split}
 \end{equation*}
 Since $f=\rho\Feq$ and $\rho_t\Feq=f_t$, we obtain the desired result
 \eqref{eq:2.7}.
\end{proof}

Now, let us reformulate the right-hand side of \eqref{eq:2.7} in a
form which is convenient for the use of entropy method.

\begin{lemma}
 \label{lem:2.4}
 Let $f$ be a solution of \eqref{eq:2.1} and let $\vec{u}$ be defined by
 \eqref{eq:2.1}. Then,
 \begin{equation}
  \label{eq:2.12}
   \int_\Omega |\vec{u}|^2f_t\,dx
   =
   \int_\Omega \vec{u}\cdot\nabla |\vec{u}|^2 f\,dx.
 \end{equation}
\end{lemma}

\begin{proof}
 Using the system \eqref{eq:2.1} and integration by parts together with
 the periodic boundary condition, we arrive at,
 \begin{equation*}
  \int_\Omega |\vec{u}|^2f_t\,dx
   =
   -\int_\Omega |\vec{u}|^2\Div(f\vec{u})\,dx
   =
   \int_\Omega \vec{u}\cdot\nabla |\vec{u}|^2f\,dx.
 \end{equation*}
\end{proof}


\begin{lemma}
 Let $f$ be a solution of \eqref{eq:2.1}, and let $\vec{u}$ be defined by
 \eqref{eq:2.1}. Then,
 \begin{equation}
  \label{eq:2.11}
   D\Feq\nabla\rho_t
   =
   -f_t\vec{u}
   +
   f
   \nabla
   \left(
    |\vec{u}|^2
   +
   (\vec{u}\cdot\nabla\phi(x))
   -
   D
   \Div\vec{u}
   \right),
 \end{equation}
 where $\Feq$ is given by \eqref{eq:1.5}.
\end{lemma}

\begin{proof}
 Since $\Feq$ is independent of $t$, we have due to \eqref{eq:2.1} that,
 \begin{equation}
  \label{eq:2.9}
   D\Feq\rho_t
   =
   Df_t
   =
   -D\Div(f\vec{u})
   =
   -
   D\vec{u}\cdot\nabla f
   -
   Df\Div\vec{u}.
 \end{equation}
 Using \eqref{eq:2.8} in \eqref{eq:2.9}, we obtain that,
 \begin{equation}
  \label{eq:2.10}
   D\Feq\rho_t
   =
   Df_t
   =
   f
   \left(
    |\vec{u}|^2
    +
    \vec{u}\cdot\nabla\phi(x)
    -
    D\Div\vec{u}
   \right).
 \end{equation}
 Next, take a gradient of \eqref{eq:2.10} and obtain using
 \eqref{eq:2.10}, that,
 \begin{equation}
  \label{eq:2.23}
 \begin{split}
   D\rho_t\nabla\Feq
   +
   D\Feq\nabla\rho_t
   &=
   \left(
   |\vec{u}|^2
   +
   \vec{u}\cdot\nabla\phi(x)
   -
   D\Div\vec{u}
   \right)
   \nabla f
   +
   f
   \nabla
   \left(
   |\vec{u}|^2
   +
   \vec{u}\cdot\nabla\phi(x)
   -
   D\Div\vec{u}
   \right) \\
   &=
   \frac{Df_t}{f}
   \nabla f
   +
   f
   \nabla
   \left(
   |\vec{u}|^2
   +
   \vec{u}\cdot\nabla\phi(x)
   -
   D\Div\vec{u}
   \right).
 \end{split}
 \end{equation}
 Now, taking a gradient of \eqref{eq:1.6} with $D(x)=D$, we have,
 \begin{equation}
   \label{eq:2.24}
 \frac{D}{\Feq}\nabla \Feq+\nabla\phi(x)=0.
 \end{equation}
 Thus, using \eqref{eq:2.24} and \eqref{eq:2.8} in \eqref{eq:2.23}, we
 have that,
 \begin{equation*}
  -\rho_t\Feq\nabla\phi(x)
   +
   D\Feq\nabla\rho_t
   =
   -
   f_t\vec{u}
   -
   f_t\nabla\phi(x)
   +
   f
   \nabla
   \left(
   |\vec{u}|^2
   +
   \vec{u}\cdot\nabla\phi(x)
   -
   D\Div\vec{u}
   \right).
 \end{equation*}
 Since $\rho_t\Feq=f_t$, we obtain the result \eqref{eq:2.11}.
\end{proof}

Now we are in a position to compute
$\int_\Omega D\vec{u}\cdot\nabla\rho_t\Feq \,dx $, which is the first term of the right
hand side of \eqref{eq:2.7}.

\begin{lemma}
 \label{lem:2.6}
 Let $f$ be a solution of \eqref{eq:2.1} and let $\vec{u}$ be given as in
 \eqref{eq:2.1}. Then,
 \begin{equation}
  \label{eq:2.21}
   2\int_\Omega D\vec{u}\cdot\nabla\rho_t\Feq\,dx
   =
   2\int_\Omega ((\nabla^2\phi(x)) \vec{u}\cdot\vec{u}) f\,dx
   -
   \int_{\Omega}\vec{u}\cdot\nabla|\vec{u}|^2 f\,dx
   +
   2\int_\Omega D|\nabla \vec{u}|^2 f\,dx.
 \end{equation}
Here, $\Feq$ is defined as in \eqref{eq:1.5}.
\end{lemma}


\begin{proof}%
 [Proof]
 First, we use \eqref{eq:2.11} and obtain,
 \begin{multline*}
  2\int_\Omega D\vec{u}\cdot\nabla\rho_t\Feq\,dx \\
   =
   -2\int_\Omega
   |\vec{u}|^2f_t\,dx
   +
   2\int_\Omega
   \vec{u}\cdot \nabla|\vec{u}|^2 f\,dx
   +
   2\int_\Omega
   \vec{u}\cdot \nabla (\vec{u}\cdot\nabla\phi(x)) f\,dx
   -
   2\int_\Omega
   D
   \vec{u}\cdot\nabla\Div\vec{u} f
   \,dx.
 \end{multline*}
 Using \eqref{eq:2.12}, the first and the second terms of the right hand
 side of the above relation are canceled, hence,
 \begin{equation}
  \label{eq:2.13}
   2\int_\Omega D\vec{u}\cdot\nabla\rho_t\Feq\,dx
   =
   2\int_\Omega
   \vec{u}\cdot \nabla (\vec{u}\cdot\nabla\phi(x)) f\,dx
   -
   2\int_\Omega
   D
   \vec{u}\cdot\nabla\Div\vec{u} f
   \,dx.
 \end{equation}

 Next, we compute $\vec{u}\cdot \nabla (\vec{u}\cdot\nabla\phi(x))$. We
 denote $\vec{u}=(u^l)_l$. Then, by direct calculations,  we obtain,
 \begin{equation}
  \label{eq:2.14}
 \begin{split}
   \vec{u}\cdot \nabla (\vec{u}\cdot\nabla\phi(x))
   &=
   \sum_{k,l}u^l(u^k\phi_{x_k}(x))_{x_l} \\
   &=
   \sum_{k,l}\phi_{x_kx_l}(x)u^lu^k
   +
   \sum_{k,l}u^k_{x_l}u^l\phi_{x_k}(x) \\
   &=
  ((\nabla^2\phi(x))\vec{u}\cdot\vec{u})
   +
   \sum_{k,l}u^k_{x_l}u^l\phi_{x_k}(x).
  \end{split}
 \end{equation}
 Since $\nabla\vec{u}=-\nabla^2(D\log\rho)$ is symmetric,
 \begin{equation}
  \label{eq:2.16}
  \sum_{l}u^k_{x_l}u^l
   =
   \sum_{l}u^l_{x_k}u^l
   =
   \frac12(|\vec{u}|^2)_{x_k},
 \end{equation}
 hence we obtain,
 \begin{equation}
 \label{eq:2.15}
  \vec{u}\cdot \nabla (\vec{u}\cdot\nabla\phi(x))
  =
  ((\nabla^2\phi(x))\vec{u}\cdot\vec{u})
  +
  \frac12\nabla(|\vec{u}|^2)\cdot\nabla\phi(x).
 \end{equation}

 Next, we compute $\vec{u}\cdot\nabla\Div\vec{u}$. By direct calculations,
 we have that,
 \begin{equation}
  \label{eq:2.17}
  \vec{u}\cdot\nabla\Div\vec{u}
   =
   \sum_{k,l}u^l(u^k_{x_k})_{x_l}
   =
   \sum_{k,l}(u^lu^k_{x_l})_{x_k}
   -
   \sum_{k,l}u_{x_k}^lu^k_{x_l}.
 \end{equation}
 Since $\nabla\vec{u}$ is symmetric, we can use \eqref{eq:2.16} and,
 \begin{equation}
  \label{eq:2.22}
  \sum_{k,l}u_{x_k}^lu^k_{x_l}
   =
   \sum_{k,l}u_{x_l}^ku^k_{x_l}
   =
   |\nabla \vec{u}|^2,
 \end{equation}
to obtain from \eqref{eq:2.17} that,
 \begin{equation}
  \label{eq:2.18}
   \vec{u}\cdot\nabla\Div\vec{u}
   =
   \frac{1}{2}\Div(\nabla |\vec{u}|^2)
   -
   |\nabla\vec{u}|^2.
 \end{equation}
 Employing \eqref{eq:2.15} and \eqref{eq:2.18} in \eqref{eq:2.13}, we
 derive that,
 \begin{multline}
  \label{eq:2.19}
  2\int_\Omega D\vec{u}\cdot\nabla\rho_t\Feq\,dx \\
   =
   2
   \int_\Omega
   ((\nabla^2\phi(x))\vec{u}\cdot\vec{u})f\,dx
   +
   \int_\Omega\nabla(|\vec{u}|^2)\cdot\nabla\phi(x) f\,dx
   -
   \int_\Omega
   D
   \Div(\nabla |\vec{u}|^2)
   f
   \,dx
  +
   2\int_\Omega D|\nabla \vec{u}|^2f\,dx.
 \end{multline}

 Next, we calculate the third term of the right-hand side of
 \eqref{eq:2.19}. Applying integration by parts together with the
 periodic boundary condition, we have,
 \begin{equation*}
  -
   \int_\Omega
   D
   \Div(\nabla |\vec{u}|^2)
   f
   \,dx
   =
   \int_\Omega
   D
   \nabla |\vec{u}|^2\cdot\nabla f
   \,dx.
    \end{equation*}
 Using \eqref{eq:2.8} in the above relation, we have,
 \begin{equation}
  \label{eq:2.20}
  -
   \int_\Omega
   D
   \Div(\nabla |\vec{u}|^2)
   f
   \,dx
   =
   -
   \int_\Omega
   \vec{u}\cdot\nabla |\vec{u}|^2 f
   \,dx
   -
   \int_\Omega
   \nabla |\vec{u}|^2\cdot\nabla\phi(x) f
   \,dx.
 \end{equation}
 Finally, using \eqref{eq:2.20} in \eqref{eq:2.19}, we obtain the
 desired result \eqref{eq:2.21}.
\end{proof}

Now, combining  results \eqref{eq:2.7}, \eqref{eq:2.12} and
\eqref{eq:2.21} from above lemmas,  we are in position to obtain the following energy law,
.
\begin{proposition}
 \label{prop:2.8}
 Let $f$ be a solution of \eqref{eq:2.1} and let $\vec{u}$ be defined in
 \eqref{eq:2.1}. Then,
 \begin{equation}
  \label{eq:2.25}
  \frac{d^2F}{dt^2}[f](t)
   =
   2\int_\Omega ((\nabla^2\phi(x)) \vec{u}\cdot\vec{u}) f\,dx
   +
   2\int_\Omega D|\nabla \vec{u}|^2 f\,dx.
 \end{equation}
\end{proposition}

From \eqref{eq:2.3}, as in Lemma \ref{lem:2.1}, we only know that there
is a subsequence $\{t_j\}_{j=1}^\infty$ such that
$\frac{dF}{dt}[f](t_j)$ converges to $0$. Now,  using \eqref{eq:2.25}, we
show full convergence of $\frac{dF}{dt}[f]$ to $0$.

\begin{proof}[Proof of Theorem \ref{thm:2.9}]
 First, from \eqref{eq:2.25} and \eqref{eq:2.3}, by the convexity
 assumption,  $\nabla^2\phi(x)\geq\lambda$, we get,
 \begin{equation}
  \label{eq:2.41}
  \frac{d^2F}{dt^2}[f](t)
   \geq
   2\lambda\int_\Omega |\vec{u}|^2 f\,dx
   \geq
   0,
 \end{equation}
 hence $\frac{dF}{dt}[f](t)$ is monotone increasing with respect to
 $t>0$. Thus, from \eqref{eq:2.31}, it follows that $\frac{dF}{dt}[f](t)$
 converges to $0$ as $t\rightarrow\infty$. Furthermore, from \eqref{eq:2.41} and \eqref{eq:2.3} we have,
 \begin{equation*}
  \frac{d}{dt}
   \left(
    -\int_\Omega|\vec{u}|^2f\,dx
   \right)
   \geq
   2\lambda\int_\Omega |\vec{u}|^2 f\,dx,
 \end{equation*}
 namely
 \begin{equation*}
  \frac{d}{dt}D_{\mathrm{dis}}[f](t)
   \leq
   -2\lambda D_{\mathrm{dis}}[f](t).
 \end{equation*}
 Hence, we can apply the Gronwall's inequality. Thus, we have
 $D_{\mathrm{dis}}[f](t)\leq e^{-2\lambda t}D_{\mathrm{dis}}[f](0)$, and
 obtain the result \eqref{eq:2.40}.
\end{proof}

\begin{remark}
 \label{rem:2.1}
Note, in Theorem \ref{thm:2.9}, we obtained the exponential decay
 of $D_{\mathrm{dis}}[f](t)$,  but we do not know long-time asymptotic
 behavior of $F[f](t)$ or $f(t)$ itself. On the other hand, using the
 logarithmic Sobolev inequality, we may show stronger convergence
 results, such as $F[f](t)\rightarrow F[f^{\mathrm{eq}}]$ exponentially,
 and exponential convergence of $f$ to $f^{\mathrm{eq}}$ in the $L^1$
 space, as $t\rightarrow\infty$.  When $\Omega=\R^n$, the logarithmic
 Sobolev inequality holds,  and we may proceed with the entropy
 dissipation method to obtain the energy convergence. We will discuss it
 in Appendix.
\end{remark}
In this section, we demonstrated the entropy method for the linear
Fokker-Planck equation in terms of the velocity $\vec{u}$. Using this approach, we will extend the entropy method to the nonlinear
Fokker-Planck equation in the next section.

\section{Inhomogeneous diffusion case}
\label{sec:3}

In this section, we consider the following evolution equation with  inhomogeneous diffusion and a
constant mobility, namely $D$ is a positive bounded function and $\pi\equiv1$ in a bounded domain in
the Euclidean space of $n$-dimension, subject to the periodic boundary condition,
\begin{equation}
 \label{eq:3.1}
 \left\{
  \begin{aligned}
   \frac{\partial f}{\partial t}
   +
   \Div
   \left(
   f\vec{u}
   \right)
   &=
   0,
   &\quad
   &x\in\Omega,\quad
   t>0, \\
   \vec{u}
   &=
   -
   \nabla
   \left(
   D(x)\log f
   +
   \phi(x)
   \right),
   &\quad
   &x\in\Omega,\quad
   t>0, \\
   f(x,0)&=f_0(x),&\quad
   &x\in\Omega.
  \end{aligned}
 \right.
\end{equation}
Without loss of generality, we take $\Omega=[0,1)^n\subset\R^n$. We first consider the strictly positive 
periodic function $D=D(x)$ with the lower bound,
 $\Cr{const:1.3}>0$ such that,
\begin{equation}
 \label{eq:3.60}
 D(x)\geq \Cr{const:1.3},
\end{equation}
 for $x\in\Omega$.

The free energy $F$ and  the basic energy law \eqref{eq:1.3} take the following specific forms,
\begin{equation}
 \label{eq:3.2}
  F[f]
  :=
  \int_\Omega
  \left(
   D(x)f(\log f -1)+f\phi(x)
  \right)
  \,dx,
\end{equation}
and
\begin{equation}
 \label{eq:3.3}
  \frac{dF}{dt}[f](t)
  =
  -\int_\Omega
  |\vec{u}|^2f\,dx
  =:
  - D_{\mathrm{dis}}[f](t).
\end{equation}

Here, first, we present the following Sobolev-type
inequality and the interpolation estimate, based on the uniform bounds
of the solution of the above system.

\begin{lemma}
 \label{lemma3.1}
Let $f$ be a solution of the equation \eqref{eq:3.1}.
 Then, there is a suitable positive constant $\Cl{const:3.1}>0$, such that for any
$t>0$,  and for any periodic vector field $\vec{v}$ on $\Omega$,
\begin{equation}
 \label{eq:3.42}
 \left(
  \int_\Omega|\vec{v}|^{p^\ast}f\,dx
 \right)^{\frac{1}{p^\ast}}
 \leq
 \Cr{const:3.1}
 \left(
  \int_\Omega|\nabla\vec{v}|^{2}f\,dx
 \right)^{\frac{1}{2}},
\end{equation}
where the exponent $p^\ast$ satisfies
$\frac{1}{p^\ast}=\frac{1}{2}-\frac{1}{n} $ for $n=3$,  and arbitrary
$2\leq p^{\ast}<\infty $ for $n=1,2$. 

In particular, with this Sobolev-type inequality
\eqref{eq:3.42} and the H\"older inequality, we have for $2\leq p\leq
p^\ast$ that,
\begin{equation}
\label{eq:3.43}
 \int_\Omega|\vec{v}|^p f\,dx
  \leq
  \left(
   \int_\Omega|\vec{v}|^{p^\ast} f\,dx
  \right)^{\frac{p}{p^\ast}}
  \left(
   \int_\Omega
   f\,dx
  \right)^{1-\frac{p}{p^\ast}}
  \leq
  \Cr{const:3.1}^p
  \left(
   \int_\Omega|\nabla\vec{v}|^{2}f\,dx
  \right)^{\frac{p}{2}}
  \left(
   \int_\Omega
   f\,dx
  \right)^{1-\frac{p}{p^\ast}}.
\end{equation}
\end{lemma}

\begin{proof}
 Let us justify the above Sobolev inequality.
 The exponent $p^\ast$ is the so-called Sobolev exponent.
 The above Sobolev inequality holds when $f$ is strictly positive and bounded
 uniformly on $\Omega\times[0,\infty)$, namely,  there are positive
 constants $\Cl{const:3.3},\Cl{const:3.4}>0$,  such that
 $\Cr{const:3.3}\leq f(x,t)\leq\Cr{const:3.4}$ for $x\in\Omega$ and
 $t\geq0$, see Section~\ref{sec:1}, Proposition~\ref{prop:1.5}.  To
 see this, we use the classical Sobolev inequality
 (see, for instance \cite{MR2424078}),
 \begin{equation*}
  \left(
  \int_\Omega|\vec{v}|^{p^\ast}\,dx
 \right)^{\frac{1}{p^\ast}}
 \leq
 \Cl{const:3.5}
 \left(
  \int_\Omega|\nabla\vec{v}|^{2}\,dx
 \right)^{\frac{1}{2}}.
 \end{equation*}
 Thus, using $\Cr{const:3.3}\leq f(x,t)\leq\Cr{const:3.4}$,  we have,
 \begin{equation*}
  \begin{split}
   \left(
   \int_\Omega|\vec{v}|^{p^\ast}f\,dx
   \right)^{\frac{1}{p^\ast}}
   &\leq
   \Cr{const:3.4}
   \left(
   \int_\Omega|\vec{v}|^{p^\ast}\,dx
   \right)^{\frac{1}{p^\ast}}
   \\
   &\leq
   \Cr{const:3.4}
   \Cr{const:3.5}
   \left(
   \int_\Omega|\nabla\vec{v}|^{p^\ast}\,dx
   \right)^{\frac{1}{p^\ast}}
   \leq
   \frac{\Cr{const:3.4}\Cr{const:3.5}}{\Cr{const:3.3}}
   \left(
   \int_\Omega|\nabla\vec{v}|^{2}f\,dx
   \right)^{\frac{1}{2}},
  \end{split}
 \end{equation*}
 so we can take
 $\Cr{const:3.1}=\frac{\Cr{const:3.4}\Cr{const:3.5}}{\Cr{const:3.3}}$.
\end{proof}

The Theorem~\ref{thm:3.13} below is the extension of the results in
Section~\ref{sec:2}, Theorem~\ref{thm:2.9}, when
$D$ is constant (that is,  $\|\nabla D\|_{L^\infty(\Omega)}=0$), to
the case of the inhomogeneous $D(x)$. In particular, 
when $\|\nabla D\|_{L^\infty(\Omega)}$ is sufficiently small, and
under some additional assumptions on
the initial condition, one can establish that the dissipation functional $D_{\mathrm{dis}}[f](t)$ in the basic energy law \eqref{eq:3.3}
will also exponentially
converges to $0$ as $t\rightarrow\infty$.

\begin{theorem}
 \label{thm:3.13} Assume $n=1,2,3$. Let $\phi=\phi(x)$ and $D=D(x)$ be
 periodic functions, and let $f_0=f_0(x)$ be a periodic probability
 density function. Let $f$ be a solution of \eqref{eq:3.1} subject to
 the periodic boundary condition. Let $\vec{u}$ be defined as in
 \eqref{eq:3.1}. Assume,  that there is a positive constant $\lambda>0$,
 such that $\nabla^2\phi\geq \lambda I$, where $I$ is the identity
 matrix. Then, there are constants $\Cl{const:3.6}$, $\Cl{const:3.7}$,
 $\Cl{const:3.8}>0$ such that,  if
 \begin{equation}
  \label{eq:3.58}
  \|\nabla D\|_{L^\infty(\Omega)}\leq \Cr{const:3.6},\qquad
   \int_\Omega |\nabla(D(x)\log f_0+\phi(x))|^2f_0\,dx
   \leq\Cr{const:3.7},
 \end{equation}
 then,  we obtain for $t>0$,
 \begin{equation}
  \label{eq:3.55}
   \int_\Omega |\vec{u}|^2f\,dx
   \leq
   \Cr{const:3.8}
   e^{-\lambda t}.
 \end{equation}
 In particular, we have that,
 \begin{equation}
  \frac{dF}{dt}[f](t)
   =
   -\int_\Omega |\vec{u}|^2f\,dx
   \rightarrow 0,
   \qquad
   \text{as}
   \
   t\rightarrow\infty.
 \end{equation}
\end{theorem}


\begin{remark}
\label{remark: 3.3}
 Let $d\mu=f^{\mathrm{eq}}\,dx$. Then, the estimate \eqref{eq:3.55} can also be written as,
\begin{equation*}
 \int_\Omega |\vec{u}\rho^{\frac12}|^2\,d\mu
  \leq
  \Cr{const:3.8}
  e^{-\lambda t}.
 \end{equation*}
 In other words,  $\vec{u}\rho^{\frac12}$ converges exponentially fast to $\vec{0}$ as
 $t\rightarrow\infty$ in $L^2(\Omega,d\mu)$. 
 
 Due to the fact that  $\nabla (D(x)\log
 f^{\mathrm{eq}}(x)+\phi(x))=\vec{0}$, we can further conclude that,
 \begin{equation}
  \label{eq:3.56}
  \left(\frac{f}{f_0}\right)^\frac{1}{2}
   \nabla (D(x)\log f+\phi(x))
   \rightarrow
   \left(\frac{f^{\mathrm{eq}}}{f_0}\right)^\frac{1}{2}
   \nabla (D(x)\log f^{\mathrm{eq}}(x)+\phi(x))\quad
   \text{exponentially fast in}\
   L^2(\Omega,d\mu),
 \end{equation}
 as $t\rightarrow\infty$, provided $f$ does not become $0$. 
 In particular,  this is true when  $f$ is strictly positive  on
 $\Omega\times[0,\infty)$.
\end{remark}

\begin{remark}
It is clear  that  in  the conditions \eqref{eq:3.58} of Theorem \ref{thm:3.13}, the first one is for $D(x)$, while the 
second one is on the initial data of $f_0$. Such conditions are needed in our analysis to  get the asymptotic convergence 
of the dissipation $D_{\mathrm{dis}}[f](t)$ in the basic energy law \eqref{eq:3.3}.
\end{remark}
\par In order to establish statement of Theorem \ref{thm:2.9}, first, we need to
obtain additional results as in
Lemmas~\ref{lem:3.1}-\ref{lem:3.12}. 
Note that as in the proof of Lemma \ref{lem:2.1}, we can take a subsequence
$\{t_j\}_{j=1}^\infty$ such that $\frac{dF}{dt}[f](t_j)$ vanishes as
$j\rightarrow\infty$. Namely,

\begin{lemma}
 \label{lem:3.1}
 Let $f$ be a solution of \eqref{eq:3.1}. Then there is an increasing
 sequence $\{t_j\}_{j=1}^\infty$ such that $t_j\rightarrow\infty$ and,
 \begin{equation}
  \label{eq:3.41}
   \frac{dF}{dt}[f](t_j)\rightarrow 0,\qquad \mbox{ as }
   j\rightarrow\infty.
 \end{equation}
\end{lemma}

The proof of Lemma \ref{lem:3.1} follows exactly the same argument as the proof of
Lemma \ref{lem:2.1}. We next show that $\frac{dF}{dt}[f]$ 
converges to $0$ as $t\rightarrow\infty$ in time $t$.

Hereafter we compute the second derivative of $F$ and represent it by
$\vec{u}$. To do this, we first establish a relationship between $\nabla f$ and
$\vec{u}$. By direct calculation of the velocity $\vec{u}$, we have the
following relation.
\begin{lemma}
 Let $\vec{u}$ be defined as in \eqref{eq:3.1}. Then,
 \begin{equation}
  \label{eq:3.8}
   f\vec{u}
   =
   -
   D(x)\nabla f
   -
   f\log f\nabla D(x)
   -
   f\nabla\phi(x),
 \end{equation}
 and
 \begin{equation}
  \label{eq:3.6}
   \rho\vec{u}
   =
   -
   D(x)\nabla \rho
   -
   \rho\log\rho\nabla D(x),
 \end{equation}
 where $\rho$ is defined in \eqref{eq:1.8}.
\end{lemma}

Next, we notice again that the nonlinearity in \eqref{eq:3.8} is the direct consequence
of the  inhomogeneity of $D(x)$. Moreover, the nonlinear part of the
system in \eqref{eq:1.9}
will become,
$ -  \log f\nabla D(x)\cdot\nabla f-
  \Delta D(x) f\log f.$

Next, again, to use the entropy method, we will take a second derivative of the free energy $F$, 
\begin{equation*}
 \frac{d^2F}{dt^2}[f]
  =
  \frac{d}{dt}
  \left(
   -
   \int_\Omega
   |\vec{u}|^2f\,dx
  \right)
  =
  -
  2
  \int_\Omega
  \vec{u}\cdot\vec{u}_t f\,dx
  -
   \int_\Omega
   |\vec{u}|^2f_t\,dx.
\end{equation*}

Next, similar to Section~\ref{sec:2}, we proceed by first computing the time derivative of $\vec{u}$.

\begin{lemma}
 Let $\vec{u}$ be defined by \eqref{eq:3.1}. Then,
 \begin{equation}
  \label{eq:3.4}
   \vec{u}_t
   =
   -
   \frac{D(x)}{\rho}\nabla \rho_t
   -
   \frac{\rho_t}{\rho}
   \left(
    \vec{u}
    +
    (\log\rho+1)
   \nabla D(x)
   \right).
 \end{equation}
\end{lemma}

\begin{proof}
 We take a time-derivative of $\vec{u}$ and we have from
 \eqref{eq:3.1} that,
 \begin{equation}
  \label{eq:3.5}
  \vec{u}_t
   =
   -\nabla \left(
	    D(x)\frac{\rho_t}{\rho}
	   \right)
   =
   -
   \frac{D(x)}{\rho}\nabla \rho_t
   +
   \frac{D(x)\rho_t}{\rho^2}\nabla \rho
   -
   \frac{\rho_t}{\rho}\nabla D(x).
 \end{equation}
 Using \eqref{eq:3.6} in \eqref{eq:3.5}, we obtain the result \eqref{eq:3.4}.
\end{proof}

Note, by comparing formula in \eqref{eq:3.4} with the formula in
\eqref{eq:2.4}, one can observe that the extra term $\frac{\rho_t}{\rho}(\log\rho+1)\nabla
D(x)$ appears in the time derivative of $\vec{u}$ due to the effect
of the  inhomogeneity.

Next, similar to Section~\ref{sec:2}, we will write the second time-derivative of $F$ in terms of $\rho_t$ and $f_t$,
instead of $\vec{u}_t$.
\begin{lemma}
 Let $f$ be a solution of \eqref{eq:3.1} and let $\vec{u}$ be given by
 \eqref{eq:3.1}. Then,
 \begin{equation}
  \label{eq:3.7}
   \frac{d^2F}{dt^2}[f](t)
   =
   2\int_\Omega D(x)\vec{u}\cdot\nabla\rho_t\Feq\,dx
   +
   \int_\Omega|\vec{u}|^2f_t\,dx
   +
   2\int_\Omega(\log\rho+1)\vec{u}\cdot\nabla D(x)f_t\,dx,
 \end{equation}
 where $\Feq$ is given in \eqref{eq:1.5}.
\end{lemma}

\begin{proof}
 Using the time-derivative of \eqref{eq:3.3} together
 with \eqref{eq:3.4}, we obtain that,
 \begin{equation*}
  \begin{split}
   \frac{d^2F}{dt^2}[f](t)
   &=
   -2\int_\Omega\vec{u}\cdot\vec{u}_t f\,dx
   -
   \int_\Omega|\vec{u}|^2f_t\,dx \\
   &=
   2\int_\Omega D(x)\vec{u}\cdot\nabla\rho_t \frac{f}{\rho}\,dx
   +
   2\int_\Omega|\vec{u}|^2\rho_t \frac{f}{\rho}\,dx
   +
   2\int_\Omega(\log\rho+1)\vec{u}\cdot\nabla D(x)\rho_t \frac{f}{\rho}\,dx \\
   &\qquad
   -
   \int_\Omega|\vec{u}|^2f_t\,dx.
  \end{split}
 \end{equation*}
 Since $f=\rho\Feq$ and $\rho_t\Feq=f_t$, we derive \eqref{eq:3.7}.
\end{proof}

Next, we compute the right-hand side of \eqref{eq:3.7}.  Similar to
Lemma \ref{lem:2.4} in Section~\ref{sec:2}, one can show the following
result using the same argument as in Lemma \ref{lem:2.4}.

\begin{lemma}
 Let $f$ be a solution of \eqref{eq:3.1} and let $\vec{u}$ be given by
 \eqref{eq:3.1}. Then,
 \begin{equation}
  \label{eq:3.12}
   \int_\Omega |\vec{u}|^2f_t\,dx
   =
   \int_\Omega \vec{u}\cdot\nabla |\vec{u}|^2 f\,dx.
 \end{equation}
\end{lemma}


Next, we express $\nabla\rho_t$ in terms of $\vec{u}$ in order to compute
the first term of the right-hand side of \eqref{eq:3.7}.

\begin{lemma}
 Let $f$ be a solution of \eqref{eq:3.1} and let $\vec{u}$ be given as in
 \eqref{eq:3.1}. Then,
 \begin{equation}
  \label{eq:3.11}
   \begin{split}
    &\qquad
    D(x)\Feq\nabla\rho_t \\
    &=
    -f_t\vec{u}
    -f_t
    \left(
    1+\log\rho
    \right)
    \nabla D(x)
    +
    f
    \nabla
    \left(
    |\vec{u}|^2
    +
    \log f
    \vec{u}\cdot\nabla D(x)
    +
    \vec{u}\cdot\nabla\phi(x)
    -
    D(x)\Div\vec{u}
    \right),
   \end{split}
 \end{equation}
 where $\Feq$ is given in \eqref{eq:1.5}.
\end{lemma}
\begin{proof}
 Since $\Feq$ is independent of $t$, we have due to \eqref{eq:3.1} that,
 \begin{equation}
  \label{eq:3.9}
   D(x)\Feq\rho_t
   =
   D(x)f_t
   =
   -D(x)\Div(f\vec{u})
   =
   -
   D(x)\vec{u}\cdot \nabla f
   -
   D(x)f\Div\vec{u}.
 \end{equation}
 Using \eqref{eq:3.8} in \eqref{eq:3.9}, we obtain that,
 \begin{equation}
  \label{eq:3.10}
   D(x)\Feq\rho_t
   =
   D(x)f_t
   =
   f
   \left(
    |\vec{u}|^2
    +
    \log f
    \vec{u}\cdot\nabla D(x)
    +
    \vec{u}\cdot\nabla\phi(x)
    -
    D(x)\Div\vec{u}
   \right).
 \end{equation}
 Next, take a gradient of \eqref{eq:3.10} and we obtain, by using
 \eqref{eq:3.10} again that,
 \begin{equation}
  \label{eq:3.14}
  \begin{split}
   &\qquad
   \rho_t\Feq\nabla D(x)
   +
   D(x)\rho_t\nabla\Feq
   +
   D(x)\Feq\nabla\rho_t \\
   &=
   \left(
   |\vec{u}|^2
   +
   \log f
   \vec{u}\cdot\nabla D(x)
   +
   \vec{u}\cdot\nabla\phi(x)
   -
   D(x)\Div\vec{u}
   \right)
   \nabla f \\
   &\qquad
   +
   f
   \nabla
   \left(
   |\vec{u}|^2
   +
   \log f
   \vec{u}\cdot\nabla D(x)
   +
   \vec{u}\cdot\nabla\phi(x)
   -
   D(x)\Div\vec{u}
   \right) \\
   &=
   \frac{D(x)f_t}{f}
   \nabla f
   +
   f
   \nabla
   \left(
   |\vec{u}|^2
   +
   \log f
   \vec{u}\cdot\nabla D(x)
   +
   \vec{u}\cdot\nabla\phi(x)
   -
   D(x)\Div\vec{u}
   \right).
  \end{split}
 \end{equation}
 Next, take a gradient of \eqref{eq:1.6}, we have,
 \begin{equation}
  \label{eq:3.16}
  \frac{D(x)}{\Feq}\nabla \Feq+\log\Feq\nabla D(x)+\nabla\phi(x)=0.
 \end{equation}
 Thus, using \eqref{eq:3.8} and \eqref{eq:3.16} in \eqref{eq:3.14}, we
 arrive at,
 \begin{equation*}
  \begin{split}
   &\qquad
   \rho_t\Feq\nabla D(x)
   -
   \rho_t
   \Feq
   \log\Feq\nabla D(x)
   -
   \rho_t
   \Feq
   \nabla \phi(x)
   +
   D(x)\Feq\nabla\rho_t \\
   &=
   -
   f_t\vec{u}
   -
   f_t\log f\nabla D(x)
   -
   f_t\nabla\phi(x)
   +
   f
   \nabla
   \left(
   |\vec{u}|^2
   +
   \log f
   \vec{u}\cdot\nabla D(x)
   +
   \vec{u}\cdot\nabla\phi(x)
   -
   D(x)\Div\vec{u}
   \right).
  \end{split}
 \end{equation*}
 Since $\rho_t\Feq=f_t$, we obtain the desired result \eqref{eq:3.11}.
\end{proof}

Now, we are in a position to compute the first term of the right
hand side of \eqref{eq:3.7}. 

\begin{lemma}
 \label{lem:3.7}
 Let $f$ be a solution of \eqref{eq:3.1} and let $\vec{u}$ be given as in
 \eqref{eq:3.1}. Then,
 \begin{equation}
  \label{eq:3.21}
   \begin{split}
    2\int_\Omega D(x)\vec{u}\cdot\nabla\rho_t\Feq\,dx
    &=
    2\int_\Omega
    ((\nabla^2\phi(x))\vec{u}\cdot\vec{u})f\,dx
    -
    \int_\Omega
    \vec{u}\cdot\nabla |\vec{u}|^2 f
    \,dx
    +
    2\int_\Omega D(x)|\nabla \vec{u}|^2f\,dx
    \\
    &\qquad
    -
    2\int_\Omega
    (1+\log\rho)\vec{u}\cdot\nabla D(x)f_t\,dx
    +
    2\int_\Omega
    \vec{u}\cdot\nabla\left(\log f\vec{u}\cdot\nabla D(x)\right) f
    \,dx
    \\
    &\qquad
    -
    \int_\Omega
    (\log f-1)
    \nabla |\vec{u}|^2\cdot\nabla D(x) f
    \,dx
    -
    2\int_\Omega \vec{u}\cdot\nabla D(x)\Div\vec{u}f\,dx,
   \end{split}
 \end{equation}
 where $\Feq$ is given in \eqref{eq:1.5}.
\end{lemma}


\begin{proof}
 [Proof]
 First, we use \eqref{eq:3.11} and obtain,
 \begin{equation*}
  \begin{split}
   &\qquad
   2\int_\Omega D(x)\vec{u}\cdot\nabla\rho_t\Feq\,dx \\
   &=
   -2\int_\Omega
   |\vec{u}|^2f_t\,dx
   -2\int_\Omega
   (1+\log\rho)\vec{u}\cdot\nabla D(x)f_t\,dx
   +
   2\int_\Omega
   \vec{u}\cdot \nabla|\vec{u}|^2 f\,dx \\
   &\qquad
   +
   2\int_\Omega
   \vec{u}\cdot\nabla\left(\log f\vec{u}\cdot\nabla D(x)\right) f
   \,dx
   +
   2\int_\Omega
   \vec{u}\cdot \nabla (\vec{u}\cdot\nabla\phi(x)) f\,dx
   -
   2\int_\Omega
   \vec{u}\cdot\nabla(D(x)\Div\vec{u}) f
   \,dx.
  \end{split}
 \end{equation*}
 Using \eqref{eq:3.12}, the first and the third terms of the right-hand
 side of the above relation are canceled, hence,
 \begin{equation}
  \label{eq:3.13}
  \begin{split}
   &\qquad
   2\int_\Omega D(x)\vec{u}\cdot\nabla\rho_t\Feq\,dx \\
   &=
   -2\int_\Omega
   (1+\log\rho)\vec{u}\cdot\nabla D(x)f_t\,dx
   +
   2\int_\Omega
   \vec{u}\cdot\nabla\left(\log f\vec{u}\cdot\nabla D(x)\right) f
   \,dx \\
   &\qquad
   +
   2\int_\Omega
   \vec{u}\cdot \nabla (\vec{u}\cdot\nabla\phi(x)) f\,dx
   -
   2\int_\Omega
   \vec{u}\cdot\nabla(D(x)\Div\vec{u}) f
   \,dx.
  \end{split}
 \end{equation}

 Since $\nabla \vec{u}=-\nabla (D(x)\log\rho)$ is symmetric, we can
 proceed with the same computations as in \eqref{eq:2.14}, \eqref{eq:2.16} in
 the proof of Lemma \ref{lem:2.6} in Section~\ref{sec:2},  hence we
 obtain that,
 \begin{equation}
 \label{eq:3.15}
  \vec{u}\cdot \nabla (\vec{u}\cdot\nabla\phi(x))
  =
  ((\nabla^2\phi(x))\vec{u}\cdot\vec{u})
  +
  \frac12\nabla |\vec{u}|^2\cdot\nabla\phi(x).
 \end{equation}

 Next, we compute $\vec{u}\cdot\nabla(D(x)\Div\vec{u})$. By the direct
 calculations, we have that,
 \begin{equation}
  \label{eq:3.17}
   \vec{u}\cdot\nabla(D(x)\Div\vec{u})
   =
   \vec{u}\cdot\nabla D(x)\Div\vec{u}
   +
   D(x)\vec{u}\cdot\nabla(\Div\vec{u}).
 \end{equation}
 Since $\nabla\vec{u}$ is symmetric, we can proceed again with the same computations
as in \eqref{eq:2.17},
 \eqref{eq:2.22} in the proof of Lemma \ref{lem:2.6} in
 Section~\ref{sec:2},  hence,  we obtain
 from \eqref{eq:3.17} that,
 \begin{equation}
  \label{eq:3.18}
   \vec{u}\cdot\nabla(D(x)\Div\vec{u})
   =
   \vec{u}\cdot\nabla D(x)\Div\vec{u}
   +
   \frac{D(x)}{2}\Div(\nabla |\vec{u}|^2)
   -
   D(x)|\nabla\vec{u}|^2.
 \end{equation}
 Using \eqref{eq:3.15} and \eqref{eq:3.18} in \eqref{eq:3.13}, we have,
 \begin{equation}
  \label{eq:3.19}
   \begin{split}
    &\qquad
    2\int_\Omega D(x)\vec{u}\cdot\nabla\rho_t\Feq\,dx \\
    &=
    -
    2\int_\Omega
    (1+\log\rho)\vec{u}\cdot\nabla D(x)f_t\,dx
    +
    2\int_\Omega
    \vec{u}\cdot\nabla\left(\log f\vec{u}\cdot\nabla D(x)\right) f
    \,dx \\
    &\qquad
    +
    2\int_\Omega
    ((\nabla^2\phi(x))\vec{u}\cdot\vec{u})f\,dx
    +
    \int_\Omega\nabla|\vec{u}|^2\cdot\nabla\phi(x) f\,dx \\
    &\qquad
    -
    \int_\Omega
    D(x)
    \Div(\nabla |\vec{u}|^2)
    f
    \,dx
    +
    2\int_\Omega D(x)|\nabla \vec{u}|^2f\,dx
    -
    2\int_\Omega \vec{u}\cdot\nabla D(x)\Div\vec{u}f\,dx.
   \end{split}
 \end{equation}

 Next, we calculate the fifth term of the right-hand side of
 \eqref{eq:3.19}. Applying integration by parts together with the
 periodic boundary condition, we arrive at,
 \begin{equation*}
  -
   \int_\Omega
   D(x)
   \Div(\nabla |\vec{u}|^2)
   f
   \,dx
   =
   \int_\Omega
   D(x)
   \nabla |\vec{u}|^2\cdot\nabla f
   \,dx
   +
   \int_\Omega
   \nabla |\vec{u}|^2\cdot\nabla D(x) f
   \,dx.
    \end{equation*}
Using \eqref{eq:3.8} in the first term of the right-hand side of
 the above relation, we have,
 \begin{equation}
  \label{eq:3.20}
   \begin{split}
    -
    \int_\Omega
    D(x)
    \Div(\nabla |\vec{u}|^2)
    f
    \,dx
    &=
    -
    \int_\Omega
    \vec{u}\cdot\nabla |\vec{u}|^2 f
    \,dx
    -
    \int_\Omega
    \nabla |\vec{u}|^2\cdot\nabla\phi(x) f
    \,dx
    \\
    &\qquad
    -
    \int_\Omega
    (\log f-1)
    \nabla |\vec{u}|^2\cdot\nabla D(x) f
    \,dx.
   \end{split}
 \end{equation}
Finally employing \eqref{eq:3.20} in \eqref{eq:3.19}, we obtain the
result \eqref{eq:3.21}.
\end{proof}

\begin{lemma}
 Let $\vec{u}$ be given by \eqref{eq:3.1}. Then,
 \begin{equation}
  \label{eq:3.23}
  \begin{split}
   &\qquad
   \int_\Omega
   \vec{u}\cdot\nabla\left(\log f\vec{u}\cdot\nabla D(x)\right) f\,dx
   \\
   &=
   \int_\Omega
   \frac{1}{D(x)}|\vec{u}|^2
   \log f\left(\vec{u}\cdot\nabla D(x)\right) f\,dx
   +
   \int_\Omega
   \frac{1}{D(x)}
   (\log f)^2\left(\vec{u}\cdot\nabla D(x)\right)^2 f\,dx
   \\
   &\qquad
   +
   \int_\Omega
   \frac{1}{D(x)}
   \log f\left(\vec{u}\cdot\nabla D(x)\right)\left(\vec{u}\cdot\nabla \phi(x)\right) f\,dx
   -
   \int_\Omega
   \log f\left(\vec{u}\cdot\nabla D(x)\right)\Div\vec{u} f\,dx.
  \end{split}
 \end{equation}
\end{lemma}

\begin{proof}
 Applying integration by parts to the left hand side of \eqref{eq:3.23}
 together with the periodic boundary condition \eqref{eq:3.1}, we
 obtain that,
 \begin{equation}
  \label{eq:3.24}
  \int_\Omega
   \vec{u}\cdot\nabla\left(\log f\vec{u}\cdot\nabla D(x)\right) f\,dx
   =
   -
   \int_\Omega
   \log f\left(\vec{u}\cdot\nabla D(x)\right) \Div(f\vec{u})\,dx.
 \end{equation}
Using \eqref{eq:3.8}, we have that,
 \begin{equation}
  \label{eq:3.25}
  \begin{split}
   \Div(f\vec{u})
   &=
   \vec{u}\cdot\nabla f
   +
   f\Div\vec{u} \\
   &=
   -
   \frac{1}{D(x)}|\vec{u}|^2f
   -
   \frac{1}{D(x)}\log f(\vec{u}\cdot\nabla D(x))f
   -
   \frac{1}{D(x)}(\vec{u}\cdot\nabla \phi(x))f
   +
   f\Div\vec{u}.
  \end{split}
 \end{equation}
 Combining \eqref{eq:3.24} and \eqref{eq:3.25}, we obtain the desired relation
 \eqref{eq:3.23}.
\end{proof}

Now combining \eqref{eq:3.7}, \eqref{eq:3.12}, \eqref{eq:3.21}, and
\eqref{eq:3.23}, we obtain the following energy law.

\begin{proposition}
 Let $f$ be a solution of \eqref{eq:3.1} and let $\vec{u}$ be given as in
 \eqref{eq:3.1}. Then,
 \begin{equation}
  \label{eq:3.22}
   \begin{split}
    \frac{d^2F}{dt^2}[f](t)
    &=
    2\int_\Omega ((\nabla^2\phi(x)) \vec{u}\cdot\vec{u}) f\,dx
    +
    2\int_\Omega D(x)|\nabla \vec{u}|^2 f\,dx
    \\
    &\qquad
    -
    \int_\Omega
    (\log f-1)
    \nabla |\vec{u}|^2\cdot\nabla D(x) f
    \,dx
    -
    2
    \int_\Omega (1+\log f)\vec{u}\cdot\nabla D(x)\Div\vec{u}f\,dx
    \\
    &\qquad
    +
    2
    \int_\Omega
    \frac{1}{D(x)}|\vec{u}|^2
    \log f\left(\vec{u}\cdot\nabla D(x)\right) f\,dx
    +
    2
    \int_\Omega
    \frac{1}{D(x)}
    (\log f)^2\left(\vec{u}\cdot\nabla D(x)\right)^2 f\,dx
    \\
    &\qquad
    +
    2
    \int_\Omega
    \frac{1}{D(x)}
    \log f\left(\vec{u}\cdot\nabla D(x)\right)\left(\vec{u}\cdot\nabla \phi(x)\right) f\,dx.
   \end{split}
 \end{equation}
\end{proposition}

Below, we will derive the condition that is sufficient  to obtain a differential inequality for
$\frac{dF}{dt}$. Note, the fifth term of the right-hand side of
\eqref{eq:3.22} involves $|\vec{u}|^3$, which is higher order than the term
$\frac{dF}{dt}$. Thus, we will handle such term using the following Sobolev
inequality.

\begin{lemma}
 \label{lem:3.13}
 Assume $n=1,2,3$, let $f_0$ be a probability density function, and let
 $f$ be a solution of \eqref{eq:3.1}. Then,
 \begin{equation}
  \label{eq:3.44}
   \int_\Omega
   |\vec{v}|^3f
   \,dx
   \leq
   \frac{3\Cr{const:3.1}^{\frac{3}{2}}}{4}
   \int_\Omega
   |\nabla\vec{v}|^2f
   \,dx
   +
   \frac{\Cr{const:3.1}^{\frac{3}{2}}}{4}
   \left(
    \int_\Omega
    |\vec{v}|^2f
    \,dx
   \right)^3,
 \end{equation}
 for any vector field $\vec{v}$.
\end{lemma}

\begin{proof}
 %
 Let $\alpha,\beta>0$,  such that $\alpha+\beta=1$,
 and let the exponent,  $p>1$. Then, by the H\"older's inequality,
 \begin{equation}
  \label{eq:3.45}
   \int_\Omega
   |\vec{v}|^3f
   \,dx
   \leq
   \left(
    \int_\Omega
   |\vec{v}|^{3\alpha p}f
   \,dx
   \right)^{\frac{1}{p}}
   \left(
    \int_\Omega
   |\vec{v}|^{3\beta p'}f
   \,dx
   \right)^{\frac{1}{p'}},
 \end{equation}
 where $p'$ is the H\"older's conjugate, namely,  $\frac1p+\frac1{p'}=1$.
 Next, we assume a constraint,  $3\alpha p=p^\ast$,  in order to apply the Sobolev
 inequality \eqref{eq:3.42} in \eqref{eq:3.45} and,
 \begin{equation}
  \label{eq:3.46}
   \int_\Omega
   |\vec{v}|^3f
   \,dx
   \leq
   \Cr{const:3.1}^{\frac{p^\ast}{p}}
   \left(
    \int_\Omega
    |\nabla\vec{v}|^2f
    \,dx
   \right)^{\frac{p^\ast}{2p}}
   \left(
    \int_\Omega
   |\vec{v}|^{3\beta p'}f
   \,dx
   \right)^{\frac{1}{p'}}.
 \end{equation}
 Next,  we assume another constraint,  $3\beta p'=2$ and
 $\frac{p^\ast}{2p}<1$. Then,  the Young's inequality implies,
 \begin{equation}
  \left(
   \int_\Omega
   |\nabla\vec{v}|^2f
   \,dx
  \right)^{\frac{p^\ast}{2p}}
  \left(
   \int_\Omega
   |\vec{v}|^{3\beta p'}f
   \,dx
  \right)^{\frac{1}{p'}}
  \leq
  \frac{p^\ast}{2p}
  \int_\Omega
   |\nabla\vec{v}|^2f
   \,dx
   +
  \left(
   1-
   \frac{p^\ast}{2p}
  \right)
  \left(
   \int_\Omega
   |\vec{v}|^{2}f
   \,dx
  \right)^{\frac{1}{p'}
  \left(
   1-
   \frac{p^\ast}{2p}
  \right)^{-1}
  },
 \end{equation}
 hence,  we obtain using \eqref{eq:3.46} that,
  \begin{equation}
  \label{eq:3.47}
   \int_\Omega
   |\vec{v}|^3f
   \,dx
   \leq
   \Cr{const:3.1}^{\frac{p^\ast}{p}}
   \frac{p^\ast}{2p}
  \int_\Omega
  |\nabla\vec{v}|^2f
  \,dx
  +
   \Cr{const:3.1}^{\frac{p^\ast}{p}}
   \left(
   1-
   \frac{p^\ast}{2p}
  \right)
  \left(
   \int_\Omega
   |\vec{v}|^{2}f
   \,dx
  \right)^{\frac{1}{p'}
  \left(
   1-
   \frac{p^\ast}{2p}
  \right)^{-1}
  }.
 \end{equation}

 Now,  we examine the constraints. First, $3\alpha p=p^\ast$, $3\beta
 p'=2$, $\alpha+\beta=1$, and the properties of $p'$, $p^\ast$ imply that,
 \begin{equation}
  \frac{3\beta}{2}=\frac{1}{p'}=1-\frac{3\alpha}{p^\ast}=1-\frac{3\alpha}{2}+\frac{3\alpha}{n},
 \end{equation}
 thus, $\alpha=\frac{n}{6}$. Next, $\frac{p^\ast}{2p}<1$ and $3\alpha
 p=p^\ast$ imply $\alpha<\frac23$. Therefore, we can
 choose $\alpha,\beta,p$ such that \eqref{eq:3.44} is true if $n\leq
 3$. Note that if $n=1,2$ we can take $p^\ast=6$, the same as in the
 case $n=3$. Taking $\alpha=\beta=\frac{1}{2}$ and $p=4$ in
 \eqref{eq:3.47}, the inequality \eqref{eq:3.44} is deduced.
\end{proof}
Using the Sobolev inequality, we obtain the following energy estimate.

\begin{proposition}
 \label{prop:3.14}
 Assume $n=1,2,3$, let $f$ be a solution of \eqref{eq:3.1}, and let
 $\vec{u}$ be given as in \eqref{eq:3.1}. Suppose, that there exists a
 positive constant $\lambda>0$,  such that $\nabla^2\phi\geq \lambda I$,
 where $I$ is the identity matrix. Then, there is a constant
 $\Cr{const:3.6}>0$ such that,  if
 \begin{equation}
  \|\nabla D\|_{L^\infty(\Omega)}\leq \Cr{const:3.6},
 \end{equation}
 then, we have,
 \begin{equation}
  \label{eq:3.28}
  \frac{d^2F}{dt^2}[f](t)
   \geq
   \lambda\int_{\Omega} |\vec{u}|^2f\,dx
   -
   \frac{2\Cr{const:1.3}}{3}
   \left(
    \int_{\Omega} |\vec{u}|^2f\,dx
   \right)^3.
 \end{equation}
\end{proposition}


\begin{proof}
 We estimate the integrands of the 3rd, 4th, 5th, and 7th terms of
 \eqref{eq:3.22}. Using the Cauchy-Schwarz inequality and relation $D(x)\nabla
 (\log D(x))=\nabla D(x)$, for any positive constants
 $\Cl[eps]{eps:3.1}, \Cl[eps]{eps:3.2}>0$, we have that,
 \begin{equation}
  \begin{split}
   |
   (\log f-1)\nabla |\vec{u}|^2\cdot\nabla D(x) f
   |
   &\leq
   2D(x)
   (|\log f|+1)
   |\vec{u}| |\nabla\vec{u}|  |\nabla (\log D(x))| f
   \\
   &\leq
   \frac{1}{2\Cr{eps:3.1}}
   D(x)
   (|\log f|+1)^2
   |\nabla \vec{u}|^2
   |\nabla (\log D(x))|^2 f
   +
   2\Cr{eps:3.1}
   D(x)
   |\vec{u}|^2 f,
  \end{split}
 \end{equation}
 \begin{equation}
  \begin{split}
   |2(1+\log f)\vec{u}\cdot\nabla D(x)\Div\vec{u}f|
   &\leq
   2D(x)
   (|\log f|+1)
   |\vec{u}| |\nabla\vec{u}|  |\nabla (\log D(x))| f
   \\
   &\leq
   \frac{1}{2\Cr{eps:3.2}}
   D(x)
   (|\log f|+1)^2
   |\nabla \vec{u}|^2
   |\nabla (\log D(x))|^2 f
   +
   2\Cr{eps:3.2}
   D(x)
   |\vec{u}|^2 f,
  \end{split}
 \end{equation}
%
 and
 \begin{equation}
  \left|
   \frac{2}{D(x)}
   \log f\left(\vec{u}\cdot\nabla D(x)\right)\left(\vec{u}\cdot\nabla \phi(x)\right) f
  \right|
  \leq
  2|\log f|
  |\nabla (\log D(x))|
  |\nabla \phi(x)|
  |\vec{u}|^2f.
 \end{equation}
 Thus, using the above inequalities in \eqref{eq:3.22}, we arrive at
 the estimate for $ \frac{d^2F}{dt^2}[f](t)$,
 \begin{equation}
  \label{eq:3.49}
  \begin{split}
   \frac{d^2F}{dt^2}[f](t)
   &\geq
   2\int_\Omega
   ((\nabla^2\phi(x)
   -
   (\Cr{eps:3.1}+\Cr{eps:3.2})D(x)I
   -
   |\log f|
   |\nabla (\log D(x))|
   |\nabla \phi(x)|I) \vec{u}\cdot\vec{u}
   ) f\,dx
   \\
   &\qquad
   +
   2\int_\Omega
   \left(
   1
   -
   \frac14
   \left(
   \frac{1}{\Cr{eps:3.1}}
   +
   \frac{1}{\Cr{eps:3.2}}
   \right)
   (|\log f|+1)^2
   |\nabla (\log D(x))|^2
   \right)
   D(x)|\nabla \vec{u}|^2 f\,dx
   \\
   &\qquad
   +
   2
   \int_\Omega
   \frac{1}{D(x)}|\vec{u}|^2
   \log f\left(\vec{u}\cdot\nabla D(x)\right) f\,dx
   +
   2
   \int_\Omega
   \frac{1}{D(x)}
   (\log f)^2\left(\vec{u}\cdot\nabla D(x)\right)^2 f\,dx.
  \end{split}
 \end{equation}

 Next, using \eqref{eq:3.44} and
 $D(x)/\Cr{const:1.3}\geq1$, we have that,
 \begin{equation}
  \label{eq:3.50}
  \begin{split}
   &\qquad
   \left|
   2
   \int_\Omega
   \frac{1}{D(x)}|\vec{u}|^2
   \log f\left(\vec{u}\cdot\nabla D(x)\right) f\,dx
   \right| \\
   &\leq
   2
   \|\log f\|_{L^\infty(\Omega\times[0,\infty))}
   \|\nabla \log D\|_{L^\infty(\Omega)}
   \int_\Omega|\vec{u}|^3f\,dx \\
   &\leq
   \frac{3\Cr{const:3.1}^{\frac{3}{2}}}{2}
   \|\log f\|_{L^\infty(\Omega\times[0,\infty))}
   \|\nabla \log D\|_{L^\infty(\Omega)}
   \int_\Omega|\nabla\vec{u}|^2f\,dx
   \\
   &\qquad
   +
   \frac{\Cr{const:3.1}^{\frac{3}{2}}}{2}
   \|\log f\|_{L^\infty(\Omega\times[0,\infty))}
   \|\nabla \log D\|_{L^\infty(\Omega)}
   \left(
   \int_\Omega |\vec{u}|^2f\,dx
   \right)^3
   \\
   &\leq
   \frac{3\Cr{const:3.1}^{\frac{3}{2}}}{2\Cr{const:1.3}}
   \|\log f\|_{L^\infty(\Omega\times[0,\infty))}
   \|\nabla \log D\|_{L^\infty(\Omega)}
   \int_\Omega D(x)|\nabla\vec{u}|^2f\,dx
   \\
   &\qquad
   +
   \frac{\Cr{const:3.1}^{\frac{3}{2}}}{2}
   \|\log f\|_{L^\infty(\Omega\times[0,\infty))}
   \|\nabla \log D\|_{L^\infty(\Omega)}
   \left(
   \int_\Omega|\vec{u}|^2f\,dx
   \right)^3.
  \end{split}
 \end{equation}
Therefore,  from \eqref{eq:3.50} and \eqref{eq:3.49}, we obtain that,
\begin{equation}
 \label{eq:3.57}
  \begin{split}
   \frac{d^2F}{dt^2}[f](t)
   &\geq
   2\int_\Omega
   ((\nabla^2\phi(x)
   -
   (\Cr{eps:3.1}+\Cr{eps:3.2})D(x)I
   -
   |\log f|
   |\nabla (\log D(x))|
   |\nabla \phi(x)|I) \vec{u}\cdot\vec{u}
   ) f\,dx
   \\
   &\qquad
   +
   2\int_\Omega
   \left(
   1
   -
   \frac14
   \left(
   \frac{1}{\Cr{eps:3.1}}
   +
   \frac{1}{\Cr{eps:3.2}}
   \right)
   (|\log f|+1)^2
   |\nabla (\log D(x))|^2
   \right)
   D(x)|\nabla \vec{u}|^2 f\,dx
   \\
   &\qquad
   -
   \frac{3\Cr{const:3.1}^{\frac{3}{2}}}{2\Cr{const:1.3}}
   \|\log f\|_{L^\infty(\Omega\times[0,\infty))}
   \|\nabla \log D\|_{L^\infty(\Omega)}
   \int_\Omega D(x)|\nabla\vec{u}|^2f\,dx
   \\
   &\qquad
   -
   \frac{\Cr{const:3.1}^{\frac{3}{2}}}{2}
   \|\log f\|_{L^\infty(\Omega\times[0,\infty))}
   \|\nabla \log D\|_{L^\infty(\Omega)}
   \left(
   \int_\Omega|\vec{u}|^2f\,dx
   \right)^3.
  \end{split}
 \end{equation}

 By the maximum principle Proposition \ref{prop:1.5}, there is a
 positive constant $\Cl{const:3.9}$ which depends only on $f_0$,
 $f^{\mathrm{eq}}$, and $D$,  such that $|\log f(x,t)|\leq \Cr{const:3.9}$
 for $x\in\Omega$ and $t>0$. If $\|\nabla D\|_{L^\infty(\Omega)}\leq
 \Cr{const:3.6}$, then $|\nabla\log D(x,t)|\leq
 \Cr{const:3.6}/\Cr{const:1.3}$,  hence we have that,
 \begin{equation}
  (\Cr{eps:3.1}+\Cr{eps:3.2})D(x)
   +
   |\log f|
   |\nabla (\log D(x))|
   |\nabla \phi(x)|
   \leq
   (\Cr{eps:3.1}+\Cr{eps:3.2})\|D\|_{L^\infty(\Omega)}
   +
   \frac{\Cr{const:3.6}\Cr{const:3.9}}{\Cr{const:1.3}}
   \|\nabla \phi\|_{L^\infty(\Omega)},
 \end{equation}
 and
 \begin{multline}
  \frac14
   \left(
   \frac{1}{\Cr{eps:3.1}}
   +
   \frac{1}{\Cr{eps:3.2}}
   \right)
   (|\log f|+1)^2
   |\nabla (\log D(x))|^2
   +
   \frac{3\Cr{const:3.1}^{\frac{3}{2}}}{4\Cr{const:1.3}}
   \|\log f\|_{L^\infty(\Omega\times[0,\infty))}
   \|\nabla \log D\|_{L^\infty(\Omega)}
   \\
  \leq
   \frac14
   \left(
   \frac{1}{\Cr{eps:3.1}}
   +
   \frac{1}{\Cr{eps:3.2}}
   \right)
   \frac{\Cr{const:3.6}^2(\Cr{const:3.9}+1)^2}{\Cr{const:1.3}^2}
   +
   \frac{3\Cr{const:3.1}^{\frac{3}{2}}\Cr{const:3.6}\Cr{const:3.9}}{4\Cr{const:1.3}^2}.
 \end{multline}
 Thus, first take small $\Cr{eps:3.1}$, $\Cr{eps:3.2}>0$,  and next
 take $\Cr{const:3.6}>0$ such that,
 \begin{equation}
  \label{eq:3.26}
 (\Cr{eps:3.1}+\Cr{eps:3.2})D(x)
   +
   |\log f|
   |\nabla (\log D(x))|
   |\nabla \phi(x)|
   \leq\frac{\lambda}{2}, 
 \end{equation}
 and
 \begin{equation}
  \label{eq:3.27}
  \frac14
   \left(
   \frac{1}{\Cr{eps:3.1}}
   +
   \frac{1}{\Cr{eps:3.2}}
   \right)
   (|\log f|+1)^2
   |\nabla (\log D(x))|^2
   +
   \frac{3\Cr{const:3.1}^{\frac{3}{2}}}{4\Cr{const:1.3}}
   \|\log f\|_{L^\infty(\Omega\times[0,\infty))}
   \|\nabla \log D\|_{L^\infty(\Omega)}
  \leq
  1.
 \end{equation}
 Note that,  $\frac{3\Cr{const:3.1}^{\frac{3}{2}}}{4\Cr{const:1.3}} \|\log
 f\|_{L^\infty(\Omega\times[0,\infty))} \|\nabla \log
 D\|_{L^\infty(\Omega)} \leq 1$,  hence we obtain that,
 \begin{equation}
  \label{eq:3.52}
  -
   \frac{\Cr{const:3.1}^{\frac{3}{2}}}{2}
   \|\log f\|_{L^\infty(\Omega\times[0,\infty))}
   \|\nabla \log D\|_{L^\infty(\Omega)}
   \left(
    \int_\Omega|\vec{u}|^2f\,dx
   \right)^3
   \geq
   -\frac{2\Cr{const:1.3}}{3}
   \left(
    \int_\Omega|\vec{u}|^2f\,dx
   \right)^3.
 \end{equation}
 Employing \eqref{eq:3.26}, \eqref{eq:3.27} and 
 \eqref{eq:3.52} in the estimate \eqref{eq:3.57}, the desired energy bound
 \eqref{eq:3.28} is deduced.
\end{proof}
The energy estimate \eqref{eq:3.28} becomes,
\begin{equation}
  \frac{d^2F}{dt^2}[f](t)
   \geq
   -\lambda\frac{dF}{dt}[f](t)
   +
   \frac{2\Cr{const:1.3}}{3}
   \left(
    \frac{dF}{dt}[f](t)
   \right)^3.
\end{equation}

Next, to proceed with a proof of the result in Theorem~\ref{thm:3.13},
  below we first provide a helpful version of the Gronwall's inequality.

\begin{lemma}
 \label{lem:3.12}
 Let $c,d,p>0$ be positive constants, such that $p>1$. Let
 $g:[0,\infty)\rightarrow\R$ be a non-negative function which satisfies the
 following differential inequality,
 \begin{equation}
  \label{eq:3.53}
  \frac{dg}{dt}\leq -cg+dg^p.
 \end{equation}
 If
 \begin{equation}
  g(0)< \left(
	 \frac{c}{d}
	\right)^{\frac{1}{p-1}},
 \end{equation}
 then,  we obtain for $t>0$,
 \begin{equation}
  \label{eq:3.54}
  g(t)
   \leq
   \left(
    g(0)^{-p+1}-\frac{d}{c}
   \right)^{-\frac{1}{p-1}}
   e^{-ct}.
 \end{equation}
\end{lemma}

\begin{proof}
 First, multiply both sides of \eqref{eq:3.53} by $e^{ct}$, then we
 have that,
 \begin{equation}
  \frac{d}{dt}(e^{ct}g)\leq de^{ct}g^p=d(e^{ct}g)^pe^{-c(p-1)t}.
 \end{equation}
 Set $G:=e^{ct}g$. Then,  $G^{-p}\frac{dG}{dt}\leq de^{-c(p-1)t}$,  hence
 for $t>0$,
 \begin{equation}
  -\frac{1}{p-1}
   \left(
    G(t)^{-p+1}
    -
    G(0)^{-p+1}
   \right)
   \leq
   \int_0^t
   de^{-c(p-1)\tau}
   \,d\tau
   =
   \frac{d}{c(p-1)}
   \left(
    1
    -e^{-c(p-1)t}
   \right)
   \leq
   \frac{d}{c(p-1)}.
 \end{equation}
 Thus, straightforward computation shows that,
 \begin{equation}
  G(t)
   \leq
   \left(
    G(0)^{-p+1}
    -
    \frac{d}{c}
   \right)^{-\frac{1}{p-1}}.
 \end{equation}
 Since $G(0)=g(0)$, we obtain the estimate \eqref{eq:3.54}.
\end{proof}

Finally,  we are in position to conclude the proof of our main
Theorem \ref{thm:3.13} in this Section, similar to the
presented homogeneous case in Theorem \ref{thm:2.9} in Section~\ref{sec:2}.

\begin{proof}%
 [Proof of Theorem \ref{thm:3.13}]
 From the differential inequality \eqref{eq:3.28}, we have that,
 \begin{equation}
  \frac{d}{dt}
   \left(
    \int_\Omega |\vec{u}|^2f\,dx
   \right)
   \leq
   -\lambda
   \int_\Omega |\vec{u}|^2f\,dx
   +
   \frac{2\Cr{const:1.3}}{3}
   \left(
    \int_\Omega |\vec{u}|^2f\,dx
   \right)^3.
 \end{equation}
 Using Lemma~\ref{lem:3.12}, the Gronwall-type inequality, there is a
 constant $\Cr{const:3.7}>0$,  such that,  if $\int_\Omega
 |\vec{u}|^2f\,dx|_{t=0}\leq \Cr{const:3.7}$, that is,
 $ \int_\Omega |\nabla(D(x)\log f_0+\phi(x))|^2f_0\,dx
   \leq\Cr{const:3.7}$,
 we obtain the desired result \eqref{eq:3.55}.
\end{proof}

\begin{remark}
In comparison with the homogeneous case in Section~\ref{sec:2}, it is
not known how to use the weighted
 $L^2$ space for the inhomogeneous problem \eqref{eq:3.1}. The
 difficulty here arises from the nonlinearity \eqref{eq:1.9}. We also
 do not know the logarithmic Sobolev inequality related to the inhomogeneous
 problem \eqref{eq:3.1},  and it is not known of how to establish the full convergence of the free energy like
was done  in \eqref{eq:A.6} and \eqref{eq:A.9}.
\end{remark}

\begin{remark}
 Here,  we want to note about the space dimension
 $n$. In this and the following section, since $D$ is not constant and
 we use the Sobolev inequality for the general vector valued function
 $\vec{u}$, (namely, did not use the fact that $\vec{u}$ is constituted of
 the solution $f$), we can only treat the dimensions $n=1,2,3$. While,
 for the dimensions $n\geq 4$, the bound \eqref{eq:3.44} does not hold
 for general vector-valued function $\vec{u}$, since
 $\alpha=\frac{n}{6}$ will be greater than $\frac{2}{3}$ in the proof of
 Lemma \ref{lem:3.12} in that case. The case $n=4$ is critical, since
 $\alpha=\frac{2}{3}$, while the case $n\geq5$ is supercritical, since
 $\alpha>\frac{2}{3}$. If we can obtain additional regularity estimates
 for $f$, such as uniform bounds for $\nabla f$, we might be able to
 treat the higher dimensional case, $n\geq 4$, which is ongoing work.
\end{remark}

The next section extends the entropy method to the nonlinear
Fokker-Planck model. A key idea is to demonstrate the entropy method in terms
of the velocity field $\vec{u}$ (the entropy method will not work if
applied directly to the solution $f$ of the model).

\section{Inhomogeneous diffusion case with variable mobility}
\label{sec:4}

Finally, in this section, we will consider  the following general evolution equation
with both inhomogeneous diffusion and a variable mobility, that is, $D=D(x)$
and $\pi=\pi(x,t)$ being both positive and bounded in a bounded domain
in the Euclidean space of $n$-dimension, subject to the periodic boundary
condition,

\begin{equation}
 \label{eq:4.1}
 \left\{
  \begin{aligned}
   \frac{\partial f}{\partial t}
   +
   \Div
   \left(
   f\vec{u}
   \right)
   &=
   0,
   &\quad
   &x\in\Omega,\quad
   t>0, \\
   \vec{u}
   &=
   -
   \frac{1}{\pi(x,t)}
   \nabla
   \left(
   D(x)\log f
   +
   \phi(x)
   \right),
   &\quad
   &x\in\Omega,\quad
   t>0, \\   
   f(x,0)&=f_0(x),&\quad
   &x\in\Omega.
  \end{aligned}
 \right.
\end{equation}
Again, without loss of generality, we take $\Omega=[0,1)^n\subset\R^n$.
The strictly positive periodic functions $\pi(x,t)$ and $D(x)$ are
bounded from below with the constants, $\Cr{const:1.2},
\Cr{const:1.3}>0$,
\begin{equation}
 \pi(x,t)\geq \Cr{const:1.2},\quad
  D(x)\geq\Cr{const:1.3}
\end{equation}
for any $x\in\Omega$ and $t>0$.

The free energy $F$ and the basic energy law \eqref{eq:1.3} still take
similar form
in this case, namely,
\begin{equation}
 \label{eq:4.2}
  F[f]
  :=
  \int_\Omega
  \left(
   D(x)f(\log f -1)+f\phi(x)
  \right)
  \,dx.
\end{equation}
and
\begin{equation}
 \label{eq:4.3}
  \frac{dF}{dt}[f](t)
  =
  -\int_\Omega
  \pi(x,t)
  |\vec{u}|^2f\,dx
  =:
  - D_{\mathrm{dis}}[f](t).
\end{equation}

As in the case of the constant mobility, Section \ref{sec:3}, we first
notice the following Sobolev inequality,  with weight being the solution
of the above general system \eqref{eq:4.1}.

\begin{lemma}
 \label{lem:4.2}
 Let $f$ be a solution of the model \eqref{eq:4.1}. For a suitable
 positive constant $\Cl{const:4.3}>0$, such that for any $t>0$ and for
 any periodic vector field $\vec{v}$ on $\Omega$,
 \begin{equation}
  \label{eq:4.48}
   \left(
    \int_\Omega|\vec{v}|^{p^\ast}f\,dx
   \right)^{\frac{1}{p^\ast}}
   \leq
   \Cr{const:4.3}
   \left(
    \int_\Omega|\nabla\vec{v}|^{2}f\,dx
	   \right)^{\frac{1}{2}},
 \end{equation}
 where the exponent $p^\ast$ satisfies
 $\frac{1}{p^\ast}=\frac{1}{2}-\frac{1}{n} $ for $n=3$,  and arbitrary
 $2\leq p^{\ast}<\infty $ for $n=1,2$. 
 
 In particular, with this Sobolev-type inequality
 \eqref{eq:4.48} and the H\"older inequality, we have for $2\leq p\leq
 p^\ast$ that,
 \begin{equation}
  \int_\Omega|\vec{v}|^p f\,dx
  \leq
  \left(
   \int_\Omega|\vec{v}|^{p^\ast} f\,dx
  \right)^{\frac{p}{p^\ast}}
  \left(
   \int_\Omega
   f\,dx
  \right)^{1-\frac{p}{p^\ast}}
  \leq
  \Cr{const:4.3}^p
  \left(
   \int_\Omega|\nabla\vec{v}|^{2}f\,dx
  \right)^{\frac{p}{2}}
  \left(
   \int_\Omega
   f\,dx
  \right)^{1-\frac{p}{p^\ast}}.
 \end{equation}
\end{lemma}

The proof of Lemma \ref{lem:4.2} follows exactly the same argument as
the proof of Lemma \ref{lemma3.1} in Section~\ref{sec:3}. 

The main Theorem~\ref{thm:4.1} of this section is the extension of the results in
Section~\ref{sec:3}, Theorem~\ref{thm:3.13}, when $\pi$ was constant,
(in particular, $\|\nabla \pi\|_{L^\infty(\Omega\times[0,\infty))}=\|
\pi_t\|_{L^\infty(\Omega\times[0,\infty))}=0$), to the case of the
variable mobility $\pi(x,t)$. To be more specific, when $\|\nabla
D\|_{L^\infty(\Omega)}$ and $\|\nabla
\pi\|_{L^\infty(\Omega\times[0,\infty))}$ are sufficiently small, and
under some additional assumptions on the initial condition, one can
establish the exponential decay of the dissipation functional
$D_{\mathrm{dis}}[f](t)$ using the basic energy law \eqref{eq:4.3}.
\begin{theorem}
 \label{thm:4.1} 
 Consider $\Omega$ being the unit box in  the Euclidean space of $n$-dimension with $n=1,2,3$. 
 Assume, that there is a positive constant $\lambda>0$,
 such that $\nabla^2\phi\geq \lambda I$, where $I$ is the identity
 matrix. Moreover, let $\phi=\phi(x)$,
 $D=D(x)$ and $\pi=\pi(x,t)$ be periodic functions which satisfy
 \eqref{eq:1.12}, and  let $f_0=f_0(x)$ be  a periodic probability density
 function.  Consider a solution $f$  of \eqref{eq:4.1} subject to the
 periodic boundary condition, and  vector field $\vec{u}$ which is defined in
 \eqref{eq:4.1}.  Then, there are positive constants $\Cl{const:4.4}$,
 $\Cl{const:4.5}$, $\Cl{const:4.9}$, $\Cl{const:4.6}$,
$\Cl{const:4.7}>0$, and $\tilde{\lambda}>0$ such that, if for  $x\in\Omega$ and $t>0$,
 \begin{equation}
  \label{cond_par:thm4}
  \|\nabla D\|_{L^\infty(\Omega)}\leq \Cr{const:4.4},\quad
   \|\nabla \pi\|_{L^\infty(\Omega\times[0,\infty))}\leq \Cr{const:4.5},\quad
   \pi_t(x,t)\geq -\Cr{const:4.9}, 
 \end{equation}
 and
 \begin{equation}
  \int_\Omega \pi(x,0)|\vec{u}(x,0)|^2f_0\,dx
   =
  \int_\Omega \frac{1}{\pi(x,0)}|\nabla(D(x)\log f_0+\phi(x))|^2f_0\,dx
   \leq\Cr{const:4.6},
 \end{equation}
 then,  the following estimate holds true, that is, for $t>0$,
 \begin{equation}
  \label{eq:4.64}
   \int_\Omega
   \pi(x,t)
   |\vec{u}|^2f\,dx
   \leq
   \Cr{const:4.7}
   e^{-\tilde{\lambda} t}.
 \end{equation}
 In particular, we have that,
 \begin{equation}
  \frac{dF}{dt}[f](t)
   =
   -\int_\Omega
   \pi(x,t)
   |\vec{u}|^2f\,dx
   \rightarrow 0,
   \qquad
   \text{as}
   \
   t\rightarrow\infty.
 \end{equation}
\end{theorem}

With respect to the result in Theorem \ref{thm:4.1}, we first remark
that there is a subsequence $\{t_j\}_{j=1}^\infty$ such that the
following lemma is true.

\begin{lemma}
 \label{lem:4.1}
 Let $f$ be a solution of \eqref{eq:4.1}. Then there is an increasing
 sequence $\{t_j\}_{j=1}^\infty$,  such that $t_j\rightarrow\infty$ and
 \begin{equation}
  \frac{dF}{dt}[f](t_j)\rightarrow 0,\qquad
   j\rightarrow\infty.
 \end{equation}
\end{lemma}

The proof of Lemma \ref{lem:4.1} follows exactly the same argument as
the proof of Lemma \ref{lem:2.1} in Section~\ref{sec:2}. 
\par In order to establish statement of Theorem \ref{thm:4.1}, first, we need to
obtain additional results as in
Lemmas~\ref{lem:4.4a}-\ref{lem:4.15a} and Proposition~\ref{prop:4.16a}
below. Hence, we proceed to show that $\frac{dF}{dt}[f]$ converges to $0$ as
$t\rightarrow\infty$ in time $t$.  Hereafter we compute the second time
derivative of $F$, and in particular, we utilize the special structure of
the velocity field $\vec{u}$. To do this, we first establish the following
relationships between $\nabla f$ and $\vec{u}$ by direct calculation of
the velocity $\vec{u}$.

\begin{lemma}
\label{lem:4.4a}
 Let $\vec{u}$ be defined by \eqref{eq:4.1}. Then,
 \begin{equation}
  \label{eq:4.8}
   \pi(x,t)f\vec{u}
   =
   -
   D(x)\nabla f
   -
   f\log f\nabla D(x)
   -
   f\nabla\phi(x),
 \end{equation}
 and
 \begin{equation}
  \label{eq:4.6}
   \pi(x,t)\rho\vec{u}
   =
   -
   D(x)\nabla \rho
   -
   \rho\log\rho\nabla D(x),
 \end{equation}
 where $\rho$ is defined in \eqref{eq:1.8}.
\end{lemma}

\begin{remark}
Recall, that the nonlinearity in \eqref{eq:4.8} is from the
 inhomogeneity of $D(x)$ and it takes the following form \eqref{eq:1.9} in the model,
\begin{equation*}
 N(f) = - 
  \frac{1}{\pi(x,t)}\log f\nabla D(x)\cdot\nabla f
  +
  \frac{\nabla \pi(x,t)\cdot\nabla D(x)}{\pi^2(x,t)}f\log f
  -
  \frac{1}{\pi(x,t)}\Delta D(x) f\log f. 
\end{equation*}
\end{remark}
We proceed again with the entropy method and we take the second in-time derivative
of the free energy $F$,

\begin{equation*}
 \begin{split}
  \frac{d^2F}{dt^2}[f]
  &=
  \frac{d}{dt}
  \left(
   -
  \int_\Omega
  \pi(x,t)
  |\vec{u}|^2f\,dx
  \right)
  \\
  &=
  -
  2
  \int_\Omega
  \pi(x,t)\vec{u}\cdot\vec{u}_t f\,dx
  -
  \int_\Omega
  \pi(x,t)|\vec{u}|^2f_t\,dx
  -
  \int_\Omega
  \pi_t(x,t)|\vec{u}|^2f\,dx.
 \end{split}
\end{equation*}

As in Sections~\ref{sec:2} and \ref{sec:3}, we first compute the time
derivative of the velocity $\vec{u}$.

\begin{lemma}
 Let $\vec{u}$ be defined as in \eqref{eq:4.1}. Then,
 \begin{equation}
  \label{eq:4.4}
   \pi(x,t)\vec{u}_t
   =
   -
   \frac{D(x)}{\rho}\nabla \rho_t
   -
   \frac{\rho\pi_t(x,t)+\rho_t\pi(x,t)}{\rho}\vec{u}
   -
   \frac{\rho_t(\log\rho+1)}{\rho}
   \nabla D(x).
 \end{equation}
\end{lemma}

\begin{proof}
 We take a time-derivative of $\pi(x,t)\vec{u}=-\nabla (D(x)\log\rho)$,
 and we have from \eqref{eq:4.1} that,
 \begin{equation}
  \label{eq:4.5}
   \pi_t(x,t)\vec{u}
   +
   \pi(x,t)\vec{u}_t
   =
   -\nabla \left(
	    D(x)\frac{\rho_t}{\rho}
	   \right)
   =
   -
   \frac{D(x)}{\rho}\nabla \rho_t
   +
   \frac{D(x)\rho_t}{\rho^2}\nabla \rho
   -
   \frac{\rho_t}{\rho}\nabla D(x).
 \end{equation}
 Using \eqref{eq:4.6} in \eqref{eq:4.5}, we obtain the result
 \eqref{eq:4.4}.
\end{proof}

By comparing formula in \eqref{eq:4.4} with the formula in
\eqref{eq:2.4} and \eqref{eq:3.4} in Sections~\ref{sec:2}-\ref{sec:3},
one can observe that the extra terms
$\frac{\rho_t}{\rho}(\log\rho+1)\nabla D(x)$ and $-\pi_t(x,t)\vec{u}$
appear in the time derivative of $\vec{u}$ due to the inhomogeneity of
the diffusion and
the variable mobility.

Again,  we will reformulate $\vec{u}_t$ in terms of $\rho_t$ and $f_t$
in the second
time-derivative of $F$.

\begin{lemma}
 Let $f$ be a solution of \eqref{eq:4.1},  and let $\vec{u}$ be given by
 \eqref{eq:4.1}. Then,
 \begin{equation}
  \label{eq:4.7}
   \begin{split}
    \frac{d^2F}{dt^2}[f](t)
    &=
    \int_\Omega\pi_t(x,t)|\vec{u}|^2f\,dx
    +
    \int_\Omega\pi(x,t)|\vec{u}|^2f_t\,dx
    \\
    &\qquad
    +
    2\int_\Omega D(x)\vec{u}\cdot\nabla\rho_t\Feq\,dx
    +
    2\int_\Omega(\log\rho+1)\vec{u}\cdot\nabla D(x)f_t\,dx,
   \end{split}
 \end{equation}
 where $\Feq$ is given in \eqref{eq:1.5}.
\end{lemma}

\begin{proof}
 Using the time-derivative of \eqref{eq:3.3} together
 with \eqref{eq:4.4}, we obtain that,
\begin{equation}
  \begin{split}
   \frac{d^2F}{dt^2}[f](t)
   &=
   -
   \int_\Omega\pi_t(x,t)|\vec{u}|^2 f\,dx
   -
   2\int_\Omega\pi(x,t)\vec{u}\cdot\vec{u}_t f\,dx
   -
   \int_\Omega\pi(x,t)|\vec{u}|^2f_t\,dx \\
   &=
   \int_\Omega\pi_t(x,t)|\vec{u}|^2 f\,dx
   +
   2\int_\Omega D(x)\vec{u}\cdot\nabla\rho_t \frac{f}{\rho}\,dx
   +
   2\int_\Omega\pi(x,t)|\vec{u}|^2\rho_t \frac{f}{\rho}\,dx
   \\
   &\qquad
   +
   2\int_\Omega(\log\rho+1)\vec{u}\cdot\nabla D(x)\rho_t \frac{f}{\rho}\,dx
   -
   \int_\Omega\pi(x,t)|\vec{u}|^2f_t\,dx.
  \end{split}
 \end{equation}
 Since $f=\rho\Feq$ and $\rho_t\Feq=f_t$, we derive \eqref{eq:4.7}.
\end{proof}

Next, we compute each term in the right-hand side of
\eqref{eq:4.7}. First, for the second term of the right-hand side of
\eqref{eq:4.7}, we obtain,

\begin{lemma}
 Let $f$ be a solution of \eqref{eq:4.1} and let $\vec{u}$ be given by
 \eqref{eq:4.1}. Then,
 \begin{equation}
  \label{eq:4.12}
   \int_\Omega \pi(x,t)|\vec{u}|^2f_t\,dx
   =
   \int_\Omega \vec{u}\cdot\nabla (\pi(x,t)|\vec{u}|^2) f\,dx.
 \end{equation}
\end{lemma}

\begin{proof}
 Using the system \eqref{eq:4.1} and integration by parts, we have that,
 \begin{equation}
  \int_\Omega \pi(x,t)|\vec{u}|^2f_t\,dx
   =
   -\int_\Omega \pi(x,t)|\vec{u}|^2\Div(f\vec{u})\,dx
   =
   \int_\Omega \vec{u}\cdot\nabla (\pi(x,t)|\vec{u}|^2)f\,dx.
 \end{equation}
\end{proof}

Next, we express $\nabla\rho_t$ in terms of $\vec{u}$ in order to compute
the first term of the right-hand side of \eqref{eq:4.7}.

\begin{lemma}
 Let $f$ be a solution of \eqref{eq:4.1} and let $\vec{u}$ be given by
 \eqref{eq:4.1}. Then,
 \begin{equation}
  \label{eq:4.11}
   \begin{split}
    &\qquad
    D(x)\Feq\nabla\rho_t \\
    &=
    -\pi(x,t)f_t\vec{u}
    -f_t
    \left(
    1+\log\rho
    \right)
    \nabla D(x)
    +
    f
    \nabla
    \left(
    \pi(x,t)
    |\vec{u}|^2
    +
    \log f
    \vec{u}\cdot\nabla D(x)
    +
    \vec{u}\cdot\nabla\phi(x)
    -
    D(x)\Div\vec{u}
    \right),
   \end{split}
 \end{equation}
 where $\Feq$ is given in \eqref{eq:1.5}.
\end{lemma}

\begin{proof}
 Since $\Feq$ is independent of $t$, we have by \eqref{eq:4.1} that
 \begin{equation}
  \label{eq:4.9}
   D(x)\Feq\rho_t
   =
   D(x)f_t
   =
   -D(x)\Div(f\vec{u})
   =
   -
   D(x)\nabla f\cdot \vec{u}
   -
   D(x)f\Div\vec{u}.
 \end{equation}
 Using \eqref{eq:4.8} and \eqref{eq:4.9}, we obtain,
 \begin{equation}
  \label{eq:4.10}
   D(x)\Feq\rho_t
   =
   D(x)f_t
   =
   f
   \left(
    \pi(x,t)
    |\vec{u}|^2
    +
    \log f
    \vec{u}\cdot\nabla D(x)
    +
    \vec{u}\cdot\nabla\phi(x)
    -
    D(x)\Div\vec{u}
   \right).
 \end{equation}
 Next, take a gradient of \eqref{eq:4.10},  and we obtain using
 \eqref{eq:4.10} that,
 \begin{equation}
  \label{eq:4.42}
  \begin{split}
   &\qquad
   \rho_t\Feq\nabla D(x)
   +
   D(x)\rho_t\nabla\Feq
   +
   D(x)\Feq\nabla\rho_t \\
   &=
   \left(
   \pi(x,t)|\vec{u}|^2
   +
   \log f
   \vec{u}\cdot\nabla D(x)
   +
   \vec{u}\cdot\nabla\phi(x)
   -
   D(x)\Div\vec{u}
   \right)
   \nabla f \\
   &\qquad
   +
   f
   \nabla
   \left(
   \pi(x,t)|\vec{u}|^2
   +
   \log f
   \vec{u}\cdot\nabla D(x)
   +
   \vec{u}\cdot\nabla\phi(x)
   -
   D(x)\Div\vec{u}
   \right) \\
   &=
   \frac{D(x)f_t}{f}
   \nabla f
   +
   f
   \nabla
   \left(
   \pi(x,t)
   |\vec{u}|^2
   +
   \log f
   \vec{u}\cdot\nabla D(x)
   +
   \vec{u}\cdot\nabla\phi(x)
   -
   D(x)\Div\vec{u}
   \right).
  \end{split}
 \end{equation}
 Now, taking a gradient of \eqref{eq:1.6}, we have,
 \begin{equation}
  \label{eq:4.41}
  \frac{D(x)}{\Feq}\nabla \Feq+\log\Feq\nabla D(x)+\nabla\phi(x)=0.
 \end{equation}
 Thus, using \eqref{eq:4.8} and \eqref{eq:4.41} in \eqref{eq:4.42}, we
 have,
 \begin{equation}
  \begin{split}
   &\qquad
   \rho_t\Feq\nabla D(x)
   -
   \rho_t
   \Feq
   \log\Feq\nabla D(x)
   -
   \rho_t
   \Feq
   \nabla \phi(x)
   +
   D(x)\Feq\nabla\rho_t \\
   &=
   -
   \pi(x,t)f_t\vec{u}
   -
   f_t\log f\nabla D(x)
   -
   f_t\nabla\phi(x)
   \\
   &\qquad
   +
   f
   \nabla
   \left(
   \pi(x,t)
   |\vec{u}|^2
   +
   \log f
   \vec{u}\cdot\nabla D(x)
   +
   \vec{u}\cdot\nabla\phi(x)
   -
   D(x)\Div\vec{u}
   \right).
  \end{split}
 \end{equation}
 Since $\rho_t\Feq=f_t$, we obtain \eqref{eq:4.11}.
\end{proof}

%
Note, using \eqref{eq:4.11} in the 3rd term of the right-hand side of
\eqref{eq:4.7}, we have,
\begin{equation}
 \label{eq:4.57}
\begin{split}
  2\int_\Omega D(x)\vec{u}\cdot\nabla\rho_t\Feq\,dx
  &=
  -
  2\int_\Omega
  \pi(x,t)f_t|\vec{u}|^2
  \,dx
  -
  2\int_\Omega
  f_t
  \left(
  1+\log\rho
  \right)
  \vec{u}
  \cdot
  \nabla D(x)
  \,dx
  \\
  &\quad
  +
  2\int_\Omega
  f
  \vec{u}
  \cdot
  \nabla
  (\pi(x,t)
  |\vec{u}|^2)
  \,dx
  +
  2\int_\Omega
  f
  \vec{u}
  \cdot
  \nabla
  (
  \log f
  \vec{u}\cdot\nabla D(x)
  )
  \,dx
  \\
  &\quad
  +
  2\int_\Omega
  f
  \vec{u}
  \cdot
  \nabla
  (
  \vec{u}\cdot\nabla\phi(x)
  )
  \,dx
  -
  2\int_\Omega
  f
  \vec{u}
  \cdot
  \nabla
  (
  D(x)\Div\vec{u}
  )
  \,dx.
 \end{split}
\end{equation}
%
Unlike Section \ref{sec:2} or \ref{sec:3}, the velocity field $\vec{u}$ do
not have a scalar potential in general. This yields that $\nabla
\vec{u}$ is not symmetric any more so the relations \eqref{eq:2.16},
\eqref{eq:2.15}, \eqref{eq:2.22} and \eqref{eq:2.18} do not hold. To
overcome this difficulty, we give the following commutator relation
between $\nabla\vec{u}$ and its transpose $\mathstrut^T\nabla\vec{u}$.

\begin{lemma}
 Let $\vec{u}$ be defined by \eqref{eq:4.1}. Then,
 \begin{equation}
  \label{eq:4.14}
  \nabla\vec{u} -\mathstrut^T\nabla \vec{u}
   =
   \frac{1}{\pi(x,t)}
   (\nabla\pi(x,t)\otimes\vec{u}-\vec{u}\otimes\nabla\pi(x,t)).
 \end{equation}
\end{lemma}

\begin{proof}
 We denote $\vec{u}=(u^k)_{k}$. Since
 $\pi(x,t)u^k=-(D(x)\log\rho)_{x_k}$ for $k=1,2,\ldots,n$, by taking a
 derivative with respect to $x_l$, we have,
 \begin{equation}
  \pi_{x_l}(x,t)u^k
   +
   \pi(x,t)u^{k}_{x_l}
   =-(D(x)\log\rho)_{x_kx_l}.
 \end{equation}
 Thus,
 \begin{equation}
  \pi(x,t)u^{k}_{x_l}
   -
   \pi(x,t)u^{l}_{x_k}
   =
   -
   \pi_{x_l}(x,t)u^k
   +
   \pi_{x_k}(x,t)u^l
   =
   (
   \nabla\pi(x,t)\otimes\vec{u}
   -
   \vec{u}\otimes\nabla\pi(x,t))_{k,l},
 \end{equation}
 hence this yields \eqref{eq:4.14}.
\end{proof}

Next, we compute $\vec{u}\cdot\nabla(\vec{u}\cdot\nabla\phi(x))$ and
$\vec{u}\cdot\nabla(D(x)\Div\vec{u})$ in equations \eqref{eq:4.57}. Note
that, we cannot use \eqref{eq:2.16} and \eqref{eq:2.22} anymore, we employ
\eqref{eq:4.14} instead. We first calculate,
$\vec{u}\cdot\nabla(\vec{u}\cdot\nabla\phi(x))$.

\begin{lemma}
 Let $\vec{u}$ be defined by \eqref{eq:4.1}. Then, we obtain,
 \begin{equation}
  \label{eq:4.25}
  \begin{split}
   \vec{u}\cdot\nabla(\vec{u}\cdot\nabla\phi(x))
   &=
   ((\nabla^2\phi(x))\vec{u}\cdot\vec{u})
   +
   \frac{1}{2}\nabla |\vec{u}|^2\cdot\nabla\phi(x)
   \\
   &\qquad
   +
   \frac{1}{\pi(x,t)}((\nabla\pi(x,t)\cdot\nabla\phi(x))|\vec{u}|^2
   -
   (\vec{u}\cdot\nabla\pi(x,t))(\vec{u}\cdot\nabla\phi(x))
   ).
  \end{split}
 \end{equation}
\end{lemma}

\begin{proof}
 We denote $\vec{u}=(u^k)_k$. Then,
 \begin{equation}
  \label{eq:4.16}
 \begin{split}
   \vec{u}\cdot\nabla(\vec{u}\cdot\nabla\phi(x))
   &=
   \sum_{k,l}u^k(u^l\phi_{x_l}(x))_{x_k}
   \\
   &=
   \sum_{k,l}u^ku^l\phi_{x_kx_l}(x)
   +
   \sum_{k,l}u^ku^l_{x_k}\phi_{x_l}(x)
   \\
   &=
   ((\nabla^2\phi(x))\vec{u}\cdot\vec{u})
   +
   ((\nabla\vec{u})\vec{u}\cdot\nabla\phi(x)).
  \end{split}
 \end{equation}
 Using \eqref{eq:4.14}, we proceed,
 \begin{equation}
  \label{eq:4.23}
  \begin{split}
   ((\nabla\vec{u})\vec{u}\cdot\nabla\phi(x))
   &=
   ((\mathstrut^T\nabla\vec{u})\vec{u}\cdot\nabla\phi(x))
   +
   ((\nabla\vec{u}-\mathstrut^T\nabla\vec{u})\vec{u}\cdot\nabla\phi(x))
   \\
   &=
   ((\mathstrut^T\nabla\vec{u})\vec{u}\cdot\nabla\phi(x))
   +
   \frac{1}{\pi(x,t)}
   ((\nabla\pi(x,t)\otimes\vec{u}-\vec{u}\otimes\nabla\pi(x,t))
   \vec{u}\cdot\nabla\phi(x)).
  \end{split}
\end{equation}
 Since,
 \begin{equation}
  \label{eq:4.22}
  ((\mathstrut^T\nabla\vec{u})\vec{u}\cdot\nabla\phi(x))
   =
   \sum_{k,l}u^{l}_{x_k}u^l\phi_{x_k}(x)
   =
   \frac12
   \sum_{k}(|\vec{u}|^2)_{x_k}\phi_{x_k}(x)
   =
   \frac12
   \nabla(|\vec{u}|^2)\cdot\nabla\phi(x),
 \end{equation}
 and
 \begin{equation}
  \label{eq:4.24}
  ((\nabla\pi(x,t)\otimes\vec{u}-\vec{u}\otimes\nabla\pi(x,t))
   \vec{u}\cdot\nabla\phi(x))
   =
   (\nabla\pi(x,t)\cdot\nabla\phi(x))|\vec{u}|^2
   -
   (\vec{u}\cdot\nabla\pi(x,t))
   (\vec{u}\cdot\nabla\phi(x)),
 \end{equation}
 we obtain \eqref{eq:4.25} by using \eqref{eq:4.16}, \eqref{eq:4.23},
 \eqref{eq:4.22} and \eqref{eq:4.24}.
\end{proof}

In order to consider
$\vec{u}\cdot\nabla(D(x)\Div\vec{u})$ in \eqref{eq:4.57}, we next
reformulate $\vec{u}\cdot\nabla\Div\vec{u}$.

\begin{lemma}
 Let $\vec{u}$ be defined by \eqref{eq:4.1}. Then, we obtain,
 \begin{equation}
  \label{eq:4.31}
   \begin{split}
    \vec{u}\cdot\nabla\Div\vec{u}
    &=
    \frac{1}{2}\Div(\nabla|\vec{u}|^2)
    -
    |\nabla\vec{u}|^2
    \\
    &\qquad
    +
    \Div\left(
    \frac{1}{\pi(x,t)}
    (|\vec{u}|^2\nabla\pi(x,t)-(\vec{u}\cdot\nabla\pi(x,t))\vec{u})
    \right)
    \\
    &\qquad
    -
    \frac{1}{2\pi(x,t)}(\nabla(|\vec{u}^2|)\cdot\nabla\pi(x,t))
    +
    \frac{1}{\pi(x,t)}((\nabla\vec{u})\vec{u}\cdot\nabla\pi(x,t)).
   \end{split}
 \end{equation}
\end{lemma}

\begin{proof}
 Again,  we denote $\vec{u}=(u^k)_k$. Then,
 \begin{equation}
   \label{eq:4.26}
    \begin{split}
     \vec{u}\cdot\nabla\Div\vec{u}
     &=\sum_{k,l}u^k(u^l_{x_l})_{x_k}
     \\
     &=
     \sum_{k,l}(u^ku^l_{x_k})_{x_l}
     -
     \sum_{k,l}u^k_{x_l}u^l_{x_k}
     \\
     &=
     \Div((\nabla\vec{u})\vec{u})
     -
     \tr((\nabla\vec{u})^2).
    \end{split}
 \end{equation}
 Using \eqref{eq:4.14}, we proceed,
 \begin{equation}
  \label{eq:4.27}
   \begin{split}
    \Div((\nabla\vec{u})\vec{u})
    &=
    \Div((\mathstrut^T\nabla\vec{u})\vec{u})
    +
    \Div((\nabla\vec{u}-\mathstrut^T\nabla\vec{u})\vec{u})
    \\
    &=
    \Div((\mathstrut^T\nabla\vec{u})\vec{u})
    +
    \Div\left(
    \frac{1}{\pi(x,t)}
    (\nabla\pi(x,t)\otimes\vec{u}-\vec{u}\otimes\nabla\pi(x,t))
    \vec{u}\right),
   \end{split}
 \end{equation}
 and
 \begin{equation}
  \label{eq:4.28}
   \begin{split}
    \tr((\nabla\vec{u})^2)
    &=
    \tr(\mathstrut^T\nabla\vec{u}\nabla\vec{u})
    +
    \tr((\nabla\vec{u}-\mathstrut^T\nabla\vec{u})\nabla\vec{u})
    \\
    &=
    |\nabla\vec{u}|^2
    +
    \tr
    \left(
    \frac{1}{\pi(x,t)}
    (\nabla\pi(x,t)\otimes\vec{u}-\vec{u}\otimes\nabla\pi(x,t))
    \nabla\vec{u}
    \right).
   \end{split}
 \end{equation}
 Since,
 \begin{equation}
  \label{eq:4.29}
   (\mathstrut^T\nabla\vec{u})\vec{u}=\frac12\nabla(|\vec{u}|^2),\qquad
   (\nabla\pi(x,t)\otimes\vec{u})\vec{u}=|\vec{u}|^2\nabla\pi(x,t),\qquad
   (\vec{u}\otimes\nabla\pi(x,t))\vec{u}=(\vec{u}\cdot\nabla\pi(x,t))\vec{u},
 \end{equation}
 and
 \begin{equation}
  \label{eq:4.30}
   \tr((\nabla\pi(x,t)\otimes\vec{u})
    \nabla\vec{u})
    =
    \frac12(\nabla(|\vec{u}|^2)\cdot\nabla\pi(x,t))
    ,\qquad
    \tr((\vec{u}\otimes\nabla\pi(x,t))
    \nabla\vec{u})
    =
    ((\nabla\vec{u})\vec{u}\cdot\nabla\pi(x,t)),
 \end{equation}
 we obtain \eqref{eq:4.31}, by using \eqref{eq:4.26}, \eqref{eq:4.27},
 \eqref{eq:4.28}, \eqref{eq:4.29}, and \eqref{eq:4.30}.
 
\end{proof}

Now, we are in a position to compute the first term of the right
hand side of \eqref{eq:4.7}.

\begin{lemma}
 Let $f$ be a solution of \eqref{eq:4.1} and let $\vec{u}$ be given by
 \eqref{eq:4.1}. Then,
 \begin{equation}
  \label{eq:4.21}
   \begin{split}
    2\int_\Omega D(x)\vec{u}\cdot\nabla\rho_t\Feq\,dx
    &=
    2\int_\Omega
    ((\nabla^2\phi(x))\vec{u}\cdot\vec{u})f\,dx
    -
    \int_\Omega
    \pi(x,t)
    \vec{u}\cdot\nabla |\vec{u}|^2 f
    \,dx
    +
    2\int_\Omega D(x)|\nabla \vec{u}|^2f\,dx
    \\
    &\qquad
    -
    2\int_\Omega
    (1+\log\rho)\vec{u}\cdot\nabla D(x)f_t\,dx
    +
    2\int_\Omega
    \vec{u}\cdot\nabla\left(\log f\vec{u}\cdot\nabla D(x)\right) f
    \,dx
    \\
    &\qquad
    -
    \int_\Omega
    (\log f-1)
    \nabla |\vec{u}|^2\cdot\nabla D(x) f
    \,dx
    -
    2\int_\Omega \vec{u}\cdot\nabla D(x)\Div\vec{u}f\,dx
    \\
    &\qquad
    -2
    \int_\Omega
    (\log f-1)
    \frac{1}{\pi(x,t)}
    |\vec{u}|^2\nabla\pi(x,t)\cdot \nabla D(x)
    f\,dx
    \\
    &\qquad
    +2
    \int_\Omega
    (\log f-1)
    \frac{1}{\pi(x,t)}
    (\vec{u}\cdot\nabla\pi(x,t))(\vec{u}\cdot\nabla D(x))
    f\,dx
    \\
    &\qquad
    +
    \int_\Omega
    \frac{D(x)}{\pi(x,t)}((\nabla|\vec{u}|^2)\cdot\nabla\pi(x,t))
    f\,dx
    \\
    &\qquad
    -
    2
    \int_\Omega
    \frac{D(x)}{\pi(x,t)}((\nabla\vec{u})\vec{u}\cdot\nabla\pi(x,t))
    f\,dx,
   \end{split}
 \end{equation}
 where $\Feq$ is given by \eqref{eq:1.5}.
\end{lemma}

\begin{proof}
 First, we use \eqref{eq:4.11} and obtain,
 \begin{equation*}
  \begin{split}
   &\qquad
   2\int_\Omega D(x)\vec{u}\cdot\nabla\rho_t\Feq\,dx \\
   &=
   -2\int_\Omega
   \pi(x,t)
   |\vec{u}|^2f_t\,dx
   -2\int_\Omega
   (1+\log\rho)\vec{u}\cdot\nabla D(x)f_t\,dx
   +
   2\int_\Omega
   \vec{u}\cdot \nabla (\pi(x,t)|\vec{u}|^2) f\,dx \\
   &\qquad
   +
   2\int_\Omega
   \vec{u}\cdot\nabla\left(\log f\vec{u}\cdot\nabla D(x)\right) f
   \,dx
   +
   2\int_\Omega
   \vec{u}\cdot \nabla (\vec{u}\cdot\nabla\phi(x)) f\,dx
   -
   2\int_\Omega
   \vec{u}\cdot\nabla(D(x)\Div\vec{u}) f
   \,dx.
  \end{split}
 \end{equation*}
 Using \eqref{eq:4.12}, the first and the third terms of the right-hand
 side of the above relation are canceled, hence,
 \begin{equation}
  \label{eq:4.13}
  \begin{split}
   &\qquad
   2\int_\Omega D(x)\vec{u}\cdot\nabla\rho_t\Feq\,dx \\
   &=
   -2\int_\Omega
   (1+\log\rho)\vec{u}\cdot\nabla D(x)f_t\,dx
   +
   2\int_\Omega
   \vec{u}\cdot\nabla\left(\log f\vec{u}\cdot\nabla D(x)\right) f
   \,dx \\
   &\qquad
   +
   2\int_\Omega
   \vec{u}\cdot \nabla (\vec{u}\cdot\nabla\phi(x)) f\,dx
   -
   2\int_\Omega
   \vec{u}\cdot\nabla(D(x)\Div\vec{u}) f
   \,dx.
  \end{split}
 \end{equation}

 Using \eqref{eq:4.25} and \eqref{eq:4.31} in \eqref{eq:4.13}, we have that,
 \begin{equation}
  \label{eq:4.19}
   \begin{split}
    &\qquad
    2\int_\Omega D(x)\vec{u}\cdot\nabla\rho_t\Feq\,dx \\
    &=
    -
    2\int_\Omega
    (1+\log\rho)\vec{u}\cdot\nabla D(x)f_t\,dx
    +
    2\int_\Omega
    \vec{u}\cdot\nabla\left(\log f\vec{u}\cdot\nabla D(x)\right) f
    \,dx
    \\
    &\qquad
    +
    2\int_\Omega
    ((\nabla^2\phi(x))\vec{u}\cdot\vec{u})f\,dx
    +
    \int_\Omega\nabla|\vec{u}|^2\cdot\nabla\phi(x) f\,dx
    \\
    &\qquad
    +
    2\int_\Omega
    \frac{1}{\pi(x,t)}(\nabla\pi(x,t)\cdot\nabla\phi(x))|\vec{u}|^2f
    \,dx
    -
    2\int_\Omega
    \frac{1}{\pi(x,t)}
    (\vec{u}\cdot\nabla\pi(x,t))(\vec{u}\cdot\nabla\phi(x))
    f\,dx
    \\
    &\qquad
    -
    \int_\Omega
    D(x)
    \Div(\nabla |\vec{u}|^2)
    f
    \,dx
    +
    2\int_\Omega D(x)|\nabla \vec{u}|^2f\,dx
    \\
    &\qquad
    -2
    \int_\Omega
    D(x)
    \Div\left(
    \frac{1}{\pi(x,t)}
    (|\vec{u}|^2\nabla\pi(x,t)-(\vec{u}\cdot\nabla\pi(x,t))\vec{u})
    \right)
    f\,dx
    \\
    &\qquad
    +
    \int_\Omega
    \frac{D(x)}{\pi(x,t)}(\nabla(|\vec{u}|^2)\cdot\nabla\pi(x,t))
    f\,dx
    -
    2
    \int_\Omega
    \frac{D(x)}{\pi(x,t)}((\nabla\vec{u})\vec{u}\cdot\nabla\pi(x,t))
    f\,dx
    \\
    &\qquad
    -
    2\int_\Omega \vec{u}\cdot\nabla D(x)\Div\vec{u}f\,dx.
   \end{split}
 \end{equation}

 To calculate the 7th term of the right-hand side of \eqref{eq:4.19},
we  apply integration by parts together with the periodic boundary
 condition, and, thus obtain,
 \begin{equation*}
  -
   \int_\Omega
   D(x)
   \Div(\nabla |\vec{u}|^2)
   f
   \,dx
   =
   \int_\Omega
   D(x)
   \nabla |\vec{u}|^2\cdot\nabla f
   \,dx
   +
   \int_\Omega
   \nabla |\vec{u}|^2\cdot\nabla D(x) f
   \,dx.
    \end{equation*}
Using \eqref{eq:4.8} in the above relation, we have that,
 \begin{equation}
  \label{eq:4.20}
 \begin{split}
  -
  \int_\Omega
  D(x)
  \Div(\nabla |\vec{u}|^2)
  f
  \,dx
  &=
  -
  \int_\Omega
  \pi(x,t)
  \vec{u}\cdot\nabla |\vec{u}|^2 f
  \,dx
  -
  \int_\Omega
  \nabla |\vec{u}|^2\cdot\nabla\phi(x) f
  \,dx
  \\
  &\qquad
  -
  \int_\Omega
  (\log f-1)
  \nabla |\vec{u}|^2\cdot\nabla D(x) f
  \,dx.
 \end{split}
 \end{equation}

 Next, we compute the 9th term of the right-hand side of
 \eqref{eq:4.19}. Applying integration by parts together with
 \eqref{eq:4.1} and the periodic boundary condition, we have that,
 \begin{equation*}
  \begin{split}
   &\qquad
   -2
   \int_\Omega
   D(x)
   \Div\left(
   \frac{1}{\pi(x,t)}
   (|\vec{u}|^2\nabla\pi(x,t)-(\vec{u}\cdot\nabla\pi(x,t))\vec{u})
   \right)
   f\,dx
   \\
   &=
   2
   \int_\Omega
   D(x)
   \left(
   \frac{1}{\pi(x,t)}
   (|\vec{u}|^2\nabla\pi(x,t)-(\vec{u}\cdot\nabla\pi(x,t))\vec{u})
   \right)
   \cdot
   \nabla f\,dx
   \\
   &\qquad
   +2
   \int_\Omega
   \left(
   \frac{1}{\pi(x,t)}
   (|\vec{u}|^2\nabla\pi(x,t)-(\vec{u}\cdot\nabla\pi(x,t))\vec{u})
   \right)
   \cdot
   \nabla D(x)
   f\,dx.
  \end{split}
\end{equation*}
 Using \eqref{eq:4.8}, we obtain,
 \begin{equation*}
  \begin{split}
   &\qquad
   D(x)
   \left(
   \frac{1}{\pi(x,t)}
   (|\vec{u}|^2\nabla\pi(x,t)-(\vec{u}\cdot\nabla\pi(x,t))\vec{u})
   \right)
   \cdot
   \nabla f
   \\
   &=
   -
   \frac{1}{\pi(x,t)}|\vec{u}|^2(\nabla\pi(x,t)\cdot\nabla D(x))f\log f
   -
   \frac{1}{\pi(x,t)}|\vec{u}|^2(\nabla\pi(x,t)\cdot\nabla \phi(x))f
   \\
   &\qquad
   +
   \frac{1}{\pi(x,t)}(\vec{u}\cdot\nabla\pi(x,t))(\vec{u}\cdot\nabla D(x))f\log f
   +
   \frac{1}{\pi(x,t)}(\vec{u}\cdot\nabla\pi(x,t))(\vec{u}\cdot\nabla\phi(x))f.
  \end{split}
 \end{equation*}
 Hence,  we have,
 \begin{equation}
 \label{eq:4.32}
  \begin{split}
   &\qquad
   -2
   \int_\Omega
   D(x)
   \Div\left(
   \frac{1}{\pi(x,t)}
   (|\vec{u}|^2\nabla\pi(x,t)-(\vec{u}\cdot\nabla\pi(x,t))\vec{u})
   \right)
   f\,dx
   \\
   &=
   2\int_\Omega
   \frac{1}{\pi(x,t)}(\vec{u}\cdot\nabla\pi(x,t))(\vec{u}\cdot\nabla\phi(x))f
   \,dx
   -
   2\int_\Omega
  \frac{1}{\pi(x,t)}|\vec{u}|^2(\nabla\pi(x,t)\cdot\nabla \phi(x))f
   \,dx
   \\
   &\qquad
   -2
   \int_\Omega
   (\log f-1)
   \frac{1}{\pi(x,t)}
   |\vec{u}|^2\nabla\pi(x,t)\cdot \nabla D(x)
   f\,dx
   \\
   &\qquad
   +2
   \int_\Omega
   (\log f-1)
   \frac{1}{\pi(x,t)}
   (\vec{u}\cdot\nabla\pi(x,t))(\vec{u}\cdot\nabla D(x))
   f\,dx.
  \end{split}
 \end{equation}
 Using \eqref{eq:4.20} and \eqref{eq:4.32} in \eqref{eq:4.19}, we
 obtain the desired result \eqref{eq:4.21}.
\end{proof}

Next, we further compute in \eqref{eq:4.21},

\begin{lemma}
 Let $\vec{u}$ be given by \eqref{eq:4.1}. Then,
 \begin{equation}
  \label{eq:4.43}
  \begin{split}
   &\qquad
   \int_\Omega
   \vec{u}\cdot\nabla\left(\log f\vec{u}\cdot\nabla D(x)\right) f\,dx
   \\
   &=
   \int_\Omega
   \frac{\pi(x,t)}{D(x)}|\vec{u}|^2
   \log f\left(\vec{u}\cdot\nabla D(x)\right) f\,dx
   +
   \int_\Omega
   \frac{1}{D(x)}
   (\log f)^2\left(\vec{u}\cdot\nabla D(x)\right)^2 f\,dx
   \\
   &\qquad
   +
   \int_\Omega
   \frac{1}{D(x)}
   \log f\left(\vec{u}\cdot\nabla D(x)\right)\left(\vec{u}\cdot\nabla \phi(x)\right) f\,dx
   -
   \int_\Omega
   \log f\left(\vec{u}\cdot\nabla D(x)\right)\Div\vec{u} f\,dx.
  \end{split}
 \end{equation}
\end{lemma}

\begin{proof}
 Applying the integration by parts to the left hand side of \eqref{eq:4.43}
 together with the periodic boundary condition \eqref{eq:4.1}, we
 obtain that,
 \begin{equation}
  \label{eq:4.44}
  \int_\Omega
   \vec{u}\cdot\nabla\left(\log f\vec{u}\cdot\nabla D(x)\right) f\,dx
   =
   -
   \int_\Omega
   \log f\left(\vec{u}\cdot\nabla D(x)\right) \Div(f\vec{u})\,dx.
 \end{equation}
 Using direct computation together with \eqref{eq:4.8}, we have,
 \begin{equation}
  \label{eq:4.45}
  \begin{split}
   \Div(f\vec{u})
   &=
   \vec{u}\cdot\nabla f
   +
   f\Div\vec{u} \\
   &=
   -
   \frac{\pi(x,t)}{D(x)}|\vec{u}|^2f
   -
   \frac{1}{D(x)}\log f(\vec{u}\cdot\nabla D(x))f
   -
   \frac{1}{D(x)}(\vec{u}\cdot\nabla \phi(x))f
   +
   f\Div\vec{u}.
  \end{split}
 \end{equation}
 Combining \eqref{eq:4.44} and \eqref{eq:4.45}, we arrive at
 \eqref{eq:4.43}.
\end{proof}

Now, combining \eqref{eq:4.7}, \eqref{eq:4.12}, \eqref{eq:4.21}, and
\eqref{eq:4.43}, we obtain the following energy law.

\begin{proposition}
 Let $f$ be a solution of \eqref{eq:4.1},  and let $\vec{u}$ be given
 as in
 \eqref{eq:4.1}. Then,
 \begin{equation}
  \label{eq:4.33}
   \begin{split}
    \frac{d^2F}{dt^2}[f](t)
    &=
    2\int_\Omega
    ((\nabla^2\phi(x))\vec{u}\cdot\vec{u})f\,dx
    +
    2\int_\Omega D(x)|\nabla \vec{u}|^2f\,dx
    \\
    &\qquad
    -
    \int_\Omega
    (\log f-1)
    \nabla |\vec{u}|^2\cdot\nabla D(x) f
    \,dx
    -
    2
    \int_\Omega
    (1+\log f)
    \vec{u}\cdot\nabla D(x)\Div\vec{u} f\,dx
    \\
    &\qquad
    +
    2
    \int_\Omega
    \frac{\pi(x,t)}{D(x)}|\vec{u}|^2
    \log f\left(\vec{u}\cdot\nabla D(x)\right) f\,dx
    +
    2
     \int_\Omega
    \frac{1}{D(x)}
    (\log f)^2\left(\vec{u}\cdot\nabla D(x)\right)^2 f\,dx
    \\
    &\qquad
    +
    2\int_\Omega
   \frac{1}{D(x)}
    \log f\left(\vec{u}\cdot\nabla D(x)\right)\left(\vec{u}\cdot\nabla \phi(x)\right) f\,dx
    \\
    &\qquad
    +
    \int_\Omega\pi_t(x,t)|\vec{u}|^2f\,dx
    +
    \int_\Omega
    |\vec{u}|^2\vec{u}\cdot\nabla \pi(x,t) f
    \,dx
    \\
    &\qquad
    -2
    \int_\Omega
    (\log f-1)
    \frac{1}{\pi(x,t)}
    |\vec{u}|^2\nabla\pi(x,t)\cdot \nabla D(x)
    f\,dx
    \\
    &\qquad
    +2
    \int_\Omega
    (\log f-1)
    \frac{1}{\pi(x,t)}
    (\vec{u}\cdot\nabla\pi(x,t))(\vec{u}\cdot\nabla D(x))
    f\,dx
    \\
    &\qquad
    +
    \int_\Omega
    \frac{D(x)}{\pi(x,t)}((\nabla|\vec{u}|^2)\cdot\nabla\pi(x,t))
    f\,dx
    -
    2
    \int_\Omega
    \frac{D(x)}{\pi(x,t)}((\nabla\vec{u})\vec{u}\cdot\nabla\pi(x,t))
    f\,dx.
   \end{split}
 \end{equation}
\end{proposition}

\begin{proof}
 Since $\pi(x,t)\nabla|\vec{u}|^2=\nabla
 (\pi(x,t)|\vec{u}|^2)-|\vec{u}|^2\nabla\pi(x,t)$, the second term of
 the right-hand side of \eqref{eq:4.21} becomes,
 \begin{equation*}
  -\int_\Omega
   \pi(x,t)
   \vec{u}\cdot\nabla |\vec{u}|^2 f
   \,dx
   =
   -
   \int_\Omega
   \vec{u}\cdot\nabla (\pi(x,t)|\vec{u}|^2) f
   \,dx
   +
   \int_\Omega
   |\vec{u}|^2\vec{u}\cdot\nabla \pi(x,t) f
   \,dx.
 \end{equation*}
 Using this relation we obtain \eqref{eq:4.33}.
\end{proof}

We are searching for a sufficient condition to obtain a differential
inequality for $\frac{dF}{dt}$. The 5th and the 9th terms of the
right-hand side of \eqref{eq:4.33} involve $|\vec{u}|^3$, the order
which is
higher than $2$. Thus,  we will handle such terms using the Sobolev
inequality below. As in the proof of Lemma \ref{lem:3.13} in Section~\ref{sec:3}, we have
following Sobolev inequality for any periodic vector field $\vec{v}$.

\begin{lemma}
\label{lem:4.15a}
 Let $n=1,2,3$. Let $f_0$ be a probability density function, and let $f$
 be a solution of \eqref{eq:4.1}. Then,
 \begin{equation}
  \label{eq:4.47}
   \int_\Omega
   |\vec{v}|^3f
   \,dx
   \leq
   \frac{3\Cr{const:4.3}^{\frac{3}{2}}}{4}
   \int_\Omega
   |\nabla\vec{v}|^2f
   \,dx
   +
   \frac{\Cr{const:4.3}^{\frac{3}{2}}}{4}
   \left(
    \int_\Omega
    |\vec{v}|^2f
    \,dx
   \right)^3,
 \end{equation}
 for any periodic vector field $\vec{v}$.
\end{lemma}

Using the Sobolev inequality, we obtain the following energy estimate.

\begin{proposition}
\label{prop:4.16a}
 Assume $n=1,2,3$, let $f$ be a solution of \eqref{eq:4.1}, and let
 $\vec{u}$ be given as in \eqref{eq:4.1}. Suppose, that there exists a
 positive constant $\lambda>0$,  such that $\nabla^2\phi\geq \lambda I$,
 where $I$ is the identity matrix. Then, there are constants,
 $\Cr{const:4.4},\ \Cr{const:4.5}>0$,  such that if,
 \begin{equation}
  \|\nabla D\|_{L^\infty(\Omega)}\leq \Cr{const:4.4},\quad
  \|\nabla \pi\|_{L^\infty(\Omega\times[0,\infty))}\leq \Cr{const:4.5},\quad
  \pi_t(x,t)\geq -\Cr{const:4.9}, 
 \end{equation}
 then, we have that,
 \begin{equation}
  \label{eq:4.46}
   \frac{d^2F}{dt^2}[f](t)
   \geq
   \frac{\lambda}{\|\pi\|_{L^\infty(\Omega\times[0,\infty))}}
   \int_{\Omega}
   \pi(x,t)|\vec{u}|^2f\,dx
   -
   \frac{2\Cr{const:1.3}}{3\Cr{const:1.2}^3}   
   \left(
    \int_{\Omega}
    \pi(x,t)|\vec{u}|^2f\,dx
   \right)^3.
 \end{equation}
\end{proposition}

\begin{proof}
 We proceed with calculations of the integrands of the 3rd, 4th, 5th, 7th, 8th, 9th, 10th,
 11st, 12nd, and 13rd terms of \eqref{eq:4.33}. As in the proof of the
 Proposition
 \ref{prop:3.14} in Section~\ref{sec:3}, for any positive constants $\Cl[eps]{eps:4.1},
 \Cl[eps]{eps:4.2}>0$, we have that,
 \begin{equation}
  |
   (\log f-1)\nabla |\vec{u}|^2\cdot\nabla D(x) f
   |
   \leq
   \frac{1}{2\Cr{eps:4.1}}
   D(x)
   (|\log f|+1)^2
   |\nabla \vec{u}|^2
   |\nabla (\log D(x))|^2 f
   +
   2\Cr{eps:4.1}
   D(x)
   |\vec{u}|^2 f,
 \end{equation}
 \begin{equation}
  |2(1+\log f)\vec{u}\cdot\nabla D(x)\Div\vec{u}f|
   \leq
   \frac{1}{2\Cr{eps:4.2}}
   D(x)
   (|\log f|+1)^2
   |\nabla \vec{u}|^2
   |\nabla (\log D(x))|^2 f
   +
   2\Cr{eps:4.2}
   D(x)
   |\vec{u}|^2 f,
 \end{equation}
 and
 \begin{equation}
  \left|
   \frac{2}{D(x)}
   \log f\left(\vec{u}\cdot\nabla D(x)\right)\left(\vec{u}\cdot\nabla \phi(x)\right) f
  \right|
  \leq
  2|\log f|
  |\nabla (\log D(x))|
  |\nabla \phi(x)|
  |\vec{u}|^2f.
 \end{equation}

 We further estimate the 10th, 11th, 12th and 13th terms of the right-hand
 side of \eqref{eq:4.33} using the Cauchy-Schwarz inequality,
 \begin{equation}
  \begin{split}
   &\qquad
   \left|
   2(\log f-1)
   \frac{1}{\pi(x,t)}
   |\vec{u}|^2\nabla\pi(x,t)\cdot \nabla D(x)
   f
   \right|
   \\
   &=
   2\left|
   D(x)(\log f-1)
   |\vec{u}|^2\nabla(\log\pi(x,t))\cdot \nabla (\log D(x))
   f
   \right|
   \\
   &\leq
   2D(x)
   (|\log f|+1)
   |\nabla(\log\pi(x,t))|
   |\nabla (\log D(x))|
   |\vec{u}|^2f,
  \end{split}
 \end{equation}

 \begin{equation}
  \begin{split}
   &\qquad
   \left|
   2(\log f-1)
   \frac{1}{\pi(x,t)}
   (\vec{u}\cdot\nabla\pi(x,t))(\vec{u}\cdot\nabla D(x))
   f
   \right|
   \\
   &=
   2
   \left|
   D(x)(\log f-1)
   (\vec{u}\cdot\nabla(\log\pi(x,t)))
   (\vec{u}\cdot \nabla (\log D(x)))
   f
   \right|
   \\
   &\leq
   2D(x)
   (|\log f|+1)
   |\nabla(\log\pi(x,t))|
   |\nabla (\log D(x))|
   |\vec{u}|^2f,
  \end{split}
 \end{equation}

 \begin{equation}
  \begin{split}
   \left|
   \frac{D(x)}{\pi(x,t)}((\nabla|\vec{u}|^2)\cdot\nabla\pi(x,t))f
   \right|
   &\leq
   2D(x)|\nabla \vec{u}||\vec{u}||\nabla(\log \pi(x,t))| f
   \\
   &\leq
   \frac{D(x)}{2\Cr{eps:4.3}}|\nabla \vec{u}|^2|\nabla(\log \pi(x,t))|^2 f
   +
   2\Cr{eps:4.3}D(x)|\vec{u}|^2f,
  \end{split}
 \end{equation}
 and
 \begin{equation}
  \begin{split}
   \left|
    2
   \frac{D(x)}{\pi(x,t)}((\nabla\vec{u})\vec{u}\cdot\nabla\pi(x,t))f
   \right|
   &\leq
   2D(x)|\nabla \vec{u}||\vec{u}||\nabla(\log \pi(x,t))| f
   \\
   &\leq
   \frac{D(x)}{2\Cr{eps:4.4}}|\nabla \vec{u}|^2|\nabla(\log \pi(x,t))|^2 f
   +
   2\Cr{eps:4.4}D(x)|\vec{u}|^2f,
  \end{split}
 \end{equation}
 where, $\Cl[eps]{eps:4.3} \mbox{ and } \Cl[eps]{eps:4.4}>0$ are positive constants.

 Thus, using all inequalities above in \eqref{eq:4.33}, we arrive at
 the estimate for $\frac{d^2F}{dt^2}[f](t)$,
 \begin{equation}
  \label{eq:4.49}
  \begin{split}
   &\qquad
   \frac{d^2F}{dt^2}[f](t)
   \\
   &\geq
   2\int_\Omega
   \biggl((\nabla^2\phi(x)\vec{u}\cdot\vec{u})
   +
   \pi_t(x,t)|\vec{u}|^2
   -
   \Bigl(
   (\Cr{eps:4.1}+\Cr{eps:4.2}+\Cr{eps:4.3}+\Cr{eps:4.4})D(x)
   \\
   &\qquad
   +
   2D(x)
   (|\log f|+1)
   |\nabla(\log\pi(x,t))|
   |\nabla (\log D(x))|
   +
   |\log f|
   |\nabla (\log D(x))|
   |\nabla \phi(x)|\Bigr) |\vec{u}|^2
   \biggr) f\,dx
   \\
   &\qquad
   +
   2\int_\Omega
   \biggl(
   1
   -
   \frac14
   \left(
   \frac{1}{\Cr{eps:3.1}}
   +
   \frac{1}{\Cr{eps:3.2}}
   \right)
   (|\log f|+1)^2
   |\nabla (\log D(x))|^2 
   \\
   &\qquad\qquad
   -
   \frac14
   \left(
   \frac{1}{\Cr{eps:4.3}}
   +
   \frac{1}{\Cr{eps:4.4}}
   \right)|\nabla(\log \pi(x,t))|^2
   \biggr)
   D(x)|\nabla \vec{u}|^2 f\,dx
   \\
   &\qquad
   +
   2
   \int_\Omega
   \frac{\pi(x,t)}{D(x)}|\vec{u}|^2
   \log f\left(\vec{u}\cdot\nabla D(x)\right) f\,dx
   +
   2
   \int_\Omega
   \frac{1}{D(x)}
   (\log f)^2\left(\vec{u}\cdot\nabla D(x)\right)^2 f\,dx
   \\
   &\qquad
   +
   \int_\Omega
   |\vec{u}|^2\vec{u}\cdot\nabla \pi(x,t) f
   \,dx.
  \end{split}
 \end{equation}

 Next,  using \eqref{eq:4.47},
 $D(x)/\Cr{const:1.3}\geq1$, and $\pi(x,t)/\Cr{const:1.2}\geq1$, we compute,
 \begin{equation}
  \label{eq:4.50}
   \begin{split}
   &\qquad
   \left|
   2
   \int_\Omega
   \frac{\pi(x,t)}{D(x)}|\vec{u}|^2
   \log f\left(\vec{u}\cdot\nabla D(x)\right) f\,dx
   \right| \\
   &\leq
   2
   \|\log f\|_{L^\infty(\Omega\times[0,\infty))}
   \|\pi\|_{L^\infty(\Omega\times[0,\infty))}
   \|\nabla \log D\|_{L^\infty(\Omega)}
   \int_\Omega|\vec{u}|^3f\,dx \\
   &\leq
   \frac{3\Cr{const:4.3}^{\frac{3}{2}}}{2}
   \|\log f\|_{L^\infty(\Omega\times[0,\infty))}
   \|\pi\|_{L^\infty(\Omega\times[0,\infty))}
   \|\nabla \log D\|_{L^\infty(\Omega)}
   \int_\Omega|\nabla\vec{u}|^2f\,dx
   \\
   &\qquad
   +
   \frac{\Cr{const:4.3}^{\frac{3}{2}}}{2}
   \|\log f\|_{L^\infty(\Omega\times[0,\infty))}
   \|\pi\|_{L^\infty(\Omega\times[0,\infty))}
   \|\nabla \log D\|_{L^\infty(\Omega)}
   \left(
   \int_\Omega |\vec{u}|^2f\,dx
   \right)^3
   \\
   &\leq
   \frac{3\Cr{const:4.3}^{\frac{3}{2}}}{2\Cr{const:1.3}}
   \|\log f\|_{L^\infty(\Omega\times[0,\infty))}
   \|\pi\|_{L^\infty(\Omega\times[0,\infty))}
   \|\nabla \log D\|_{L^\infty(\Omega)}
   \int_\Omega D(x)|\nabla\vec{u}|^2f\,dx
   \\
   &\qquad
   +
   \frac{\Cr{const:4.3}^{\frac{3}{2}}}{2\Cr{const:1.2}^3}
   \|\log f\|_{L^\infty(\Omega\times[0,\infty))}
   \|\pi\|_{L^\infty(\Omega\times[0,\infty))}
   \|\nabla \log D\|_{L^\infty(\Omega)}
   \left(
   \int_\Omega\pi(x,t)|\vec{u}|^2f\,dx
   \right)^3.
   \end{split}
 \end{equation}

 Next, again using \eqref{eq:4.47} and $D(x)/\Cr{const:1.3}\geq1$, we estimate,
 \begin{equation}
  \label{eq:4.51}
  \begin{split}
   &\qquad
   \left|
   \int_\Omega
   |\vec{u}|^2\vec{u}\cdot\nabla\pi(x,t)
   f\,dx
   \right|
   \\
   &\leq
   \|\nabla\pi\|_{L^\infty(\Omega\times[0,\infty))}
   \int_\Omega|\vec{u}|^3f\,dx
   \\
   &\leq
   \frac{3\Cr{const:4.3}^{\frac{3}{2}}}{4}
   \|\nabla\pi\|_{L^\infty(\Omega\times[0,\infty))}
   \int_\Omega|\nabla\vec{u}|^2f\,dx
   +
   \frac{\Cr{const:4.3}^{\frac{3}{2}}}{4}
   \|\nabla\pi\|_{L^\infty(\Omega\times[0,\infty))}
   \left(
   \int_\Omega |\vec{u}|^2f\,dx
   \right)^3
   \\
   &\leq
   \frac{3\Cr{const:4.3}^{\frac{3}{2}}}{4\Cr{const:1.3}}
   \|\nabla\pi\|_{L^\infty(\Omega\times[0,\infty))}
   \int_\Omega D(x)|\nabla\vec{u}|^2f\,dx
   +
   \frac{\Cr{const:4.3}^{\frac{3}{2}}}{4\Cr{const:1.2}^3}
   \|\nabla\pi\|_{L^\infty(\Omega\times[0,\infty))}
   \left(
   \int_\Omega \pi(x,t)|\vec{u}|^2f\,dx
   \right)^3.
  \end{split}
 \end{equation}

 Therefore, using \eqref{eq:4.49}, \eqref{eq:4.50}, and
 \eqref{eq:4.51}, we obtain that,
 \begin{equation}
  \label{eq:4.56}
  \begin{split}
   &\qquad
   \frac{d^2F}{dt^2}[f](t)
   \\
   &\geq
   2\int_\Omega
   \biggl((\nabla^2\phi(x)\vec{u}\cdot\vec{u})
   +
   \pi_t(x,t)|\vec{u}|^2
   -
   \Bigl(
   (\Cr{eps:4.1}+\Cr{eps:4.2}+\Cr{eps:4.3}+\Cr{eps:4.4})D(x)
   \\
   &\qquad
   -
   2D(x)
   (|\log f|+1)
   |\nabla(\log\pi(x,t))|
   |\nabla (\log D(x))|
   -
   |\log f|
   |\nabla (\log D(x))|
   |\nabla \phi(x)|\Bigr) |\vec{u}|^2
   \biggr) f\,dx
   \\
   &\qquad
   +
   2\int_\Omega
   \biggl(
   1
   -
   \frac14
   \left(
   \frac{1}{\Cr{eps:3.1}}
   +
   \frac{1}{\Cr{eps:3.2}}
   \right)
   (|\log f|+1)^2
   |\nabla (\log D(x))|^2 
   \\
   &\qquad\qquad
   -
   \frac14
   \left(
   \frac{1}{\Cr{eps:4.3}}
   +
   \frac{1}{\Cr{eps:4.4}}
   \right)|\nabla(\log \pi(x,t))|^2
   \biggr)
   D(x)|\nabla \vec{u}|^2 f\,dx
   \\
   &\qquad
   -
   \frac{3\Cr{const:4.3}^{\frac{3}{2}}}{2\Cr{const:1.3}}
   \|\log f\|_{L^\infty(\Omega\times[0,\infty))}
   \|\pi\|_{L^\infty(\Omega\times[0,\infty))}
   \|\nabla \log D\|_{L^\infty(\Omega)} 
   \int_\Omega
   D(x)|\nabla \vec{u}|^2 f\,dx
   \\
   &\qquad
   -
   \frac{3\Cr{const:4.3}^{\frac{3}{2}}}{4\Cr{const:1.3}}
   \|\nabla\pi\|_{L^\infty(\Omega\times[0,\infty))}
   \int_\Omega
   D(x)|\nabla \vec{u}|^2 f\,dx
   \\
   &\qquad
   -
   \frac{\Cr{const:4.3}^{\frac{3}{2}}}{2\Cr{const:1.2}^3}
   \|\log f\|_{L^\infty(\Omega\times[0,\infty))}
   \|\pi\|_{L^\infty(\Omega\times[0,\infty))}
   \|\nabla \log D\|_{L^\infty(\Omega)}
   \left(
   \int_\Omega\pi(x,t)|\vec{u}|^2f\,dx
   \right)^3
   \\
   &\qquad
   -
   \frac{\Cr{const:4.3}^{\frac{3}{2}}}{4\Cr{const:1.2}^3}
   \|\nabla\pi\|_{L^\infty(\Omega\times[0,\infty))}
   \left(
   \int_\Omega \pi(x,t)|\vec{u}|^2f\,dx
   \right)^3.
  \end{split}
 \end{equation}
 
 Due to the maximum principle, Proposition \ref{prop:1.5}, there is a
 positive constant $\Cl{const:4.8}>0$ which depends only on $f_0$,
 $f^{\mathrm{eq}}$, and $D$,  such that,  $|\log f(x,t)|\leq \Cr{const:4.8}$
 for $x\in\Omega$ and $t>0$. If we further assume that,  $\|\nabla D\|_{L^\infty(\Omega)}\leq
 \Cr{const:4.4}$, $\|\nabla
 \pi\|_{L^\infty(\Omega\times[0,\infty))}\leq \Cr{const:4.5}$ and
 $\pi_t(x,t)\geq -\Cr{const:4.9}$, then
 $|\nabla\log D(x,t)|\leq \Cr{const:4.4}/\Cr{const:1.3}$ and $|\nabla
 \log\pi(x,t)|\leq \Cr{const:4.5}/\Cr{const:1.2}$.  Hence,  we have
 the following estimate,
 \begin{equation}
  \begin{split}
   &\qquad
   -\pi_t(x,t)
   +
   (\Cr{eps:4.1}+\Cr{eps:4.2}+\Cr{eps:4.3}+\Cr{eps:4.4})D(x)
   \\
   &\qquad
   +
   2D(x)
   (|\log f|+1)
   |\nabla(\log\pi(x,t))|
   |\nabla (\log D(x))|
   +
   |\log f|
   |\nabla (\log D(x))|
   |\nabla \phi(x)| \\
   &\leq
   \Cr{const:4.9}
   +
   (\Cr{eps:4.1}+\Cr{eps:4.2}+\Cr{eps:4.3}+\Cr{eps:4.4})\|D\|_{L^\infty(\Omega)}
   +
   \frac{2\Cr{const:4.4}\Cr{const:4.5}}{\Cr{const:1.2}\Cr{const:1.3}}
   (\Cr{const:4.8}+1)
   \|D\|_{L^\infty(\Omega)}
   +
   \frac{\Cr{const:4.4}\Cr{const:4.8}}{\Cr{const:1.3}}
   \|\nabla\phi\|_{L^\infty(\Omega)},
  \end{split} 
 \end{equation}
 and
 \begin{equation}
  \begin{split}
   &\qquad
   \frac14
   \left(
    \frac{1}{\Cr{eps:3.1}}
    +
    \frac{1}{\Cr{eps:3.2}}
   \right)
   (|\log f|+1)^2
   |\nabla (\log D(x))|^2 
   +
   \frac14
   \left(
   \frac{1}{\Cr{eps:4.3}}
   +
   \frac{1}{\Cr{eps:4.4}}
   \right)|\nabla(\log \pi(x,t))|^2 
   \\
   &\qquad
   +
   \frac{3\Cr{const:4.3}^{\frac{3}{2}}}{4\Cr{const:1.3}}
   \|\log f\|_{L^\infty(\Omega\times[0,\infty))}
   \|\pi\|_{L^\infty(\Omega\times[0,\infty))}
   \|\nabla \log D\|_{L^\infty(\Omega)} 
   +
   \frac{3\Cr{const:4.3}^{\frac{3}{2}}}{8\Cr{const:1.3}}
   \|\nabla\pi\|_{L^\infty(\Omega\times[0,\infty))} 
   \\
   &\leq
   \frac14
   \left(
    \frac{1}{\Cr{eps:3.1}}
    +
    \frac{1}{\Cr{eps:3.2}}
   \right)
   (\Cr{const:4.8}+1)^2
   \frac{\Cr{const:4.4}^2}{\Cr{const:1.3}^2}
   +
   \frac14
   \left(
   \frac{1}{\Cr{eps:4.3}}
   +
   \frac{1}{\Cr{eps:4.4}}
   \right)
   \frac{\Cr{const:4.5}^2}{\Cr{const:1.2}^2}
   +
   \frac{3
   \Cr{const:4.3}^{\frac{3}{2}}\Cr{const:4.4}\Cr{const:4.8}}{4\Cr{const:1.3}^2}
   \|\pi\|_{L^\infty(\Omega\times[0,\infty))}
   +
   \frac{3
   \Cr{const:4.3}^{\frac{3}{2}}\Cr{const:4.5}}{8\Cr{const:1.3}}.
   \end{split} 
 \end{equation}
 Therefore, first take small positive constants,  $\Cr{eps:4.1},
 \Cr{eps:4.2}, \Cr{eps:4.3} \mbox{ and }
 \Cr{eps:4.4}$.  Next, take sufficiently small positive constants, $\Cr{const:4.4},
 \Cr{const:4.5} \mbox{ and }\Cr{const:4.9}$,  such that,
 \begin{multline}
  \label{eq:4.52}
  -\pi_t(x,t)
  +
  (\Cr{eps:4.1}+\Cr{eps:4.2}+\Cr{eps:4.3}+\Cr{eps:4.4})D(x)
  \\
   +
   2D(x)
   (|\log f|+1)
   |\nabla(\log\pi(x,t))|
   |\nabla (\log D(x))|
   +
   |\log f|
   |\nabla (\log D(x))|
   |\nabla \phi(x)|
  \leq
  \frac{\lambda}{2},
 \end{multline}
 and
 \begin{multline}
  \label{eq:4.53}
  \frac14
  \left(
  \frac{1}{\Cr{eps:3.1}}
  +
  \frac{1}{\Cr{eps:3.2}}
  \right)
  (|\log f|+1)^2
  |\nabla (\log D(x))|^2 
  +
  \frac14
  \left(
  \frac{1}{\Cr{eps:4.3}}
  +
  \frac{1}{\Cr{eps:4.4}}
  \right)|\nabla(\log \pi(x,t))|^2 
  \\
  +
  \frac{3\Cr{const:4.3}^{\frac{3}{2}}}{4\Cr{const:1.3}}
  \|\log f\|_{L^\infty(\Omega\times[0,\infty))}
  \|\pi\|_{L^\infty(\Omega\times[0,\infty))}
  \|\nabla \log D\|_{L^\infty(\Omega)} 
  +
  \frac{3\Cr{const:4.3}^{\frac{3}{2}}}{8\Cr{const:1.3}}
  \|\nabla\pi\|_{L^\infty(\Omega\times[0,\infty))} 
  \leq
  1. 
 \end{multline}
 Note that, 
 $\frac{3\Cr{const:4.3}^{\frac{3}{2}}}{4\Cr{const:1.3}}
 \|\log f\|_{L^\infty(\Omega\times[0,\infty))}
 \|\pi\|_{L^\infty(\Omega\times[0,\infty))}
 \|\nabla \log D\|_{L^\infty(\Omega)} 
 +
 \frac{3\Cr{const:4.3}^{\frac{3}{2}}}{8\Cr{const:1.3}}
 \|\nabla\pi\|_{L^\infty(\Omega\times[0,\infty))} 
 \leq
 1$, thus we have the estimate,
 \begin{multline}
  \label{eq:4.55}
 -
  \frac{\Cr{const:4.3}^{\frac{3}{2}}}{2\Cr{const:1.2}^3}
   \|\log f\|_{L^\infty(\Omega\times[0,\infty))}
  \|\pi\|_{L^\infty(\Omega\times[0,\infty))}
  \|\nabla \log D\|_{L^\infty(\Omega)}
  \left(
  \int_\Omega\pi(x,t)|\vec{u}|^2f\,dx
  \right)^3
  \\
  -
  \frac{\Cr{const:4.3}^{\frac{3}{2}}}{4\Cr{const:1.2}^3}
  \|\nabla\pi\|_{L^\infty(\Omega\times[0,\infty))}
  \left(
  \int_\Omega \pi(x,t)|\vec{u}|^2f\,dx
  \right)^3
  \geq
  -\frac{2\Cr{const:1.3}}{3\Cr{const:1.2}^3}
  \left(
  \int_\Omega \pi(x,t)|\vec{u}|^2f\,dx
  \right)^3.
 \end{multline}

 Note that, $\pi(x,t)/\|\pi\|_{L^\infty(\Omega\times[0,\infty))}\leq1$
 hence
 \begin{equation*}
  \int_\Omega |\vec{u}|^2f\,dx
     \geq
     \frac{1}{\|\pi\|_{L^\infty(\Omega\times[0,\infty))}}\int_\Omega \pi(x,t)|\vec{u}|^2f\,dx.
 \end{equation*}
 Combining the estimates \eqref{eq:4.52}, \eqref{eq:4.53},
 \eqref{eq:4.55}, and \eqref{eq:4.56}, we obtain the desired bound
 \eqref{eq:4.46} on $\frac{d^2F}{dt^2}[f](t)$.
\end{proof}

Therefore, the  energy estimate \eqref{eq:4.46} takes the form,
\begin{equation}
 \label{eq:4.54}
  \frac{d^2F}{dt^2}[f](t)
  \geq
  -\frac{\lambda}{\|\pi\|_{L^\infty(\Omega\times[0,\infty))}}
  \frac{dF}{dt}[f](t)
  +
  \frac{2\Cr{const:1.3}}{3\Cr{const:1.2}^3}
  \left(
   \frac{dF}{dt}[f](t)
  \right)^3.
\end{equation}

Finally,  we are in the position to show the main result of this Section, Theorem \ref{thm:4.1}.

\begin{proof}%
 [Proof of Theorem \ref{thm:4.1}]
 From the differential inequality \eqref{eq:4.54} and \eqref{eq:4.3},
 we obtain that,
\begin{equation}
 \frac{d}{dt}
  \left(
   \int_\Omega \pi(x,t)|\vec{u}|^2f\,dx
  \right)
  \leq
  -\frac{\lambda}{\|\pi\|_{L^\infty(\Omega\times[0,\infty))}}
  \int_\Omega \pi(x,t)|\vec{u}|^2f\,dx
  +
  \frac{2\Cr{const:1.3}}{3\Cr{const:1.2}^3}
  \left(
   \int_\Omega \pi(x,t)|\vec{u}|^2f\,dx
  \right)^3.
\end{equation}
 Utilizing the same argument as in the proof of the Theorem \ref{thm:3.13}, using similar
 version of the Gronwall's inequality as in the Section~\ref{sec:3}, we can
 show that there exists positive constants
 $\Cr{const:4.6}$, $\Cr{const:4.7}>0$,  such that,  if $\int_\Omega
 \pi(x,t)|\vec{u}|^2f\,dx|_{t=0}\leq \Cr{const:4.6}$, namely,
 $  \int_\Omega \frac{1}{\pi(x,0)}|\nabla(D(x)\log f_0+\phi(x))|^2f_0\,dx
   \leq\Cr{const:4.6},$
 then, we derive \eqref{eq:4.64}, that is, for $t>0$, 
 \begin{equation*}
  \int_\Omega
   \pi(x,t)
   |\vec{u}|^2f\,dx
   \leq
   \Cr{const:4.7}
   e^{-\tilde{\lambda} t},
 \end{equation*}
 where
 $\tilde{\lambda}=\lambda/\|\pi\|_{L^\infty(\Omega\times[0,\infty))}$.
\end{proof}



\section{Conclusion and Numerical Insights}\label{sec:6}

In this work,  we studied several nonlinear Fokker-Planck type
equations with inhomogeneous diffusion and with variable mobility parameters.
These systems appear as a part of grain
growth modeling in polycrystalline materials. Such models
satisfy  energy laws and exhibit special energetic
variational structures as described in the previous sections.

Followed our earlier work on the local
existence and uniqueness of the solution of the Fokker-Planck
system in \cite{epshteyn2022local}, here, we investigated the large
time asymptotic analysis, as well as numerical simulations of these
nonstandard Fokker-Planck systems.
In particular, we reformulated and generalized the classical entropy
method to the nonlinear
Fokker-Planck systems with inhomogeneous diffusion and with variable
mobility parameters (note, the classical entropy method has been previously developed only for the study of the
homogeneous linear Fokker-Planck equations). Due to the limitations of the existing analytical
techniques, our theory has been derived under assumption of the
convex potential and the periodic boundary conditions. However,  our numerical
tests presented below seem to indicate that the developed theoretical results
could be extended to a more general class of models, in particular, to
systems with the non-convex potential and
with no-flux boundary conditions. In addition, the global existence of solutions under various physically relevant boundary conditions was not addressed yet. These important points will be part of our future research
and will require the design of very different analytical
methods. We
will also further extend the study of such Fokker-Planck equations to
 the systems in higher dimensions than those studied in the current
paper. This is especially relevant to
the modeling of the
evolution of the
grain boundary network that undergoes disappearance/critical events, e.g. \cite{epshteyn2021stochastic,Katya-Chun-Mzn4}.
\par As we discussed, in this paper, we seek to show exponential decay of the free energy \eqref{eq:1.2}.
However, because of the nonlinearity, in the inhomogeneous diffusion $D$ case, we have only shown the weaker result of $	\frac{dF}{dt}[f](t) = -D_{\mathrm{dis}}[f](t)$ converges to 0 exponentially, 
as opposed to the stronger conclusions of exponential convergence of
free energy $\energy[f]$ (or the solution $f$ itself) such as in
\appref{sec:5} for the linear Cauchy problem, (see the discussion in \rmref{remark: 3.3}, for example).
We also restrict analysis to the \emph{periodic} \BC. 
However, in applications it is also common to consider the \emph{natural, no-flux} \BC.
We would like to show that, \emph{numerically}, we indeed observe
exponential decay of the free energy,
even in the more general case of inhomogeneous diffusion $D$, as well as variable mobility $\pi$, as in \secref{sec:4}. Moreover, 
there is no significant difference in the exponential decay rates of
the free energy, when the periodic \BC\ is changed to the no-flux \BC, \emph{numerically}.
We also note that in the numerical experiments we can impose much more
relaxed conditions in the parameters than those in our main
Theorem~\ref{thm:4.1}, while observing stronger, more robust
conclusions of the exponential decay than shown in the current theorems.

The numerical experiments are set up as follows.
Consider the domain $\dom = [-1,\ 1]^n$ in $n = 1,\ 2,\ 3$ space dimensions.
We use a uniform grid on $\dom$ of size $N$ in each space dimension, and a uniform time grid of size $N_t$ (with a total of $\Ntot = N^n\times N_t$ grid points).
We use a first-order accurate finite-volume scheme in space, with upwind numerical fluxes and discrete gradients; the time discretization is done using backward Euler method.
In the numerical results, the free energy \eqref{eq:1.2} is measured discretely using the cell-average values from the scheme.

To be consistent with the theoretical assumptions, we set the parameters to be smooth, bounded, and  periodic in space.
In $n = 1$ dimension, set the potential to be, 
\begin{equation}\label{phi}
	\phi(x) = 1 + \frac{1}{4} \sin^2 {\frac{k_p\pi}{2} x}.
\end{equation}
Note that  $\phi''(x) = \frac{(k_p\pi)^2}{8}\cos k_p \pi x,$ so \eqref{phi} is in general \emph{not} convex on $[-1,\ 1]$,
and in particular does not satisfy the strict convexity condition of Theorems~\ref{thm:2.9}, \ref{thm:3.13} and \ref{thm:4.1}.
For example, we choose $k_p = 2$, and the potential \eqref{phi} is
plotted in \figref{fig: par_1d}, top left.
We will present here numerical results for non-convex potentials,
since we do not observe numerically any dependence on the convexity of
$\phi$ (we also conducted numerical tests with the convex potential
$\phi$,  and obtained a very similar results to the results presented below).
Define a parameter $\gamma$ as,
\begin{equation}
	\gamma(x,\ t) = (1+\cos^2 {\pi x})(1+\frac{1}{2}\sin {10 t}),
\end{equation}
and set the mobility
\begin{equation}\label{PI}
	\pi(x,\ t) = \frac{1}{\gamma(x,\ t)},
\end{equation}
as plotted in \figref{fig: par_1d}, top right.
Note that \eqref{PI} is smooth and bounded, strongly positive, and both $\pi_x$ and $\pi_t$ are bounded. In particular, the mobility \eqref{PI} satisfies the conditions of Propositions~\ref{prop:1.1} and \ref{prop:1.5}. 
In the homogeneous case set the diffusion coefficient $D = 1,$ whereas
in the inhomogeneous case we consider,
\begin{equation}\label{D}
	D(x) = D_1(x) \equiv 1-\frac{1}{2}\sin^2 {2\pi x} = \frac{1}{4}(3 + \cos {4\pi x} ),
\end{equation} 
as plotted in \figref{fig: par_1d}, bottom left.
The function \eqref{D} is smooth and bounded, strongly positive, and $D'$ is also bounded. In particular, \eqref{D} satisfies the conditions of Propositions~\ref{prop:1.1} and \ref{prop:1.5}. 
Note that $D_1$ \eqref{D} only contains a single mode of frequency.
Thus,  for a more general and interesting results, we further 
consider a positive, smooth, even, periodic function, 
\begin{equation}\label{D_modes}
	D(x) = D_M(x) \equiv  1 + \sum_{m=1}^M A_n (\cos \frac{m\pi x}{2} + 1) ,
\end{equation}
where $A_m \geq 0$ is the coefficient of each mode of frequency. 
For instance, if we select $M = N/2 = 100$, where $N = 200$ is the size of the grid in space,
and the coefficients $A_m = 0.01$ for $m = 1,\ \cdots,\ M,$ then the function $D$ is plotted in \figref{fig: par_1d}, bottom right.
Note that the resulting $D_M$ is much more oscillatory, and its gradient $D'$ is generally not bounded by 
some given constant; that is,
the \IC\ 
of Theorem~\ref{thm:3.13} or \ref{thm:4.1} is generally not satisfied.
Finally, for simplicity, set the Gaussian \IC\
\begin{equation}\label{IC}
	f(x,\ 0) = f_0(x) = \frac{1}{\sqrt{2\pi \sigma^2}} e^{-\frac{x^2}{2\sigma^2}},\quad \sigma^2 = 0.01.
\end{equation}
Note that $f_0$ is very close to 0 near the boundaries, and does not satisfy the strong positivity condition of Proposition~\ref{prop:1.5}, and in turn that of Lemmas~\ref{lemma3.1} and \ref{lem:4.2}.
Also, since \eqref{IC} is defined independent of the parameters $\phi,\ \pi$ and $D$, $D_{\mathrm{dis}}[f_0]$ is generally not bounded by some given constant; that is,
the condition in \eqref{eq:3.58} or \eqref{cond_par:thm4}
of Theorem~\ref{thm:3.13} or \ref{thm:4.1} is generally not satisfied.

In higher $n = 2,\ 3$ dimensions, the parameters $\phi,\ \gamma,\ f_0$ are set to be the ``tensor-products" in $\MT{x}$ of their respective one-dimensional counterparts.
In $n = 2$ dimensions, take the potential,
\begin{equation}\label{phi_2d}
	\phi(x,\ y) = (1 + \frac{1}{4} \sin^2 {\frac{k_p^x\pi}{2} x} ) (1 + \frac{1}{4} \sin^2 {\frac{k_p^y\pi}{2} y}),
\end{equation}
and, in particular,  we choose $k_p^x = 1,\ k_p^y = 2$; 
and the mobility $\pi(x,\ y,\ t) = 1/\gamma(x,\ y,\ t),$ where, 
\begin{equation}
	\gamma(x,\ y,\ t) = (1+\cos^2 {\pi x})(1+\cos^2 {2\pi y})(1+\frac{1}{2}\sin {10 t}).
\end{equation}
Similarly, in $n = 3$ dimensions, 
\begin{equation}\label{phi_3d}
	\phi(x,\ y,\ z) = (1 + \frac{1}{4} \sin^2 {\frac{k_p^x\pi}{2} x} ) (1 + \frac{1}{4} \sin^2 {\frac{k_p^y\pi}{2} y})(1 + \frac{1}{4} \sin^2 {\frac{k_p^z\pi}{2} z}),
\end{equation}
and in particular we choose $k_p^x = 1,\ k_p^y = 2,\ k_p^z = 3$; 
and $\pi(x,\ y,\ z,\ t) = 1/\gamma(x,\ y,\ z,\ t),$ where, 
\begin{equation}
	\gamma(x,\ y,\ z,\ t) = (1+\cos^2 {\pi x})(1+\cos^2 {2\pi y})(1+\cos^2 {3\pi z})(1+\frac{1}{2}\sin {10 t}).
\end{equation}
We extend the single-mode inhomogeneous diffusion coefficient $D_1$ \eqref{D} in the same way.
In two dimensions, consider the separable,
\begin{equation}\label{D_2d}
	D_1(x,\ y) \equiv (1-\frac{1}{2}\sin^2 {\pi x}) (1-\frac{1}{2}\sin^2 {3\pi y}),
\end{equation} 
and similarly in three dimensions, 
\begin{equation}\label{D_3d}
	D_1(x,\ y,\ z) \equiv (1-\frac{1}{2}\sin^2 {\pi x}) (1-\frac{1}{2}\sin^2 {3\pi y}) (1-\frac{1}{2}\sin^2 {4\pi z}).
\end{equation} 
To extend the multi-mode diffusion coefficient $D_M$ \eqref{D_modes},
consider, in two dimensions, the more interesting non-separable function,
\begin{equation}\label{D_2d_modes}
	D_{M_1 M_2}(x,\ y) \equiv  1 + \sum_{m_1=1}^{M_1}\sum_{m_2=1}^{M_2} A_{m_1 m_2}\, (\cos \frac{m_1\pi x}{2}\, \cos \frac{m_2\pi y}{2} + 1) ,
\end{equation} 
where $A_{m_1 m_2} \geq 0$ is the coefficient of each frequency. 
In particular, we can choose $M_1 = N/2,\ M_2 = N/4$,
and $A_{m_1 m_2} = 0.01$ if $m_1 < m_2$, and 0 otherwise.
The resulting function $D_M$ is plotted in \figref{fig: D_2d_modes}, for 
$(M_1,\ M_2) = (20,\ 10)$ for $N = 40$ and $(M_1,\ M_2) = (40,\ 20)$ for $N = 80$, respectively.
Similarly, in three dimensions, consider 
\begin{equation}\label{D_3d_modes}
	D_{M_1 M_2 M_3}(x,\ y,\ z) \equiv  1 + \sum_{m_1=1}^{M_1}\sum_{m_2=1}^{M_1} \sum_{m_3=1}^{M_1} A_{m_1 m_2 m_3}\, 
	(\cos \frac{m_1\pi x}{2}\, \cos \frac{m_2\pi y}{2}\, \cos \frac{m_3\pi z}{2} + 1) ,
\end{equation} 
where $A_{m_1 m_2 m_3} \geq 0$. In particular, for $N = 20$,
we can choose 
$M_1 = N/2=10,\ M_2 = N/2-2 = 8$  
$M_3 = 4$,
and
$A_{m_1 m_2 m_3} = 0.01$ if $m_1 \geq m_2 \geq m_3$,
and 0 otherwise.

The numerical decays of free energy are presented in \figref{fig: FE_1d}, \figref{fig: FE_2d} and \figref{fig: FE_3d}, for one, two, and three dimensions, respectively. We can draw several important observations from the numerical results:
\begin{itemize}[nosep]
	\item For both the homogeneous $D = 1$ (linear) and the inhomogeneous $D(\MT{x})$ (nonlinear) cases, the \emph{free energy decays exponentially with either periodic or no-flux \BC}, until the numerical results hit round-off errors.
	\item We can observe some discrepancy in the exponential decay
          rates of the free energy between periodic and no-flux \BC s, but this is due to numerical errors, as seen to be greatly reduced when the mesh is refined (for example, from \figref{fig: FE_2d_C} to \figref{fig: FE_2d} in two dimensions, and from \figref{fig: FE_3d_C} to \figref{fig: FE_3d} in three dimensions).
	\item As with the theoretical results, the exponential decay
          of the free energy is observed in all tested \emph{space dimensions $n = 1,\ 2,\ 3$}.
	\item The numerical results do not seem to rely on the restricted conditions on the parameters as given in the main Theorems~\ref{thm:2.9}, \ref{thm:3.13} and \ref{thm:4.1}, such as the \emph{convexity of the potential $\phi$}, the strong positivity of the \IC\ $f_0$, and the restricted bound on the gradient of the diffusion coefficient $\grad{D}$.
	\item In particular, we compare the homogeneous diffusion
          coefficient $D=1$, the single-mode $D_1$, and the multi-mode
          $D_M$. The numerical free energy decays exponentially fast
          regardless of how oscillatory $D$ is. We observe that the
          free energy decays slower in the $D_1$ case than $D=1$,
          while much faster in the $D_M$ case (as expected due to the
          magnitude of $D$ in each case).		
\end{itemize}

\begin{figure}[h!]
	\newcommand{\figWidth}{.45\linewidth}
	\begin{center}	
		\plotFig{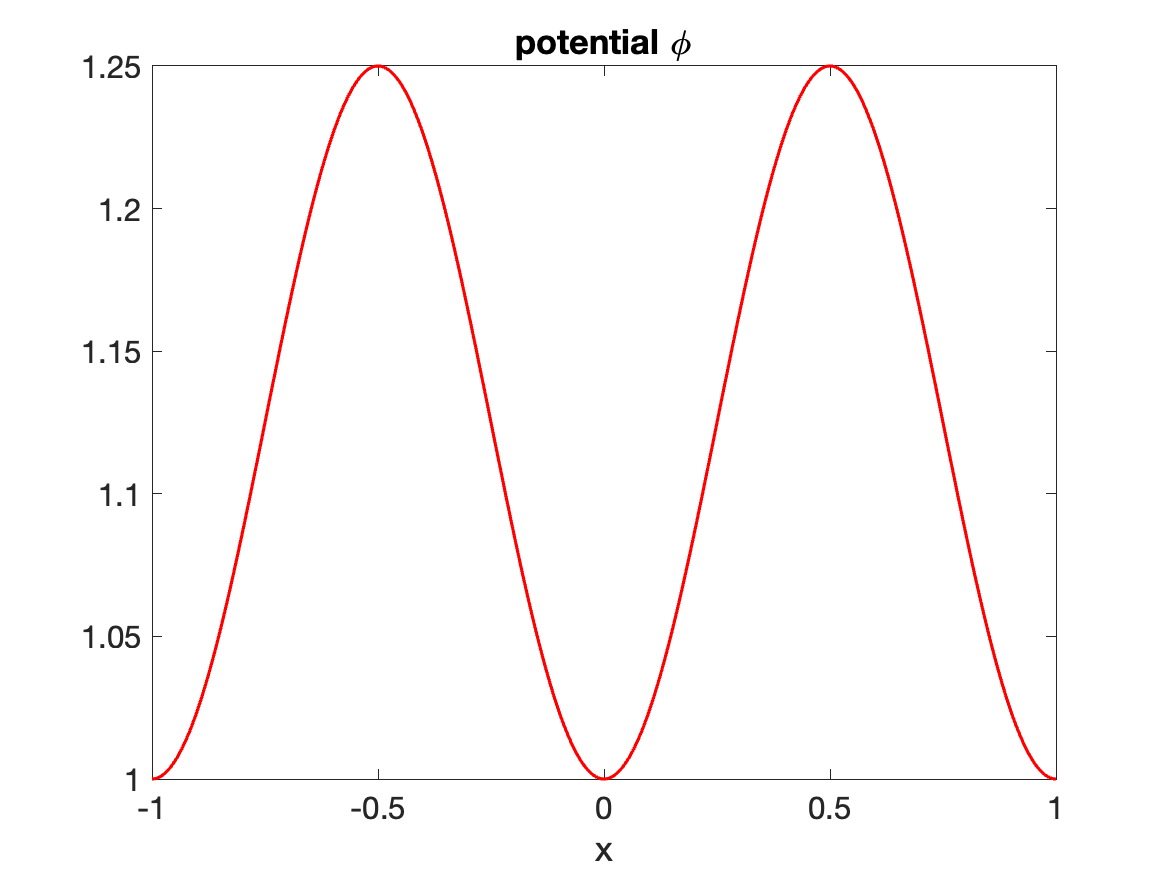}{\figWidth}
		\plotFig{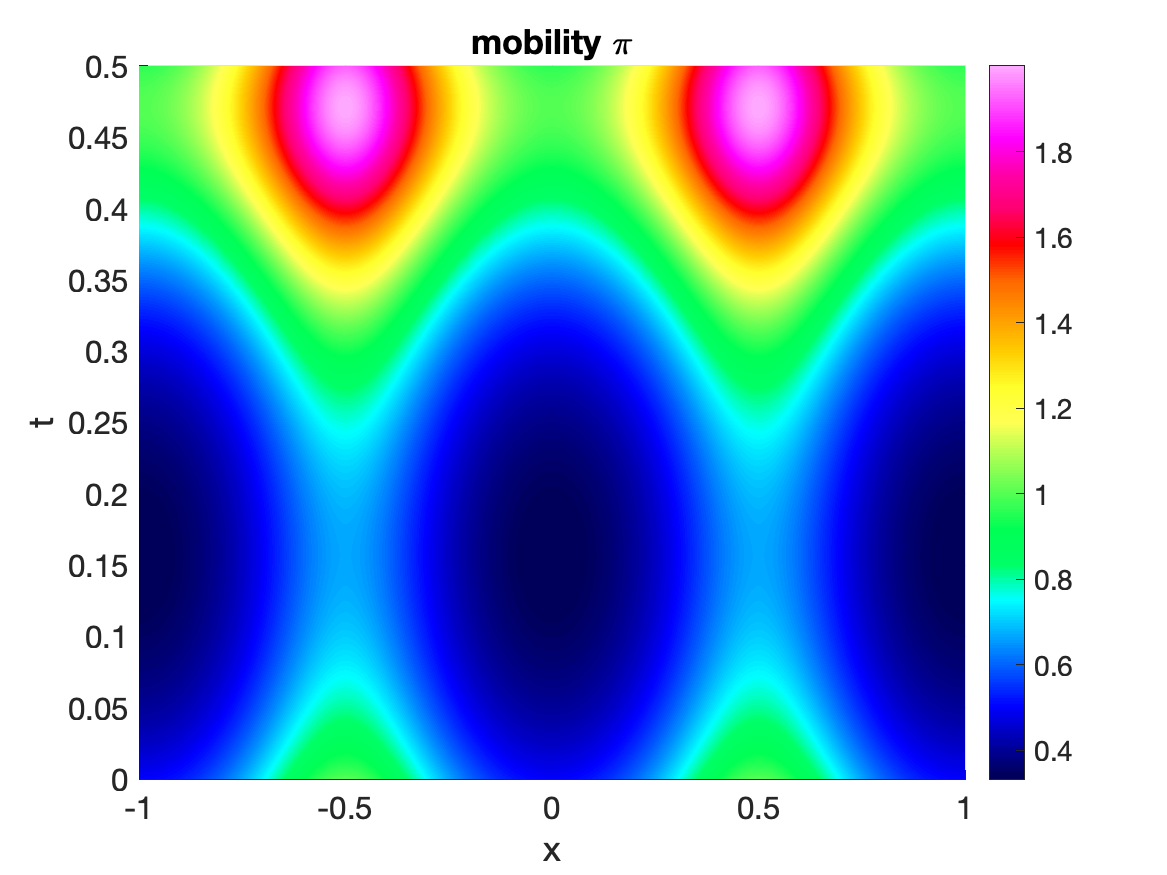}{\figWidth}	
		\plotFig{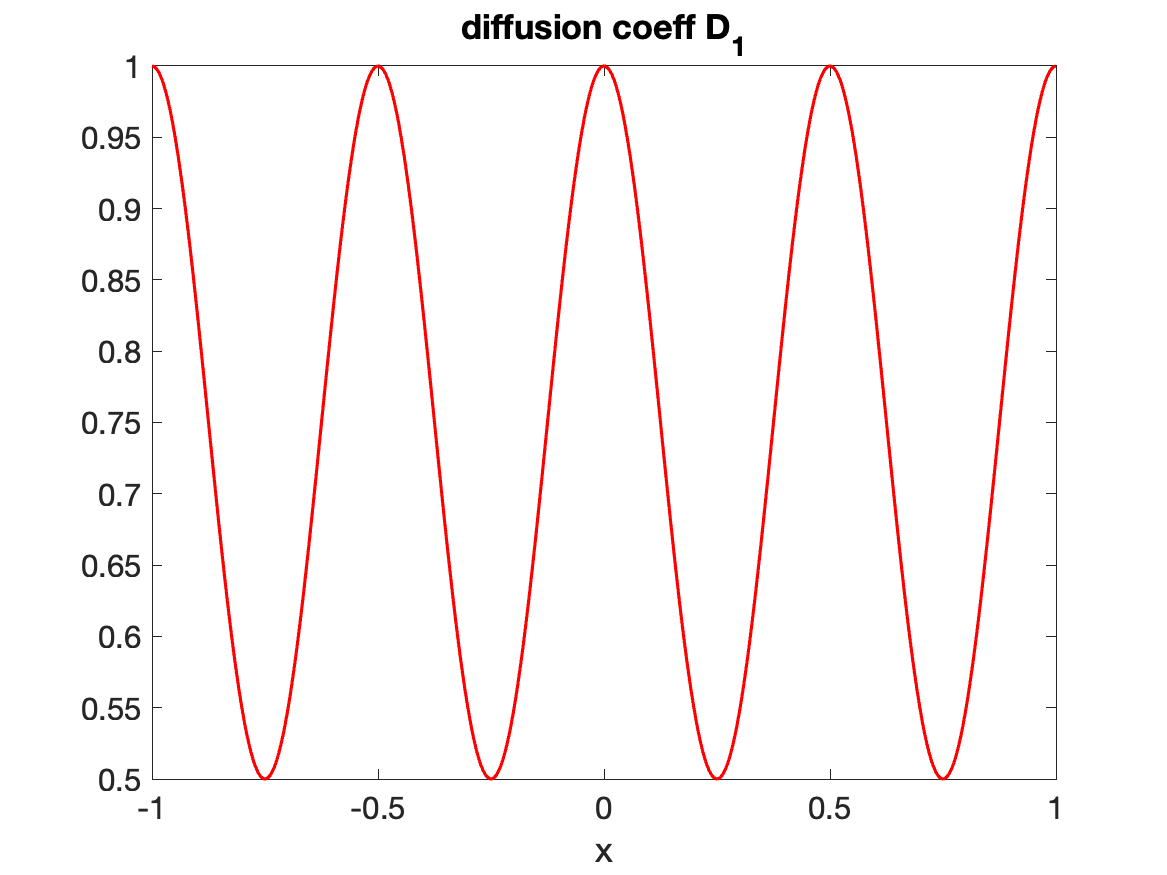}{\figWidth}	
		\plotFig{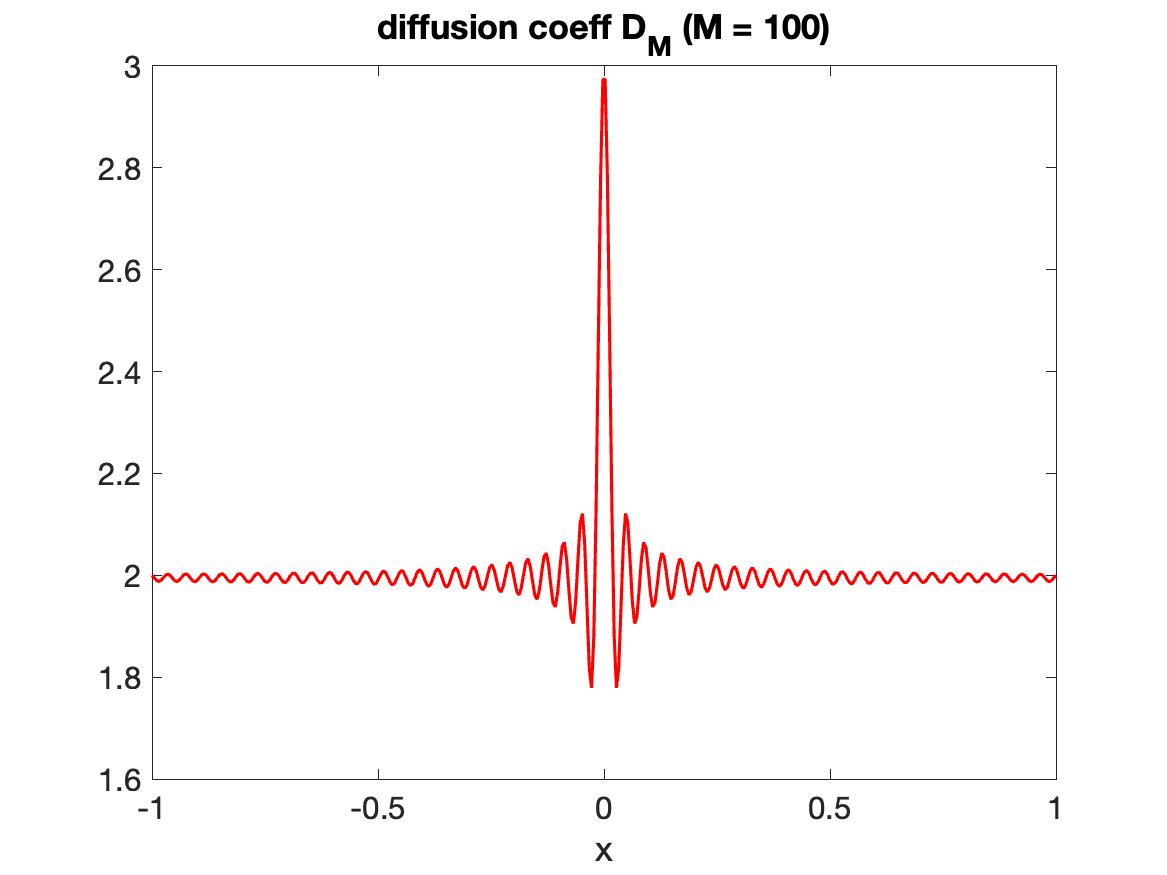}{\figWidth}		
	\end{center}
	\caption{Periodic parameters in one dimension on $\dom = [-1,\ 1]$.
		Top left: potential $\phi(x)$ \eqref{phi} with $k_p = 2$.
		Top right: mobility $\pi(x,\ t)$ \eqref{PI}, $t \in [0,\ 0.5]$.
		Bottom: inhomogeneous diffusion coefficient $D(x)$; left: single-mode $D_1(x)$ \eqref{D}; 
		right: multi-mode $D_M(x)$ \eqref{D_modes} with $M = 100$ modes of oscillation, and $A_m = 0.01$.
	 }
	\label{fig: par_1d}
\end{figure}

\begin{figure}[h!]
	\newcommand{\figWidth}{.45\linewidth}
	\begin{center}
		\plotFig{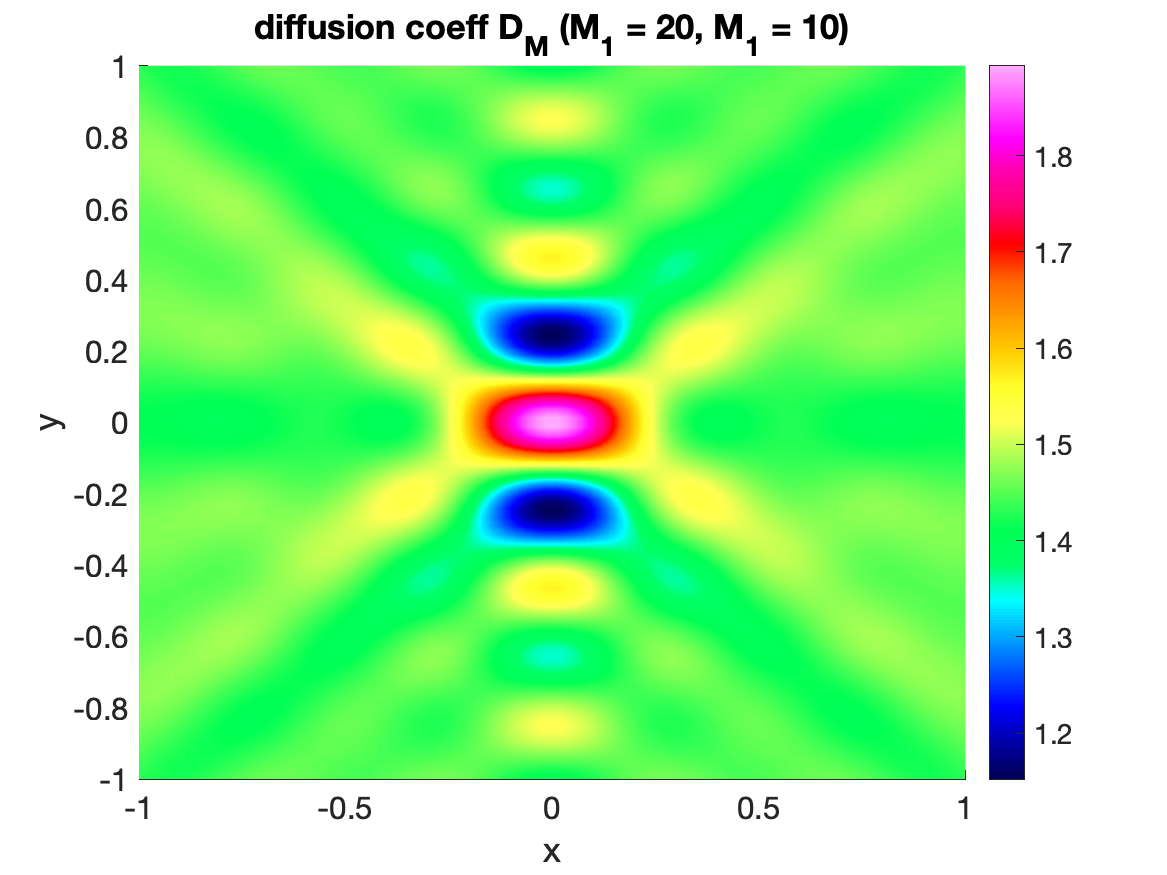}{\figWidth}		
		\plotFig{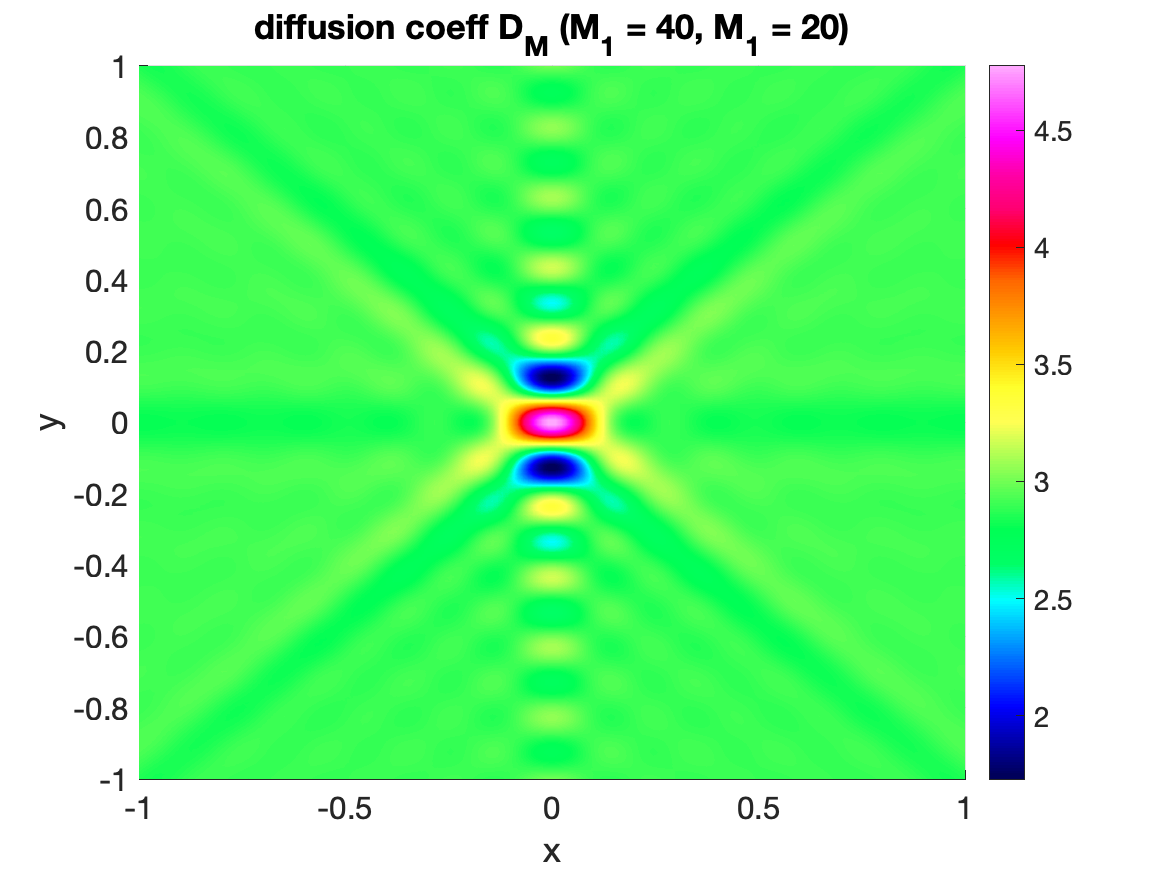}{\figWidth}	
	\end{center}
	\caption{Multi-mode inhomogeneous diffusion coefficient $D_{M_1,\ M_2}(x,\ y)$ \eqref{D_2d_modes}, with $(M_1,\ M_2)$ modes of oscillation, and $A_{m_1 m_2} = 10^{-2}$ if $m_1 < m_2$, and 0 otherwise. Left: $(M_1,\ M_2) = (20,\ 10)$; right: $(M_1,\ M_2) = (40,\ 20)$. }
	\label{fig: D_2d_modes}
\end{figure}

\begin{figure}[h!]
	\newcommand{\figWidth}{.45\linewidth}
	\begin{center}	
		\plotFig{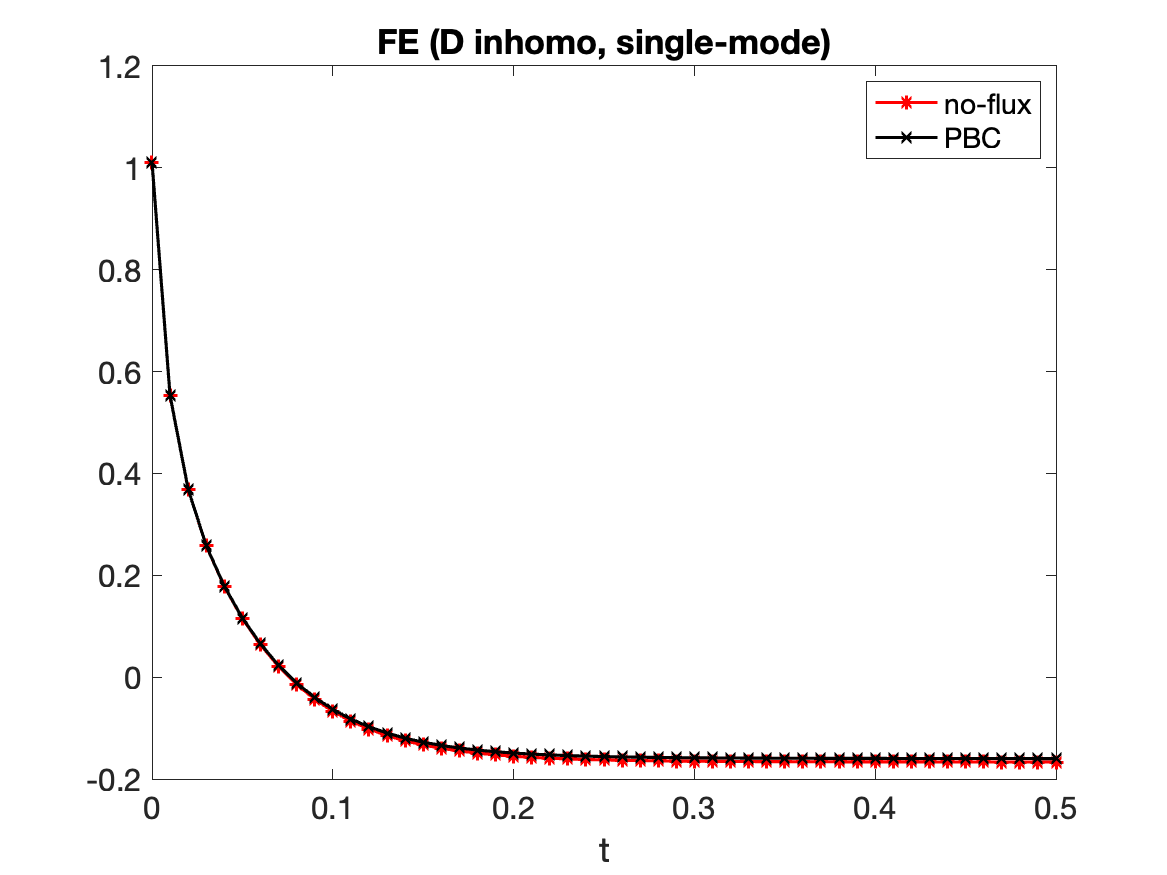}{\figWidth}
		\plotFig{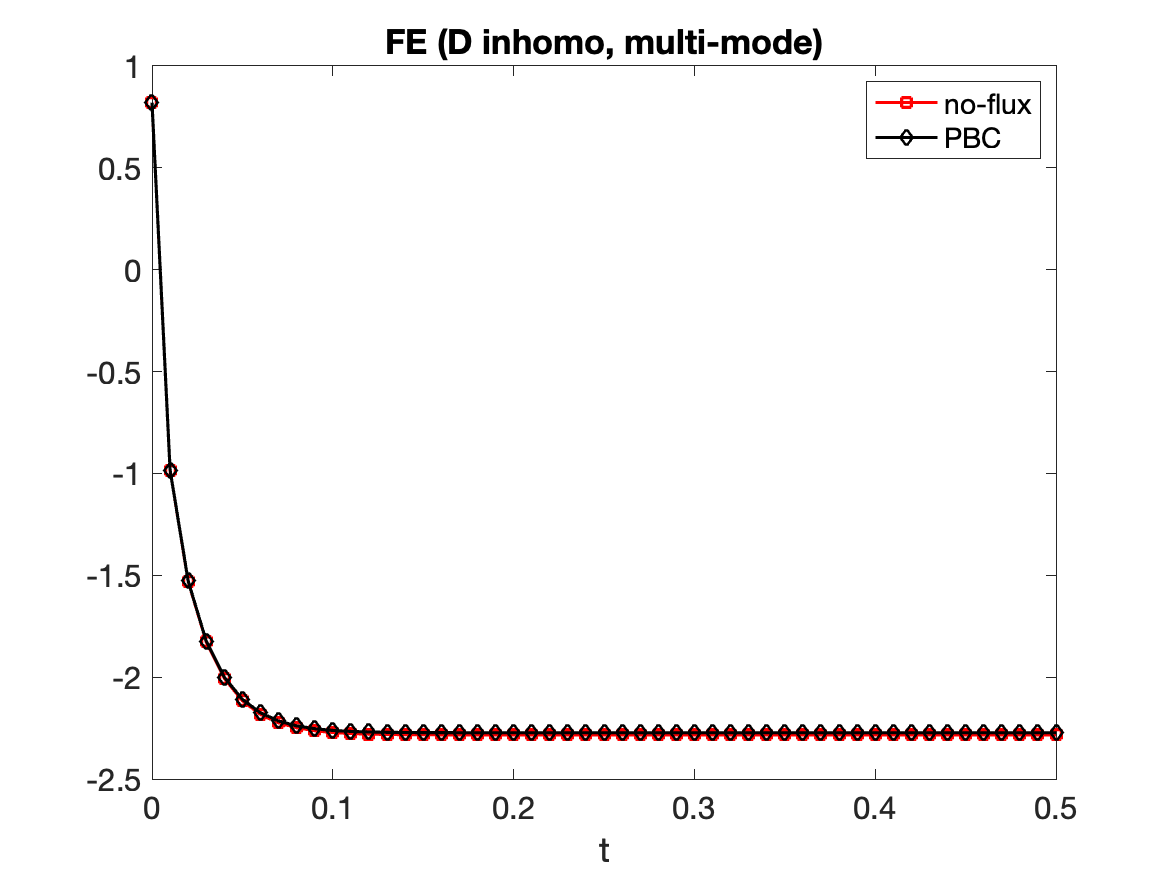}{\figWidth}
		\plotFig{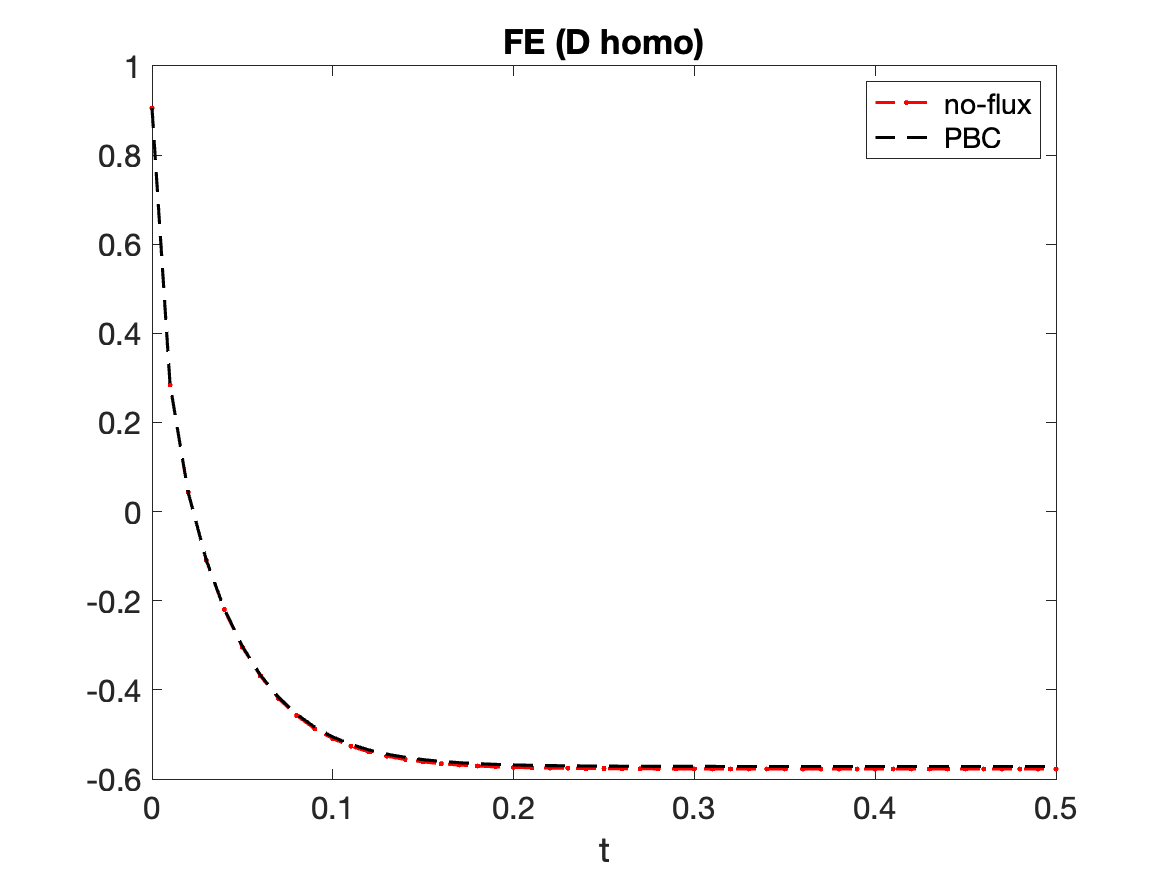}{\figWidth}
		\plotFig{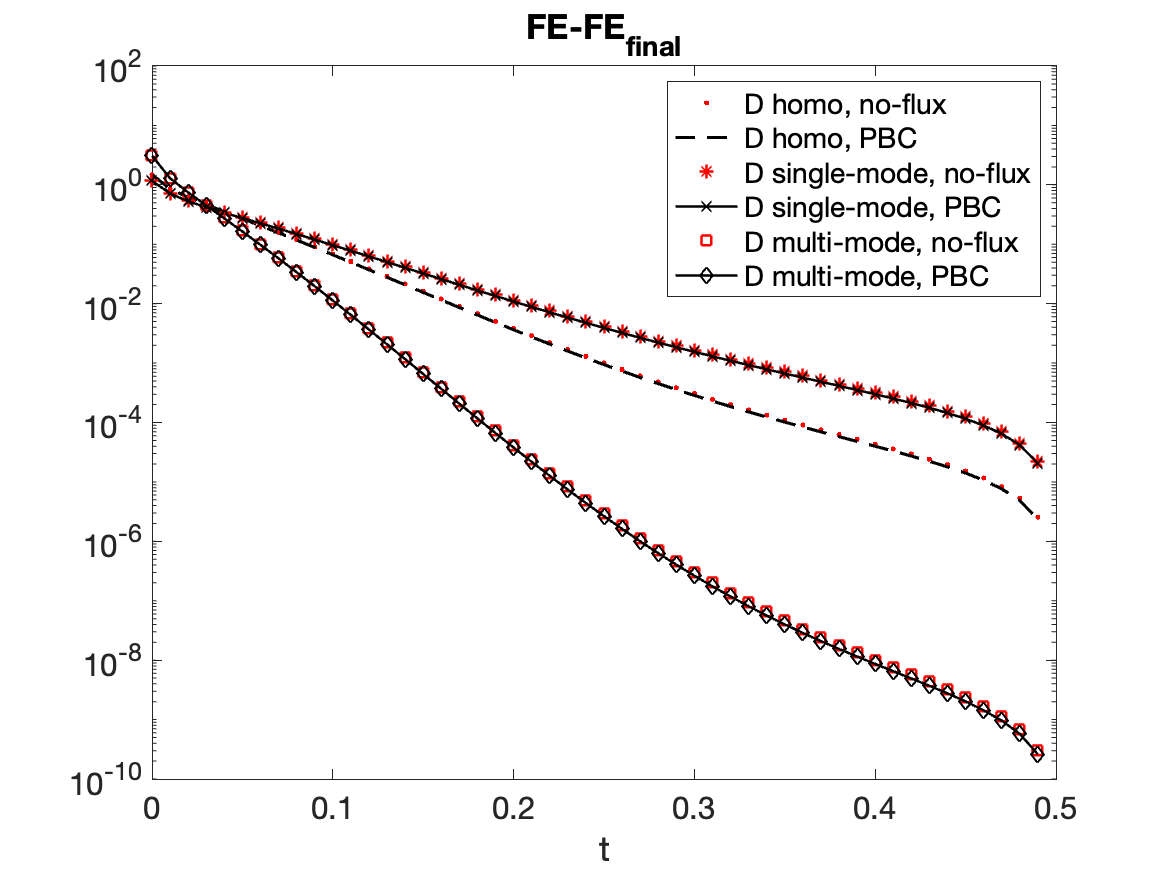}{\figWidth}
	\end{center}
	\caption{Exponential decay of free energy (FE), in one dimension, 
	comparing no-flux (red) against periodic (black) \BC.
	Top: inhomogeneous $D(x)$; 
	left: single-mode $D_1$ \eqref{D} (\figref{fig: par_1d}, bottom left); 
	right: multi-mode $D_M$ \eqref{D_modes}, with $M = N/2 = 100,\ A_m = 0.01$ (\figref{fig: par_1d}, bottom right).
	Bottom left: homogeneous $D=1$. 
	Bottom right: direct comparison of the exponential decay rates. ($\Ntot \approx 200\times 50.$)}
	\label{fig: FE_1d}
\end{figure}


\begin{figure}[h!]
	\newcommand{\figWidth}{.45\linewidth}
	\begin{center}	
		\plotFig{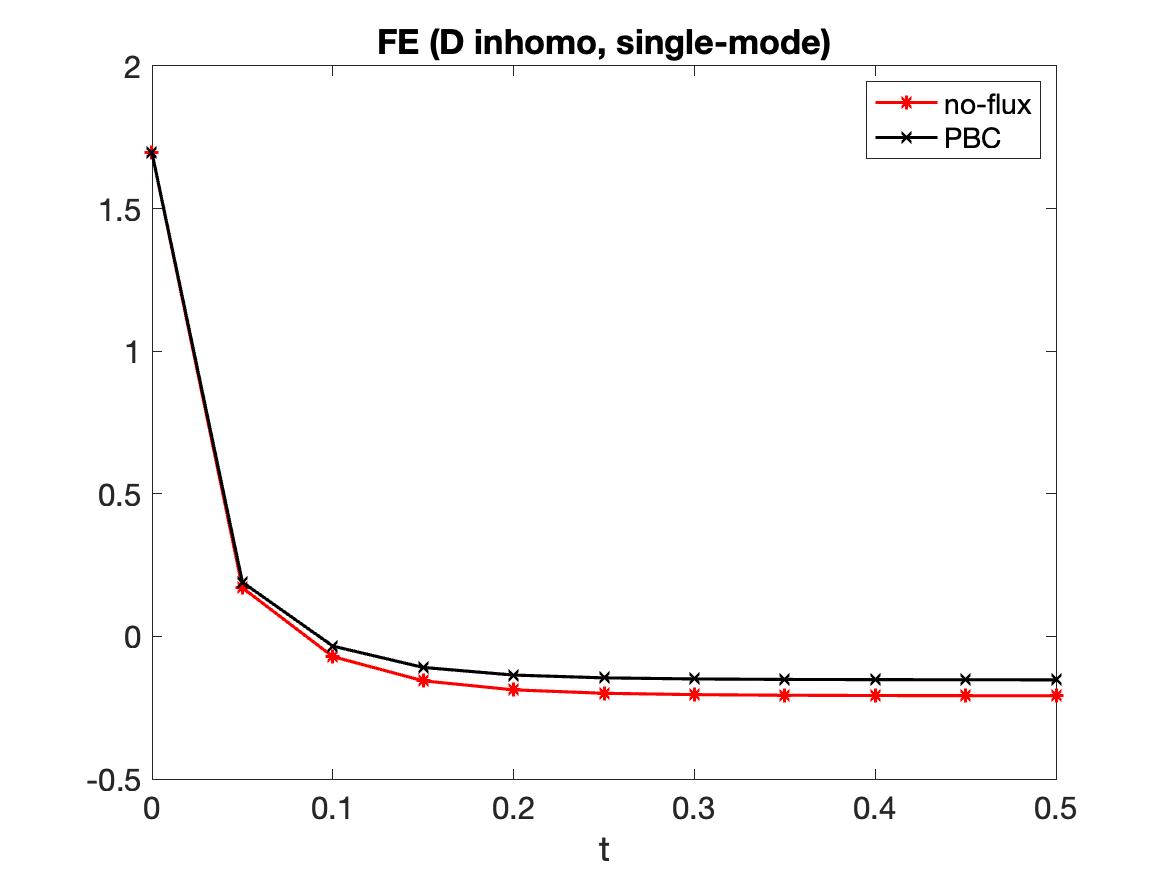}{\figWidth}
		\plotFig{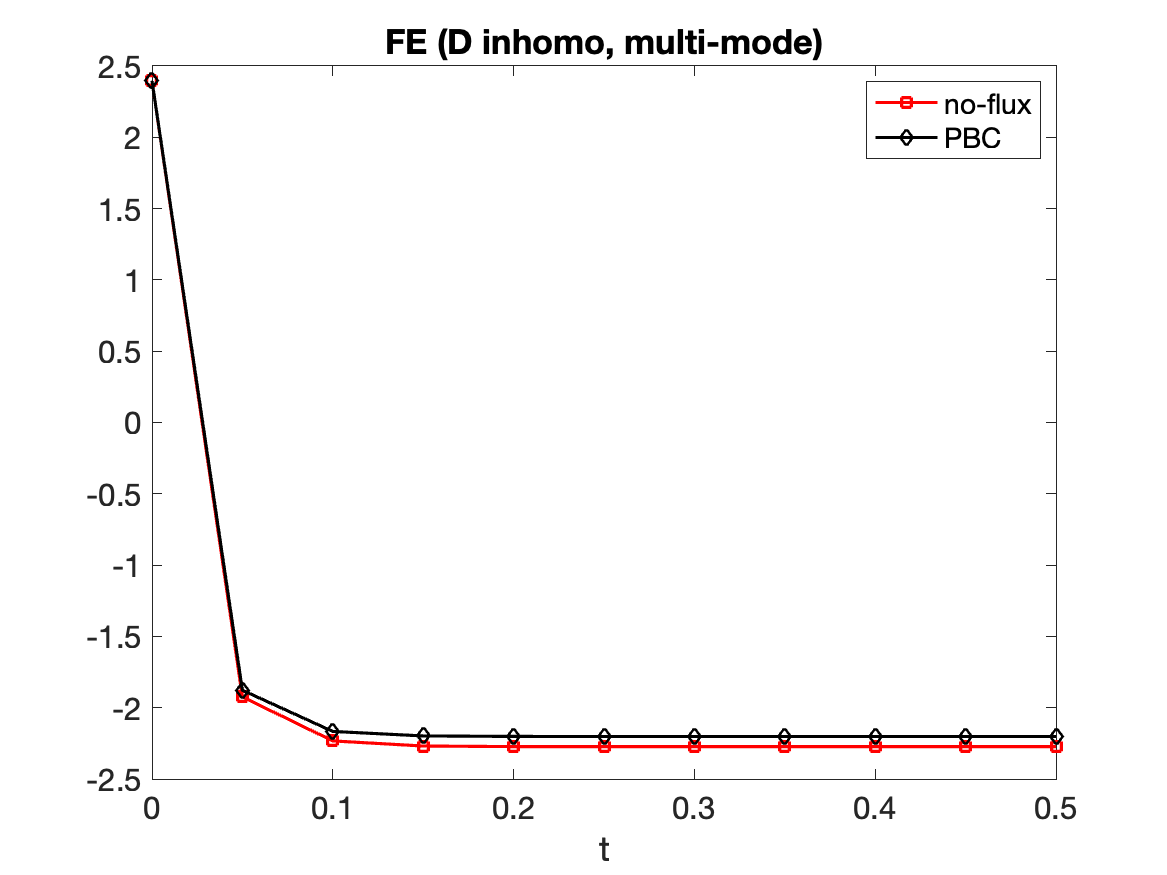}{\figWidth}
		\plotFig{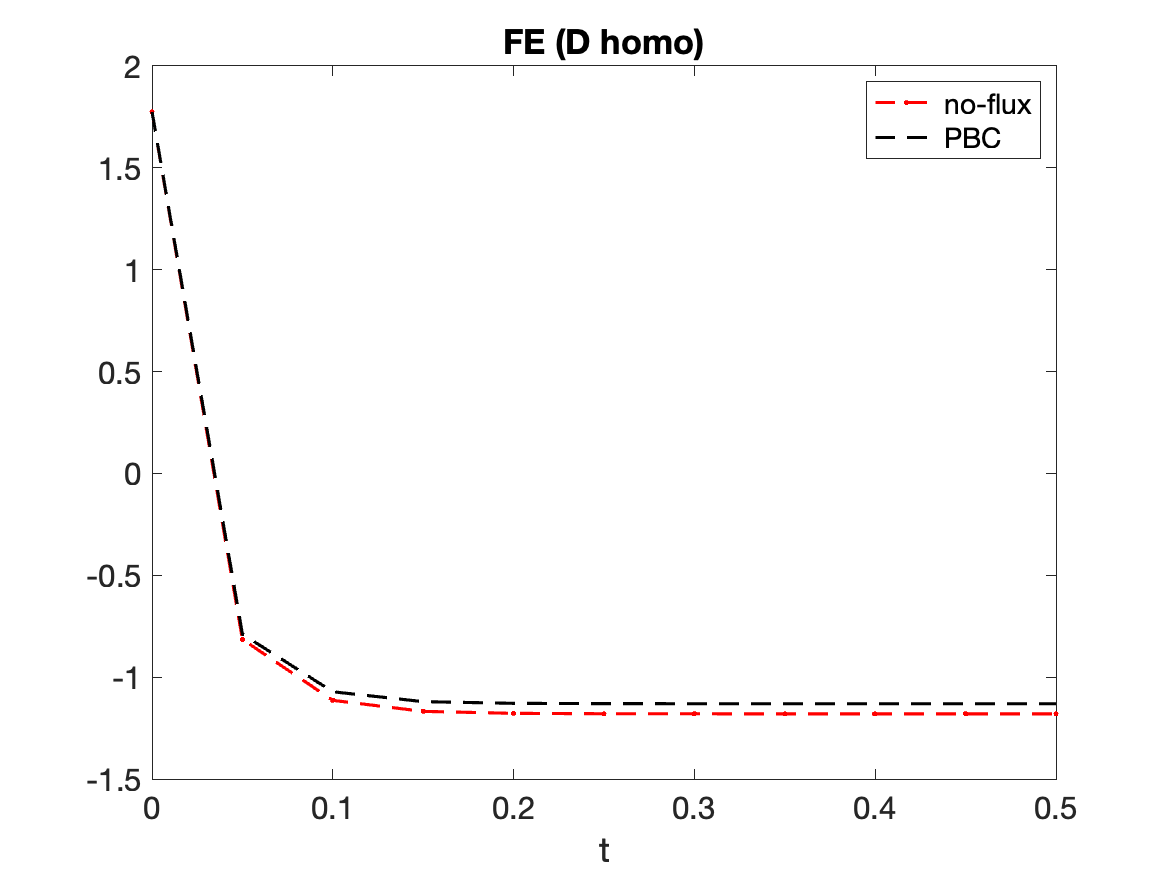}{\figWidth}
		\plotFig{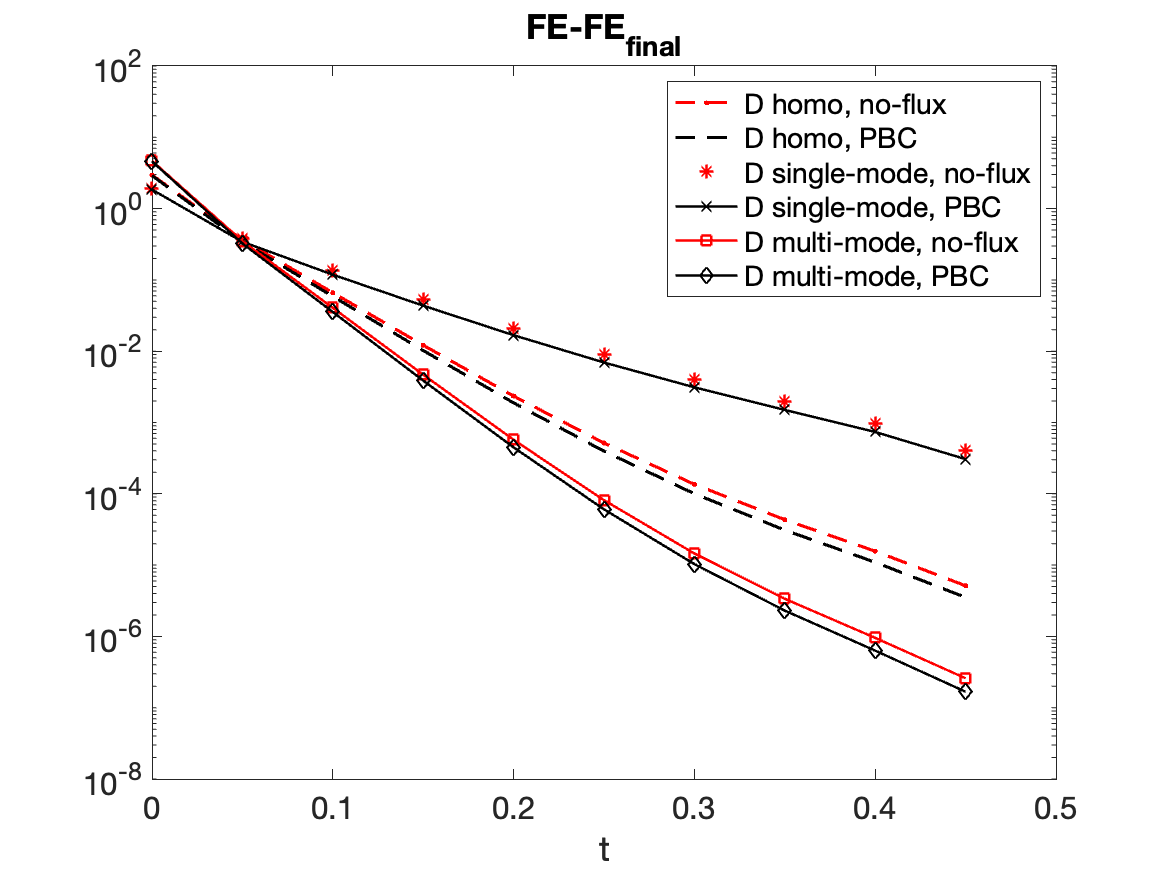}{\figWidth}
	\end{center}
	\caption{Exponential decay of free energy (FE), in two dimensions, 
		comparing no-flux (red) against periodic (black) \BC.
		Top: inhomogeneous $D(\MT{x})$; 
		left: single-mode $D_1$ \eqref{D_2d}; 
		right: multi-mode $D_{M_1,\ M_2}$ \eqref{D_2d_modes} with $(M_1,\ M_2) = (20,\ 10),\ A_{m_1 m_2} = 0.01$ if $m_1 < m_2$, and 0 otherwise (\figref{fig: D_2d_modes}, left).
		Bottom left: homogeneous $D=1$. 
		Bottom right: direct comparison of the exponential decay rates. ($\Ntot \approx 40^2\times 10.$)}
	\label{fig: FE_2d_C}
\end{figure}

\begin{figure}[h!]
	\newcommand{\figWidth}{.45\linewidth}
	\begin{center}	
		\plotFig{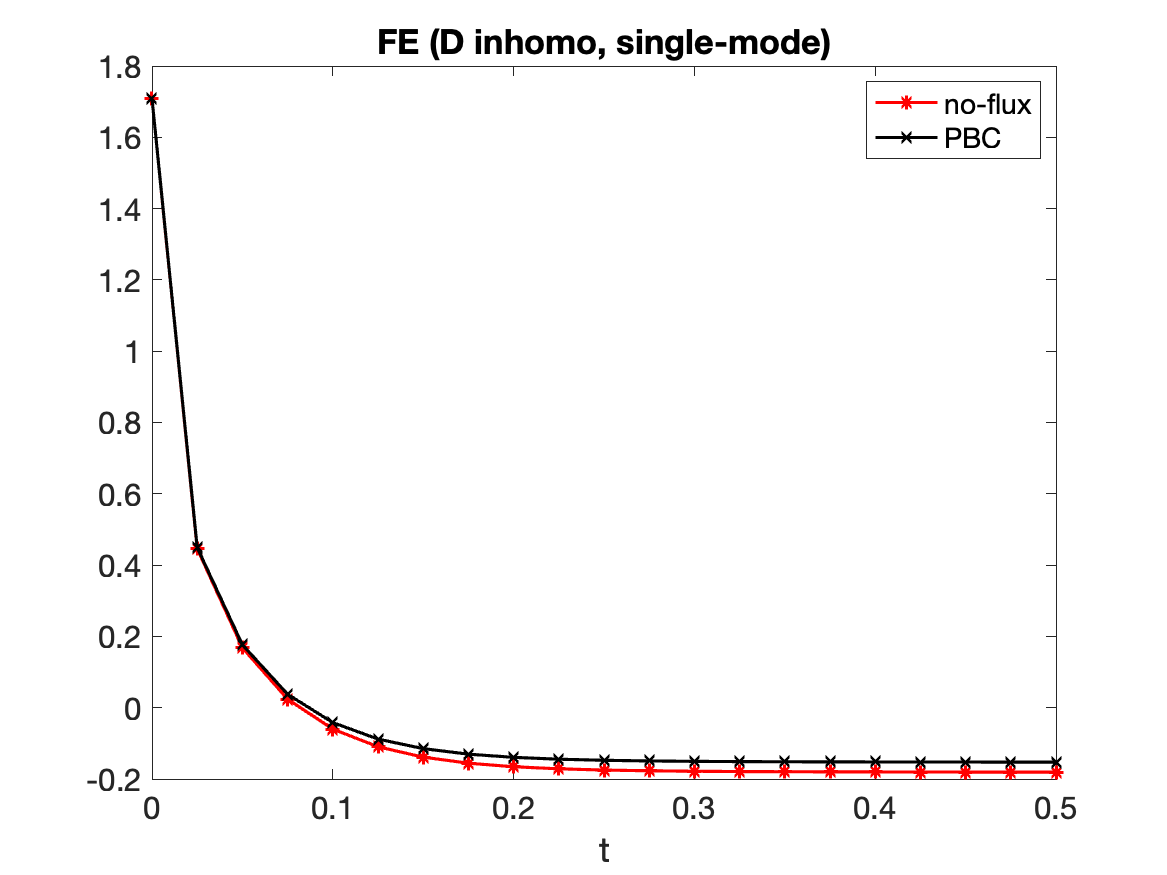}{\figWidth}
		\plotFig{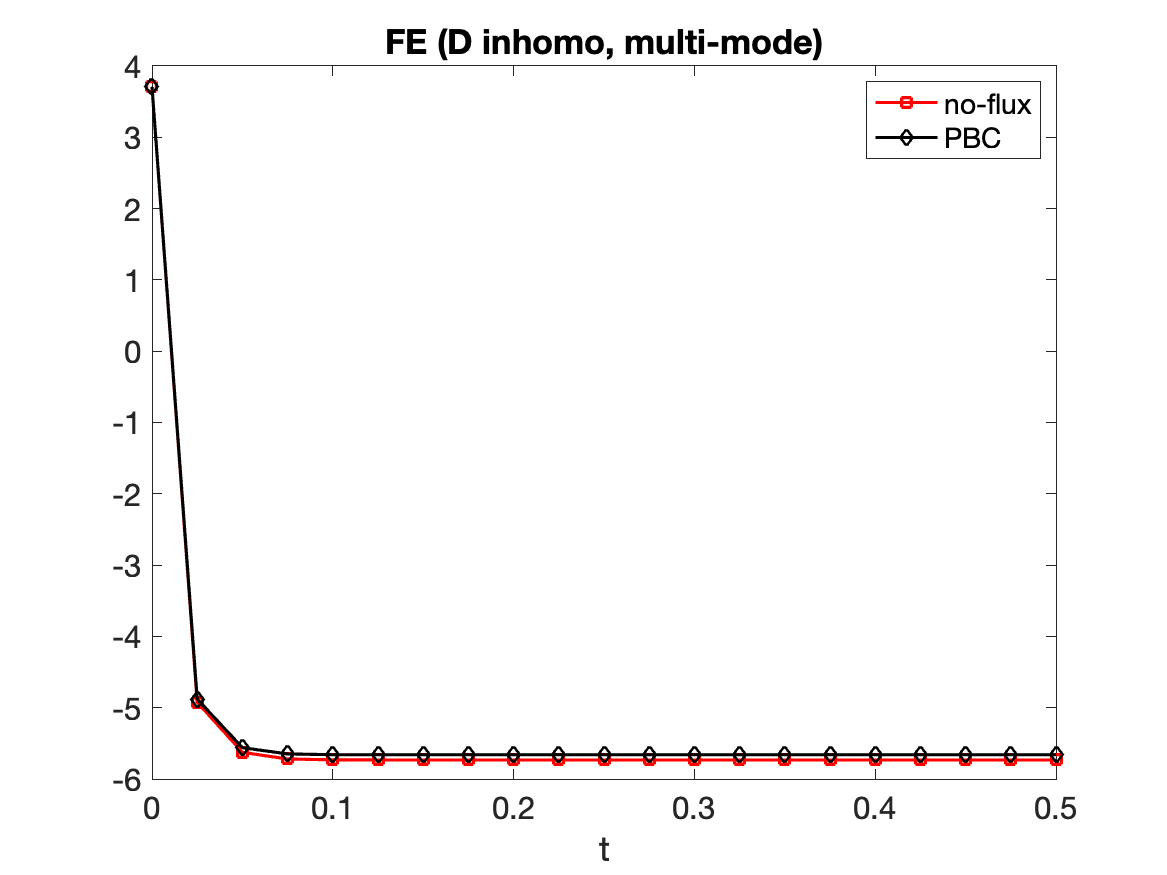}{\figWidth}
		\plotFig{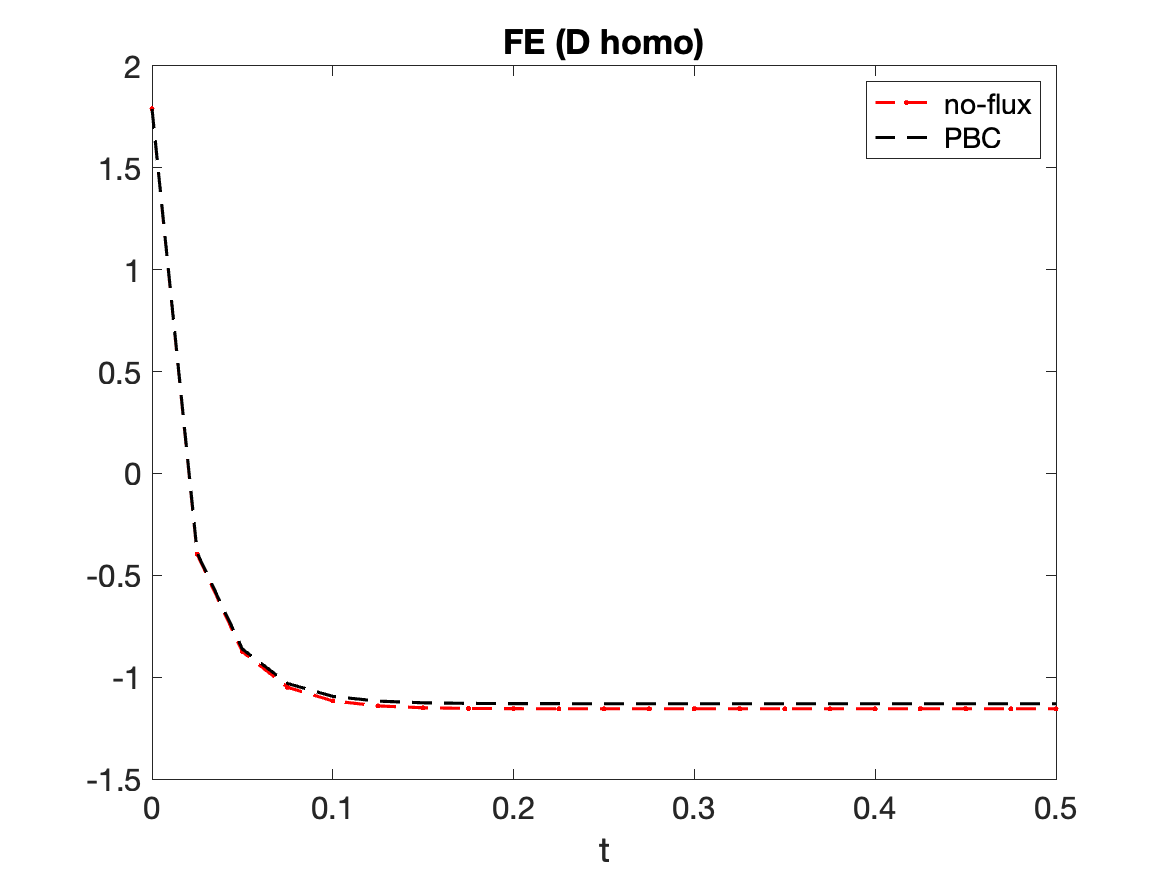}{\figWidth}
		\plotFig{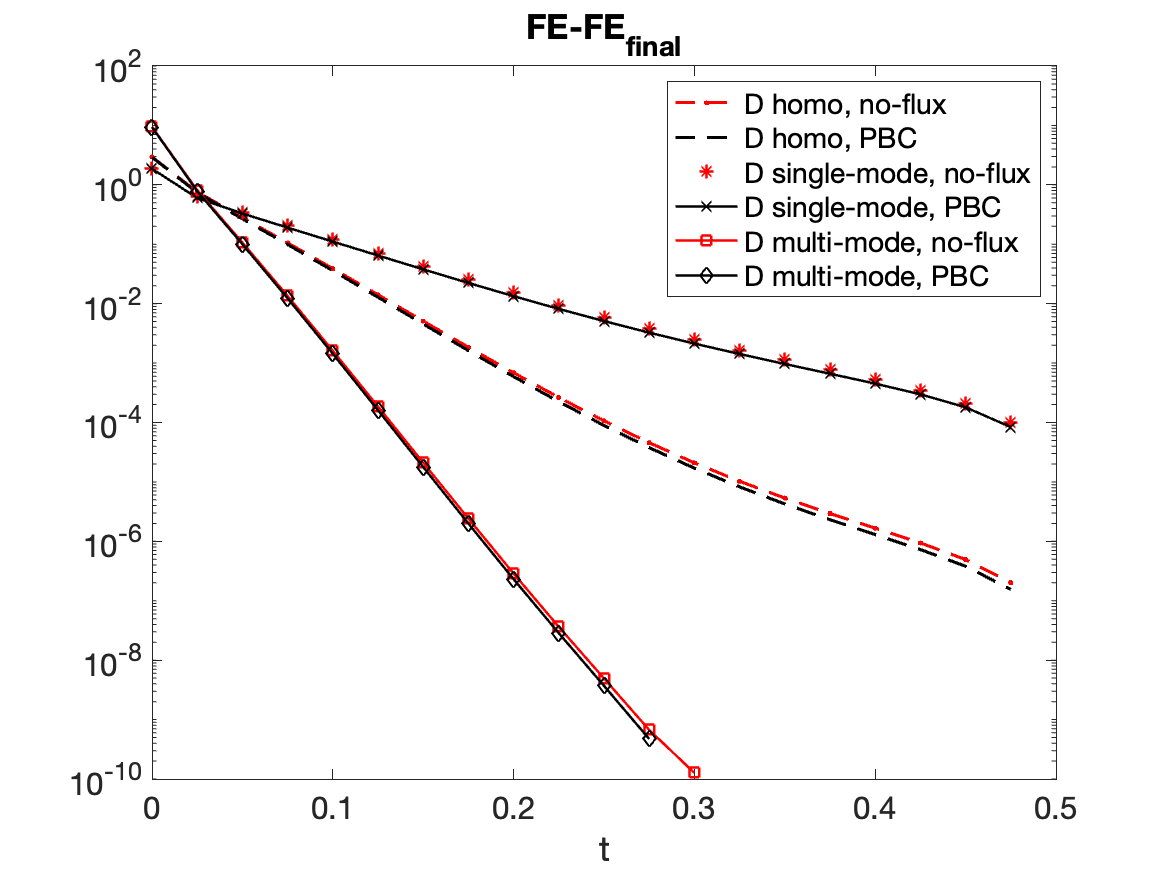}{\figWidth}
	\end{center}
	\caption{Exponential decay of free energy (FE), in two dimensions, 
		comparing no-flux (red) against periodic (black) \BC.
		Top: inhomogeneous $D(\MT{x})$; 
		left: single-mode $D_1$ \eqref{D_2d}; 
		right: multi-mode $D_{M_1,\ M_2}$ \eqref{D_2d_modes} with $(M_1,\ M_2) = (40,\ 20),\ A_{m_1 m_2} = 0.01$ if $m_1 < m_2$, and 0 otherwise (\figref{fig: D_2d_modes}, right).
		Bottom left: homogeneous $D=1$. 
		Bottom right: direct comparison of the exponential decay rates. ($\Ntot \approx 80^2\times 20.$)}
	\label{fig: FE_2d}
\end{figure}

\begin{figure}[h!]
	\newcommand{\figWidth}{.45\linewidth}
	\begin{center}	
		\plotFig{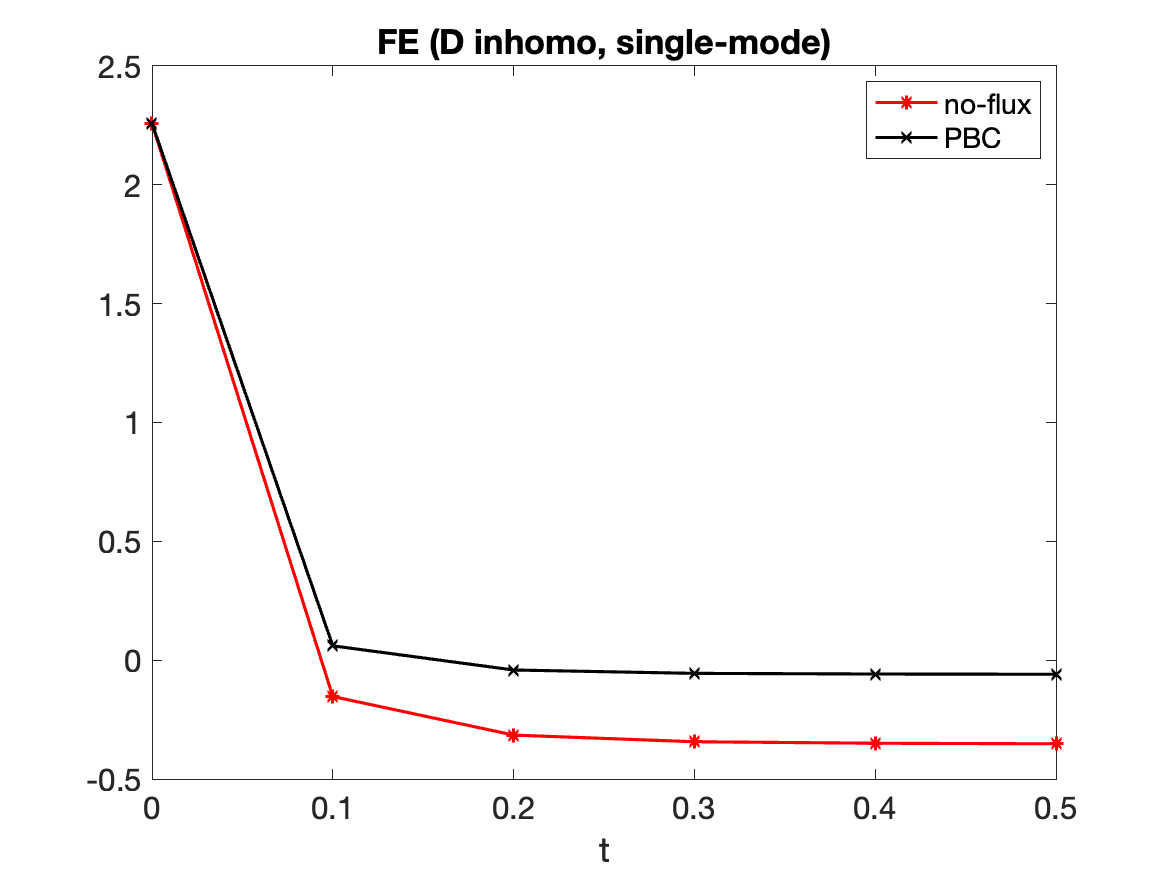}{\figWidth}
		\plotFig{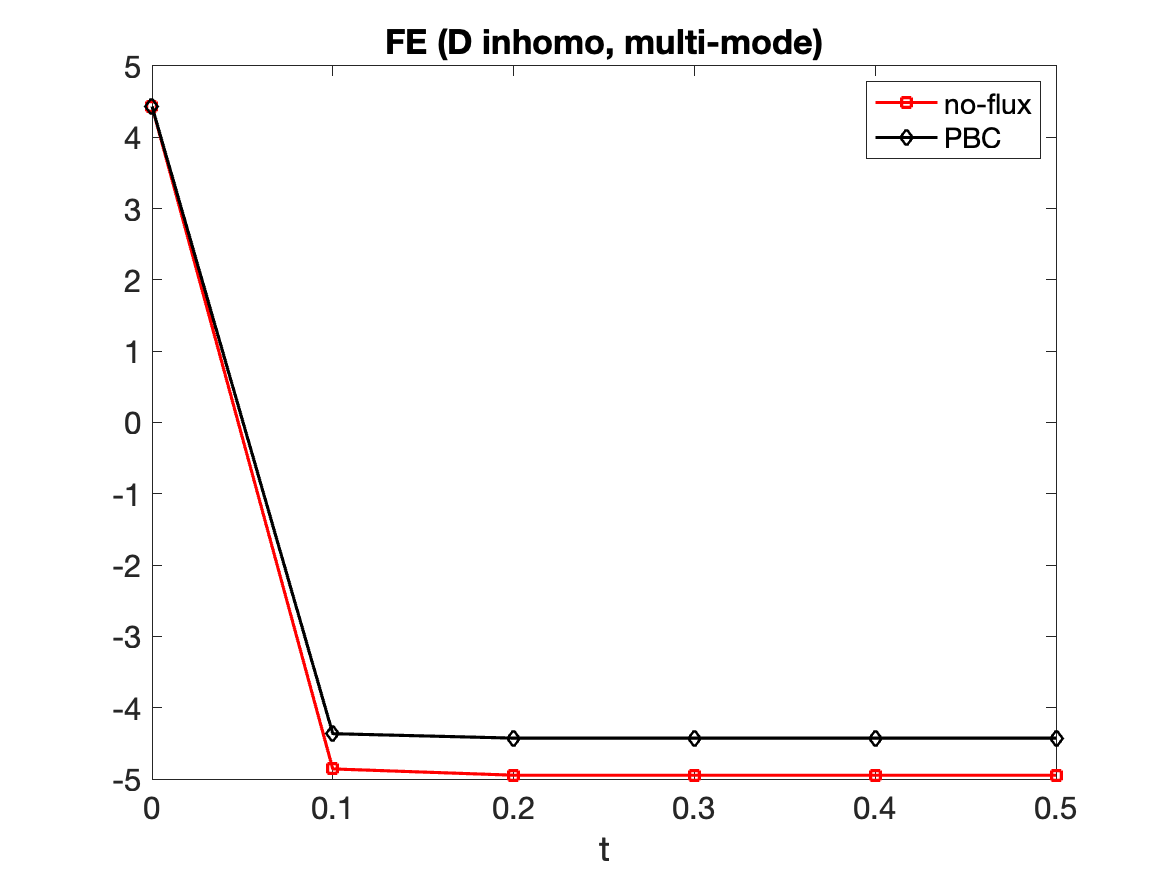}{\figWidth}
		\plotFig{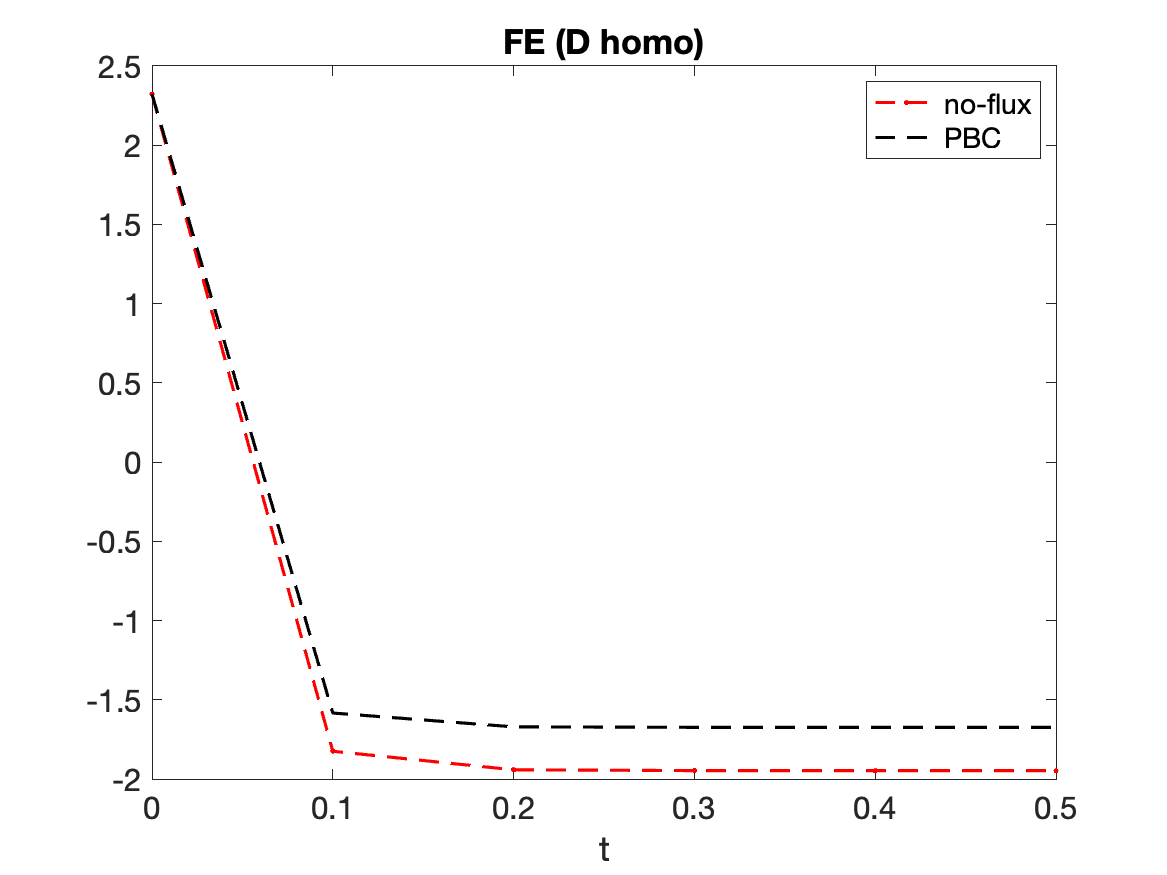}{\figWidth}
		\plotFig{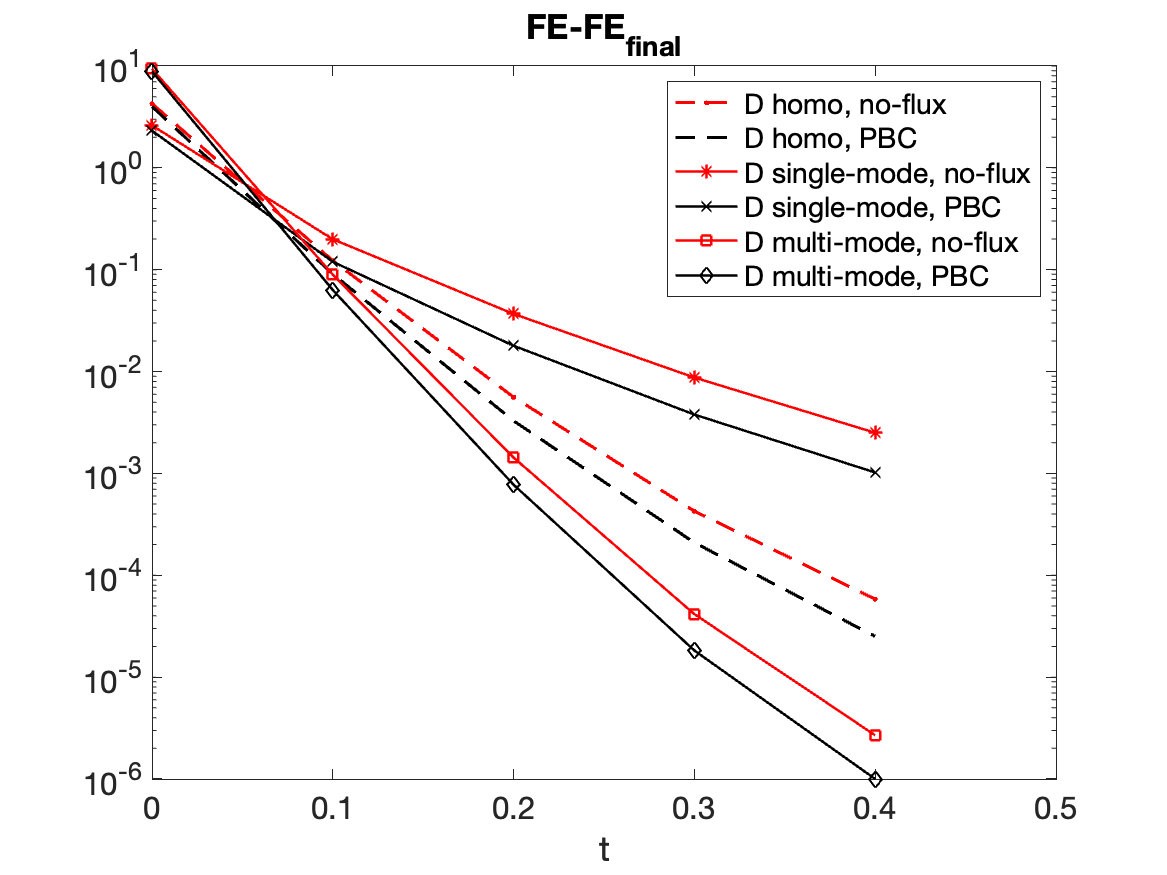}{\figWidth}
	\end{center}
	\caption{Exponential decay of free energy (FE), in three dimensions, 
		comparing no-flux (red) against periodic (black) \BC.
		Top: inhomogeneous $D(\MT{x})$; 
		left: single-mode \eqref{D_3d}; 
		right: multi-mode \eqref{D_3d_modes} with $(M_1,\ M_2,\ M_3) = (5,\ 3,\ 4),\ 
		A_{m_1 m_2 m_3} = 0.04$ if $m_1 \geq m_2 \geq m_3$,
		and 0 otherwise.
		Bottom left: homogeneous $D=1$. 
		Bottom right: direct comparison of the exponential decay rates. ($\Ntot \approx 10^3\times 5.$) }
	\label{fig: FE_3d_C}
\end{figure}

\begin{figure}[h!]
	\newcommand{\figWidth}{.45\linewidth}
	\begin{center}	
		\plotFig{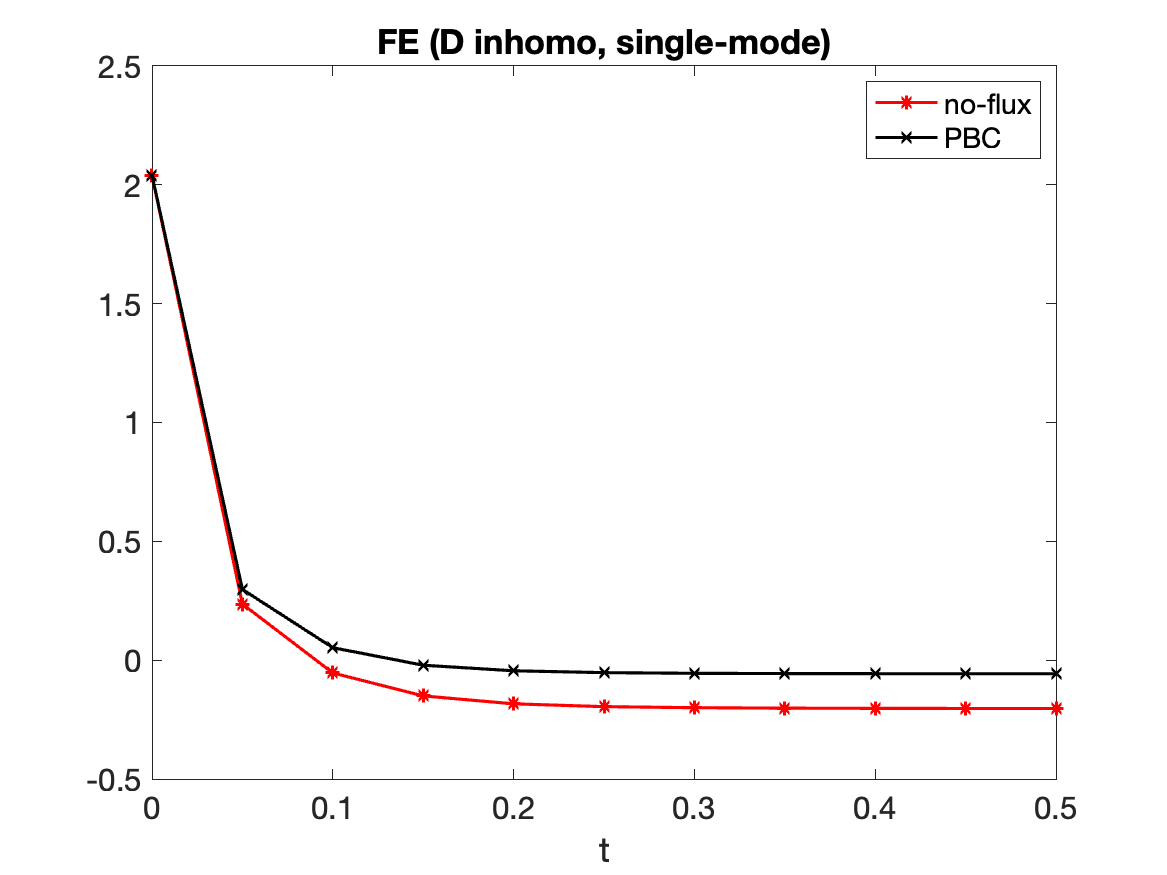}{\figWidth}
		\plotFig{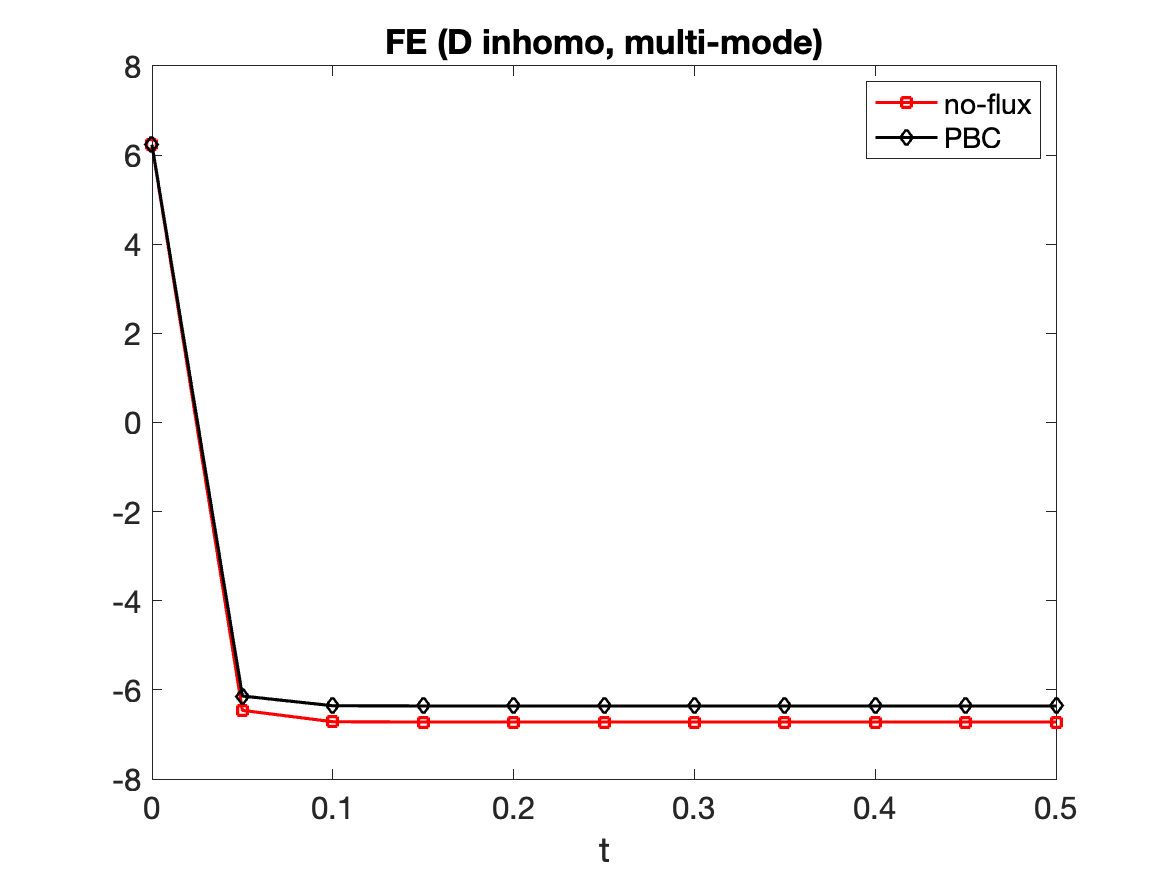}{\figWidth}
		\plotFig{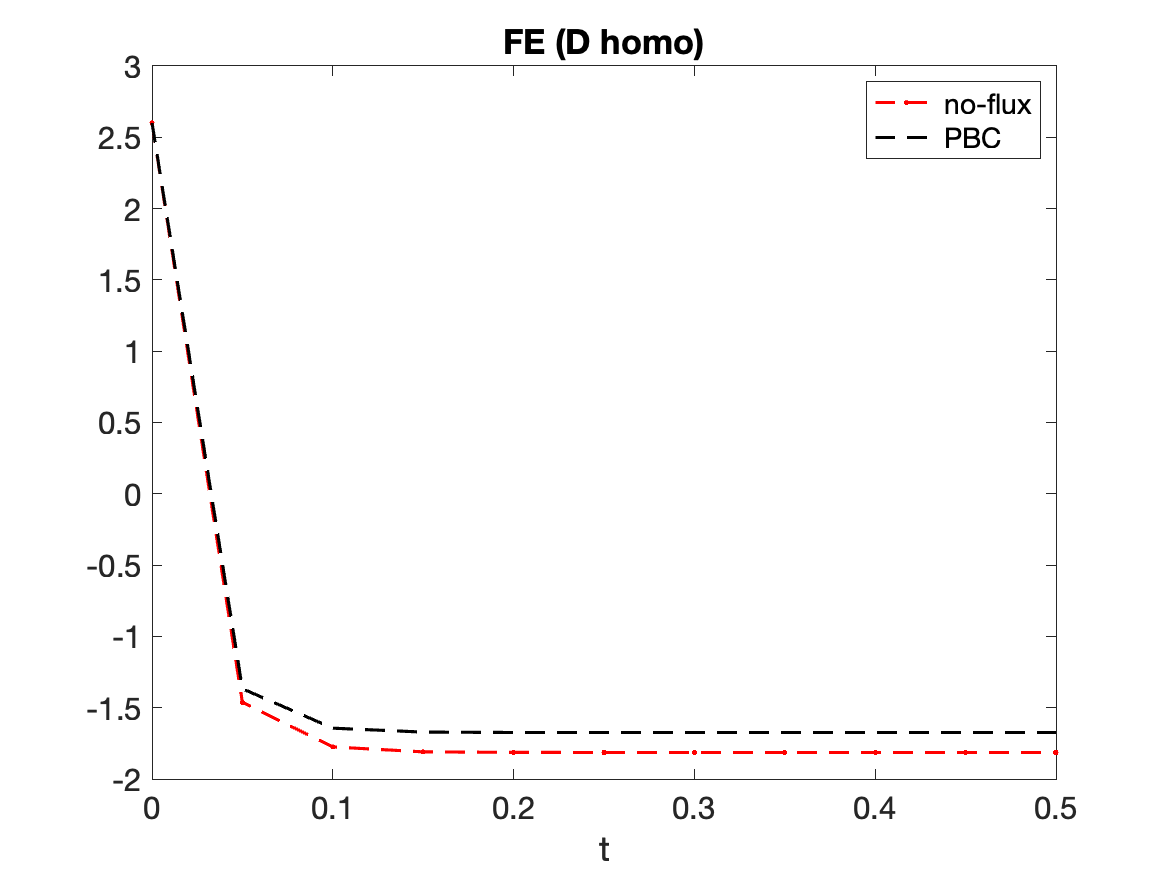}{\figWidth}
		\plotFig{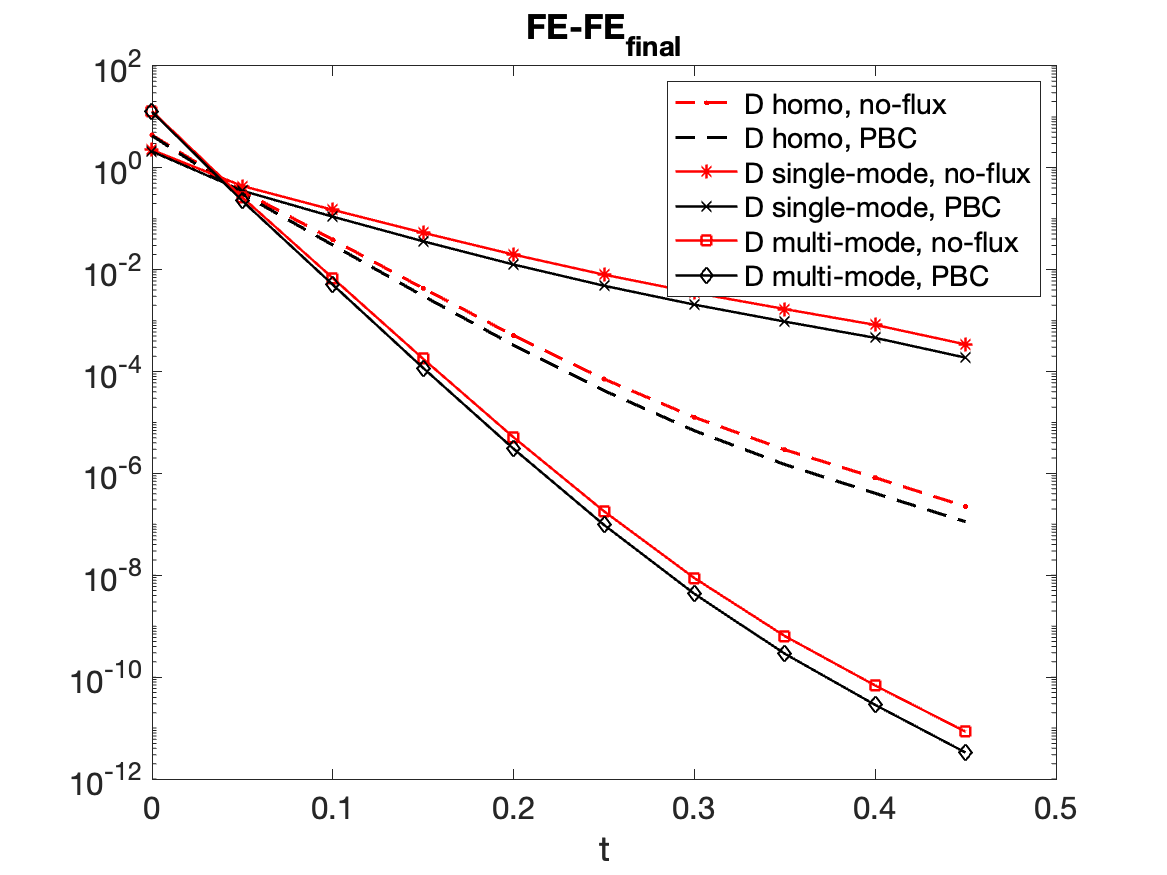}{\figWidth}
	\end{center}
	\caption{Exponential decay of free energy (FE), in three dimensions, 
		comparing no-flux (red) against periodic (black) \BC.
		Top: inhomogeneous $D(\MT{x})$; 
		left: single-mode \eqref{D_3d}; 
		right: multi-mode \eqref{D_3d_modes} with $(M_1,\ M_2,\ M_3) = (10,\ 8,\ 4),\ 
		A_{m_1 m_2 m_3} = 0.01$ if $m_1 \geq m_2 \geq m_3$,
		and 0 otherwise.
		Bottom left: homogeneous $D=1$. 
		Bottom right: direct comparison of the exponential decay rates. ($\Ntot \approx 20^3\times 10.$) }
	\label{fig: FE_3d}
\end{figure}

\appendix

\section{The Cauchy problem for linear homogeneous Fokker-Planck equation}
\label{sec:5}
In this appendix, as noticed in Remark \ref{rem:2.1}, we will
reformulate the entropy dissipation method \cite{MR3497125} for the Cauchy problem of the
linear homogeneous Fokker-Planck equation in the framework of the
general diffusion, in particular, the velocity field $\vec{u}$. 

Here,  we consider the following problem,
\begin{equation}
 \label{eq:A.1}
 \left\{
  \begin{aligned}
   \frac{\partial f}{\partial t}
   +
   \Div
   \left(
   f\vec{u}
   \right)
   &=
   0,
   &\quad
   &x\in\R^n,\quad
   t>0, \\
   \vec{u}
   &=
   -
   \nabla
   \left(
   D\log f
   +
   \phi(x)
   \right),
   &\quad
   &x\in\R^n,\quad
   t>0, \\
   f(x,0)&=f_0(x),&\quad
   &x\in\R^n,
  \end{aligned}
 \right.
\end{equation}
where $D>0$ is a positive constant and $\phi=\phi(x)$ is a smooth
function on $\R^n$. The free energy $F$ and the energy law
\eqref{eq:A.1} take the form,
\begin{equation}
 \label{eq:A.4}
  F[f]
  :=
  \int_{\R^n}
  \left(
   Df(\log f -1)+f\phi(x)
  \right)
  \,dx,
\end{equation}
and
\begin{equation}
 \label{eq:A.5}
  \frac{dF}{dt}[f](t)
  =
  -
  \int_{\R^n}
  |\vec{u}|^2f\,dx := - D_{\mathrm{dis}}[f](t).
\end{equation}
Following the same argument from Section \ref{sec:2},  as stated in Proposition
\ref{prop:2.8} and Theorem \ref{thm:2.9}, we can obtain the following
assertion.

\begin{proposition}
 \label{prop:A.2}
 Let $f$ be a solution of \eqref{eq:A.1} and let $\vec{u}$ be defined in
 \eqref{eq:A.1}. Then,
 \begin{equation}
  \label{eq:A.7}
   \frac{d^2F}{dt^2}[f](t)
   =
   2\int_{\R^n} ((\nabla^2\phi(x)) \vec{u}\cdot\vec{u}) f\,dx
   +
   2\int_{\R^n} D|\nabla \vec{u}|^2 f\,dx.
 \end{equation}
\end{proposition}

\begin{proposition}
 \label{prop:A.1}
 Let $\phi=\phi(x)$ be a function on $\R^n$, and let $f_0=f_0(x)$ be a
 probability density function on $\R^n$, satisfying $F[f_0]<\infty$ and $D_{\mathrm{dis}}[f_0]<\infty$, where $F$ and $D_{\mathrm{dis}}[f_0]$ are defined by
 \eqref{eq:A.4} and
 \eqref{eq:A.5}. 
  Let $f$ be a solution of
 \eqref{eq:A.1}. Let $\vec{u}$ be defined as in \eqref{eq:A.1}. Assume further 
 that there is a positive constant $\lambda>0$, such that
 $\nabla^2\phi\geq \lambda I$, where $I$ is the identity matrix. Then,
 the following is true,
 \begin{equation}
  \label{eq:A.2}
   \int_{\R^n}
   |\vec{u}|^2f\,dx
   \leq
   e^{-2\lambda t}
   \int_{\R^n}
   |\nabla
   \left(
    D\log f_0
    +
    \phi(x)
   \right)|^2f_0\,dx.
 \end{equation}
 In particular, we have that,
 \begin{equation}
  \label{eq:A.3}
  \frac{dF}{dt}[f](t)
   =
   -
   \int_{\R^n}
   |\vec{u}|^2f\,dx
   \rightarrow 0,
   \qquad
   \text{as}
   \
   t\rightarrow\infty.
 \end{equation}
\end{proposition}

Again as we note in Remark \ref{rem:2.1}, from this
proposition \ref{prop:A.1}, we can obtain the exponential decay of
$\frac{dF}{dt}[f](t)=:-D_{\mathrm{dis}}[f](t)$, but not necessarily the
long-time asymptotic behavior of the free energy $F[f](t)$ or the
solution $f(t)$.  The next theorem \ref{thm:A.1} gives a 
stronger convergence result, namely, exponential convergence of $f$ to
$f^{\mathrm{eq}}$ in the $L^1$ space as $t\rightarrow\infty$.

\begin{theorem}
 \label{thm:A.1}
 Let $\phi=\phi(x)$ be a function on $\R^n$, and $f_0=f_0(x)$ be a
 probability density function on $\R^n$. Assume that $f$ is a solution of
 \eqref{eq:A.1}, $\vec{u}$ is defined as in
 \eqref{eq:A.1}, and  there is a positive constant $\lambda>0$,
 such that,  $\nabla^2\phi\geq \lambda I$, where $I$ is the identity
 matrix. Further,  assume that $F[f_0]<\infty$ and $D_{\mathrm{dis}}[f_0]<\infty$, where $F$ and $D_{\mathrm{dis}}[f_0]$ are defined by
 \eqref{eq:A.4} and
 \eqref{eq:A.5}. Then,  any smooth solution of \eqref{eq:A.1} converges
 exponentially fast to the equilibrium state, that is, there is a
 positive constant $\Cl{const:A.1}>0$ which depends only on $D$, $F[\Feq]$
 and $F[f_0]$ such that,
 \begin{equation}
  \label{eq:A.6}
   \|f-\Feq\|_{L^1(\R^n)}\leq
   \Cr{const:A.1}e^{-\lambda t}, \qquad
   t>0.
 \end{equation}
\end{theorem}

Key ideas to show Theorem \ref{thm:A.1} are two inequalities. One is the
Gross logarithmic Sobolev inequality \cite{MR0420249}. The logarithmic
Sobolev inequality and \eqref{eq:A.3} deduce the convergence of
$F[f](t)$ to $F[f^{\mathrm{eq}}]$ as $t\rightarrow\infty$. Using
$\nabla^2\phi\geq\lambda I$ and Proposition \ref{prop:A.2}, we can show
that the relative entropy convergences exponentially fast, $F[f](t)$ converges
exponentially to $F[f^{\mathrm{eq}}]$, as $t\rightarrow\infty$. The other
key inequality is the classical Csisz\'ar-Kullback-Pinsker inequality
\cite{MR3497125} which connects $L^1$ convergence of $f$
to $f^{\mathrm{eq}}$ and the relative entropy convergence. Thus, we
obtain the exponential convergence of $f$ in $L^1$ spaces.

%

Now,  we show that $F[f](t)$ converges to $F[\Feq]$ as
$t\rightarrow\infty$. The following Gross logarithmic Sobolev inequality,
\begin{equation}
 \label{eq:A.8}
 \int_{\R^n}
  g^2
  \log
  \left(
   \frac{g^2}{\int_{\R^n}g^2\,d\mu}
  \right)
  \,d\mu
  \leq
  2
  \int_{\R^n}
  |\nabla g|^2
  \,d\mu,
\end{equation}
helps to show that the relative entropy $F[f](t)-F[\Feq]\to 0$ as
$t\rightarrow\infty$, where $d\mu=\Feq\,dx$ and $g\in H^1(d\mu)$,  (see
\cite{MR0420249}).

\begin{lemma}
 \label{lem:A.1}
 Assume,  that there is a positive constant $\lambda>0$ such that
 $\nabla^2\phi\geq \lambda I$, where $I$ is the identity matirx. Let $f$
 be a solution of \eqref{eq:A.1}. Then,
 \begin{equation}
  \label{eq:A.9}
   F[f](t) \rightarrow F[\Feq],
   \qquad
   \text{as}
   \
   t\rightarrow\infty.
 \end{equation}
\end{lemma}

\begin{proof}
 From \eqref{eq:A.5}, $F[f](t)$ is monotone decreasing,  so we can give the
 estimate of $F[f](t)-F[\Feq]$.  By direct calculation of
 $F[f](t)-F[\Feq]$ together with $D\log \Feq+\phi=\Cr{const:1.1}$,
 \eqref{eq:1.5}, we obtain that,
 \begin{equation*}
  F[f](t)-F[\Feq]
   =
   \int_{\R^n}
   (Df\log f-Df+f\phi +D\Feq-\Cr{const:1.1}\Feq)
   \,dx.
 \end{equation*}
 Using $\phi=\Cr{const:1.1}-D\log\Feq$ and
 $\int_{\R^n} f\,dx=\int_{\R^n} \Feq\,dx=1$, we have that,
 \begin{equation}
  \label{eq:A.10}
   F[f](t)-F[\Feq]
   =
   \int_{\R^n}
   D(f\log f-f\log \Feq)
   \,dx.
 \end{equation}

 Recall that $\rho=f/f^{\mathrm{eq}}$ and,
 \begin{equation*}
   F[f](t)-F[\Feq]
   =
   D
   \int_{\R^n}
   \rho\log \rho \Feq
   \,dx. 
 \end{equation*}
 Since $s\log s$ is a convex function on $s>0$, we can apply the Jensen's
 inequality \cite{MR1817225} and \eqref{eq:1.10} to have,
 \begin{equation*}
  \int_{\R^n}
   \rho\log \rho \Feq
   \,dx
   \geq
   \left(
    \int_{\R^n}
    \rho \Feq
   \,dx
   \right)
   \log
   \left(
    \int_{\R^n}
    \rho \Feq
   \,dx
   \right)
   \geq
   0,
 \end{equation*}
 hence $f[f](t)-F[\Feq]\geq 0$.

 Put $g=\sqrt{\rho}$ to \eqref{eq:A.8}, where $\rho=f/\Feq$. Since
 $\int_{\R^n} \rho\,d\mu=\int_{\R^n} f\,dx=1$, we have,
 \begin{equation}
  \label{eq:A.18}
  \int_{\R^n}\rho\log\rho \Feq\,dx
   \leq
   2
   \int_{\R^n}|\nabla \sqrt{\rho}|^2\Feq\,dx.
 \end{equation}
 Using $\rho\log\rho f^{\mathrm{eq}}=f\log f-f\log f^{\mathrm{eq}}$,
 \eqref{eq:A.10}, and \eqref{eq:A.18}, we obtain that,
 \begin{equation}
  \label{eq:A.11}
   F[f](t)-F[\Feq]
   \leq
   2D
   \int_{\R^n}|\nabla \sqrt{\rho}|^2\Feq\,dx.
 \end{equation}
 By direct calculation of $|\vec{u}|^2f$ and
 $|\nabla\sqrt{\rho}|^2\Feq$, we obtain,
\begin{equation*}
 |\vec{u}|^2f
  =
  |\nabla(D\log\rho)|^2f
  =
  D^2\frac{|\nabla \rho|^2}{\rho^2}f,
\end{equation*}
 and
 \begin{equation*}
  |\nabla\sqrt{\rho}|^2\Feq
   =
   \left|
    \frac12
    \rho^{-\frac12}\nabla \rho
   \right|^2
   \Feq
   =
   \frac14
   \frac{|\nabla\rho|^2}{\rho}\Feq
   =
   \frac14
   \frac{|\nabla\rho|^2}{\rho^2}f
   =
   \frac1{4D^2}
   |\vec{u}|^2f.
 \end{equation*}
 Hence,  we have,
 \begin{equation}
  \label{eq:A.12}
   2D
   \int_{\R^n}|\nabla \sqrt{\rho}|^2\Feq\,dx
   =
   \frac{1}{2D}
   \int_{\R^n}|\vec{u}|^2f\,dx.
 \end{equation}
 Combining \eqref{eq:A.11}, \eqref{eq:A.12}, and \eqref{eq:A.3}, we have
 $F[f](t)\rightarrow F[\Feq]$ as $t\rightarrow\infty$.
\end{proof}

Note that, in the proof of Lemma \ref{lem:A.1}, we obtain due to
\eqref{eq:A.10} that,
\begin{equation}
 \label{eq:A.15}
  F[f](t)-F[\Feq]
  =
  \int_{\R^n}
  D(f\log f-f\log \Feq)
  \,dx.
\end{equation}

 We next derive exponential decay of the relative entropy
$F[f](t)-F[\Feq]$.

\begin{lemma}
 \label{lem:A.2}
 Assume,  that there is a positive constant $\lambda>0$ such that
 $\nabla^2\phi\geq \lambda I$, where $I$ is the identity matirx and
 $F[f_0]<0$. Let $f$ be a solution of \eqref{eq:A.1}. Then,  we obtain for
 $t>0$,
 \begin{equation}
  \label{eq:A.14}
   (F[f](t)-F[\Feq])
   \leq
   e^{-2\lambda t}
   (F[f_0]-F[\Feq]).
 \end{equation}

\end{lemma}

\begin{proof}
 First, from \eqref{eq:A.5} and \eqref{eq:A.7}, due to the convexity assumption
 $\nabla^2\phi(x)\geq\lambda$, we have that,
 \begin{equation*}
  \frac{d^2F}{dt^2}[f](t)
   \geq
   2\lambda\int_{\R^n} |\vec{u}|^2 f\,dx
   =
   -
   2\lambda
   \frac{dF}{dt}[f](t).
 \end{equation*}
 Integrating on $[t,s]$, we obtain, 
 \begin{equation*}
  \frac{dF}{dt}[f](s)
   -
   \frac{dF}{dt}[f](t)
   \geq
   2\lambda
   (
   F[f](t)
   -
   F[f](s)
   ).
 \end{equation*}
 Taking $s\rightarrow\infty$ together with \eqref{eq:A.3},
 \eqref{eq:A.9}, and $\frac{d}{dt}F[\Feq]=0$, we arrive at,
 \begin{equation}
  \label{eq:A.13}
  \frac{d}{dt}(F[f](t)-F[\Feq])
   \leq
   -2\lambda
   (F[f](t)-F[\Feq]).
 \end{equation}
 Using the Gronwall's inequality in \eqref{eq:A.13} and $F[f_0]<\infty$,
 we obtain the result \eqref{eq:A.14}.
\end{proof}

Next, we state the classical Csisz\'ar-Kullback-Pinsker inequality, in
order to combine $L^1$ norm and the relative entropy $F[f](t)-F[\Feq]$.

\begin{proposition}[{Classical Csisz\'ar-Kullback-Pinsker inequality \cite{MR3497125}}]
 Let $\Omega\subset\R^n$ be a domain, let $f,g\in L^1(\Omega)$
 satisfy $f\geq0$, $g>0$, and $\int_\Omega f\,dx=\int_\Omega
 g\,dx=1$. Then,
 \begin{equation}
  \label{eq:A.17}
   \|f-g\|_{L^1(\Omega)}^2
   \leq
   2\int_\Omega(f\log f-f\log g)\,dx.
 \end{equation}
\end{proposition}

Using the classical Csisz\'ar-Kullback-Pinsker inequality, we show
Theorem \ref{thm:A.1}.

\begin{proof}[Proof of Theorem \ref{thm:A.1}]
 Combining \eqref{eq:A.15} and \eqref{eq:A.14}, we have that,
 \begin{equation}
  \label{eq:A.16}
  \int_{\R^n}
  (f\log f-f\log \Feq)
  \,dx
  =
  \frac{1}{D}
  (F[f](t)-F[\Feq])
  \leq
  \frac{1}{D}
  (F[f_0]-F[\Feq])
  e^{-2\lambda t}.
 \end{equation}
 Combining \eqref{eq:A.16} and the classical Csisz\'ar-Kullback-Pinsker
 inequality \eqref{eq:A.17}, we have,
 \begin{equation*}
  \|f-\Feq\|_{L^1(\R^n)}^2
   \leq
   \frac{2}{D}
   (F[f_0]-F[\Feq])
   e^{-2\lambda t},
 \end{equation*}
 Hence,  we obtain \eqref{eq:A.6} by selecting
   $\Cr{const:A.1}=\sqrt{\frac{2}{D} (F[f_0]-F[\Feq])}$.
\end{proof}

\begin{remark}
The strict convexity for $\phi$ is essential for the
entropy dissipation method. On the other hand, if $\phi$ is not strictly
convex, we can proceed to study the long-time asymptotic behavior in the
weighted $L^2$ space $L^2(\Omega,e^{\phi/D}\,dx)$~\cite{MR1842428,
epshteyn2021stochastic}. Using the result of the asymptotics in the
weighted $L^2$ space, we may show Lemma \ref{lem:A.1} and Lemma
\ref{lem:A.2} without using the Gross logarithmic Sobolev
inequality. Furthermore, it is known that the logarithmic Sobolev
inequality can be deduced from the differential inequality
\eqref{eq:A.13} of the relative entropy~\cite{MR3497125}, Thus, there is
a close relationship among the long-time asymptotics in the weighted
$L^2$ space, the logarithmic Sobolev inequality, the differential
inequality \eqref{eq:A.13} and the exponential decay for the relative
entropy \eqref{eq:A.14}.
\end{remark}

\section*{Acknowledgments}
Yekaterina Epshteyn acknowledges partial support of NSF DMS-1905463 and of NSF DMS-2118172, Chun Liu acknowledges partial support of
NSF DMS-1950868 and NSF DMS-2118181, and Masashi Mizuno
acknowledges partial support of JSPS KAKENHI Grant No. JP18K13446, JP22K03376.

\bibliographystyle{plain}
\bibliography{references}

\end{document}